\newif\ifBASIC
\BASICfalse
\newif\ifWP
\WPfalse
\newif\ifFULL
\FULLfalse
\newif\ifLATIN
\LATINfalse

\WPtrue

\LATINtrue  

\newif\ifnotFULL	
\notFULLtrue
\ifFULL\notFULLfalse\fi

\newif\ifnotLATIN	
\notLATINtrue
\ifLATIN\notLATINfalse\fi

\ifBASIC
  \newcommand{\CTI}{Vovk:arXiv0712.1275}
  \newcommand{\CTII}{Vovk:arXiv0712.1483}
  \newcommand{\CTIII}{Vovk:arXiv0801.1309}
  \newcommand{\CTV}{Vovk:arXiv1005}
  \newcommand{\LevyZeroOne}{Shafer/etal:2012JTP}
  \newcommand{\Dawid}{Dawid/etal:2011-full}
\fi
\ifWP
  \newcommand{\CTI}{Vovk:2009rand}
  \newcommand{\CTII}{Vovk:2008ECP}
  \newcommand{\CTIII}{GTP26arXiv}
  \newcommand{\CTV}{Vovk:2011-LMJ}

  \newcommand{\LevyZeroOne}{Shafer/etal:2012JTP}
  \newcommand{\Dawid}{Dawid/etal:2011}
\fi

\ifnotLATIN
  \newcommand{\Takeuchi}{Takeuchi:2004}
\fi
\ifLATIN
  \newcommand{\Takeuchi}{Takeuchi:2004latin}
\fi

\ifBASIC
  \documentclass{article}
  \usepackage{amsmath,amsthm,amsfonts,amssymb,latexsym,euscript}
  \newcommand{\Extra}[1]{}
\fi

\ifWP
  \documentclass{article}
  \usepackage{amsmath,amsthm,amsfonts,amssymb,latexsym,euscript,graphicx,url}

\makeatletter

\newif\iftwodates
\twodatesfalse

\renewcommand\maketitle{\begin{titlepage}%
  \let\footnotesize\small
  \let\footnoterule\relax
  \let \footnote \thanks
  \null\vfil
  \vskip 30\p@
  \begin{center}%
    {\LARGE \bf \@title \par}%
    \vskip 3em%
    {\large
     \lineskip .75em%
     \begin{tabular}[t]{c}%
       \@author
     \end{tabular}\par}%
     \vskip 1.5em%
  \end{center}\par
  \vfill
  \begin{center}
    \raisebox{1.5cm}{\includegraphics[width=0.58\textwidth]%
      {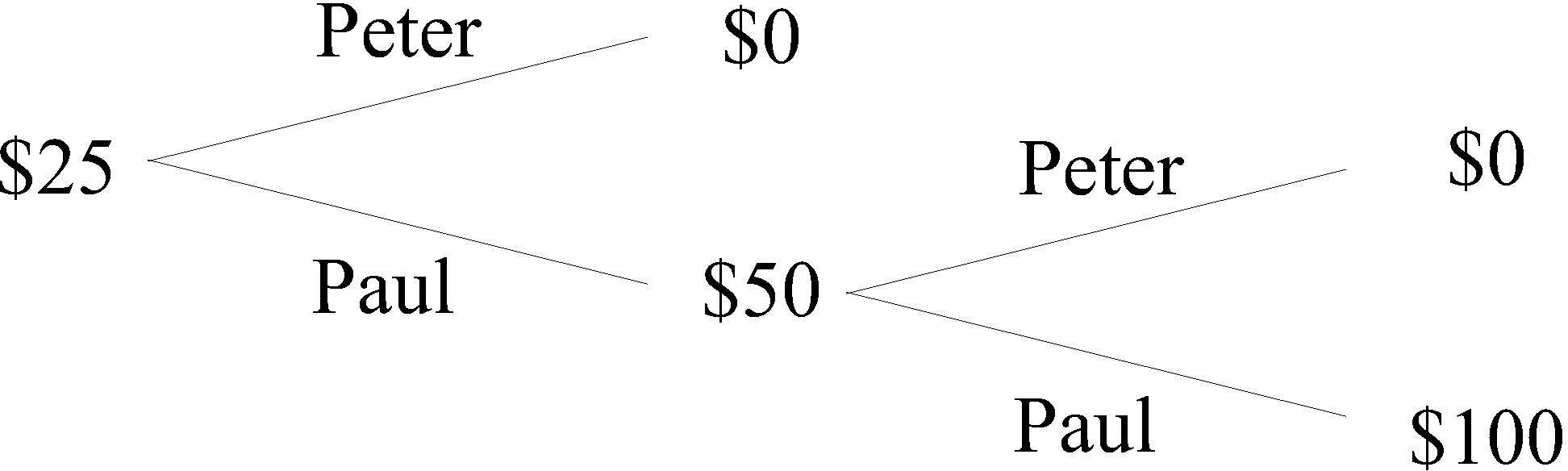}}%
    \hskip 3em%
    \includegraphics[width=0.29\textwidth]%
      {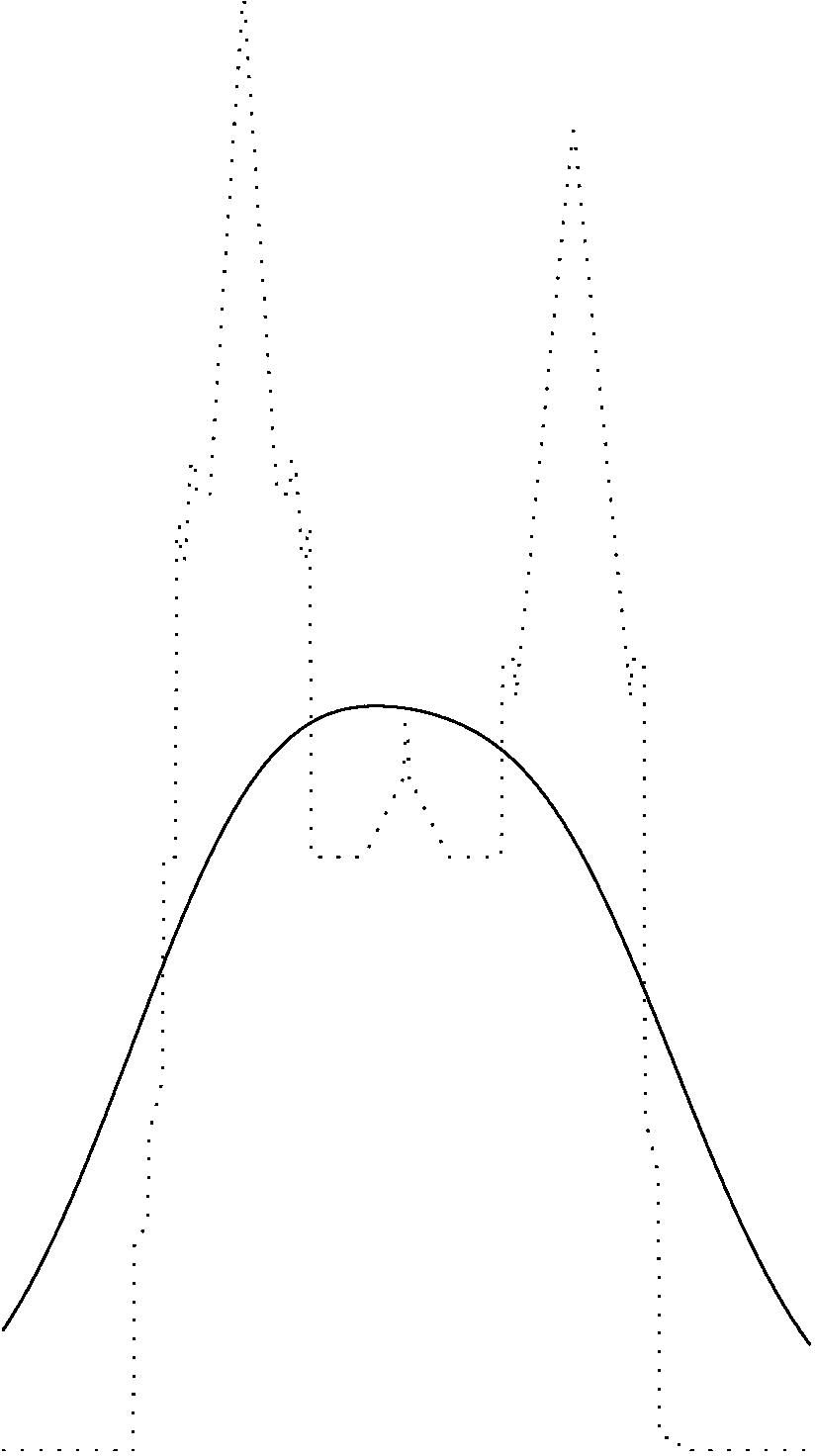}%
  \end{center}
  \@thanks
  \vfill
  \begin{center}
    {\large \bf The Game-Theoretic Probability and Finance Project}
  \end{center}
  \begin{center}
    {\large Working Paper \#\No}
  \end{center}
  \begin{center}
    {\iftwodates\large First posted \firstposted.
    Last revised \@date.\else\large\@date\fi}
  \end{center}
  \begin{center}
    Project web site:\\
    http://www.probabilityandfinance.com
  \end{center}
  \end{titlepage}%
  \setcounter{footnote}{0}%
  \global\let\thanks\relax
  \global\let\maketitle\relax
  \global\let\@thanks\@empty
  \global\let\@author\@empty
  \global\let\@date\@empty
  \global\let\@title\@empty
  \global\let\title\relax
  \global\let\author\relax
  \global\let\date\relax
  \global\let\and\relax
}

\renewenvironment{abstract}{%
  \titlepage
  \null\vfil
  \@beginparpenalty\@lowpenalty
  \begin{center}%
    \Large \bfseries \abstractname
    \@endparpenalty\@M
  \end{center}}%
  {\par\vfill\tableofcontents\endtitlepage}

\renewenvironment{thebibliography}[1]
  {\section*{\refname}%
  \addcontentsline{toc}{section}{\refname}
  \@mkboth{\MakeUppercase\refname}{\MakeUppercase\refname}%
  \list{\@biblabel{\@arabic\c@enumiv}}%
    {\settowidth\labelwidth{\@biblabel{#1}}%
    \leftmargin\labelwidth
    \advance\leftmargin\labelsep
    \@openbib@code
    \usecounter{enumiv}%
    \let\p@enumiv\@empty
    \renewcommand\theenumiv{\@arabic\c@enumiv}}%
    \sloppy
    \clubpenalty4000
    \@clubpenalty \clubpenalty
    \widowpenalty4000%
    \sfcode`\.\@m}
    {\def\@noitemerr
    {\@latex@warning{Empty `thebibliography' environment}}%
  \endlist}

\makeatother

  \newcommand{\Extra}[1]{}
\fi

\ifFULL
\usepackage{color}
\renewcommand{\Extra}[1]{\blue{#1}}

\newcommand{\blue}[1]{\textcolor{blue}{#1}}
\newcommand{\bluebegin}{\begingroup\color{blue}}
\newcommand{\blueend}{\endgroup}

\fi

\allowdisplaybreaks[4]

\ifnotLATIN
  \usepackage{CJK}
\fi

\newcommand{\st}{\mathrel{|}}		
\newcommand{\givn}{\mathrel{|}}
\renewcommand{\And}{\mathrel{\&}}	

\newcommand{\dd}{d}			

\newcommand{\CDOT}{\,\cdot\,}		

\newcommand{\K}{\EuScript{K}}		
\newcommand{\GGG}{\EuScript{G}}		

\newcommand{\FFF}{\EuScript{F}}		
\newcommand{\KKK}{\EuScript{I}}		
\newcommand{\Normal}{\EuScript{N}}	
\newcommand{\Wiener}{\EuScript{W}}	

\DeclareMathOperator{\III}{\boldsymbol{1}}		

\newcommand{\bbbp}{\mathbb{P}}		
\DeclareMathOperator{\Prob}{\bbbp}
\DeclareMathOperator{\UpProb}{\overline{\bbbp}}		
\DeclareMathOperator{\LowProb}{\underline{\bbbp}}	

\newcommand{\UpQ}{\smash{\overline{Q}}}		
\newcommand{\LowQ}{\smash{\underline{Q}}}	

\newcommand{\UpA}{\smash{\overline{A}}}		
\newcommand{\LowA}{\smash{\underline{A}}}	

\newcommand{\bbbe}{\mathbb{E}}		
\DeclareMathOperator{\Expect}{\bbbe}
\DeclareMathOperator{\UpExpect}{\overline{\bbbe}}   
\DeclareMathOperator{\LowExpect}{\underline{\bbbe}} 

\DeclareMathOperator{\DS}{DS}		
\DeclareMathOperator{\ntt}{ntt}		
\DeclareMathOperator{\m}{m}		

\DeclareMathOperator{\var}{v}		
\DeclareMathOperator{\vi}{vi}		
\DeclareMathOperator{\qvar}{w}		

\newcommand{\bbbr}{\mathbb{R}}		
\newcommand{\bbbq}{\mathbb{Q}}		
\newcommand{\bbbd}{\mathbb{D}}		
\newcommand{\bbbz}{\mathbb{Z}}		

  \theoremstyle{plain}
  \newtheorem{theorem}{Theorem}[section]
  \newtheorem{proposition}[theorem]{Proposition}
  \newtheorem{corollary}[theorem]{Corollary}
  \newtheorem{lemma}[theorem]{Lemma}

  \theoremstyle{definition}
  \newtheorem{remark}[theorem]{Remark}

  \makeatletter
  \@addtoreset{equation}{section}
  
  \makeatother

\newlength{\IndentI}
\newlength{\IndentII}
\newlength{\IndentIII}
\newlength{\WidthI}
\newlength{\WidthII}
\newlength{\WidthIII}
\ifBASIC
\title{Continuous-time trading and\\the emergence of probability}
\author{Vladimir Vovk\\
\texttt{vovk{\rm@}cs.rhul.ac.uk}\\
\texttt{http://vovk.net}}
\fi

\ifWP
  \title{Continuous-time trading and\\the emergence of probability}
  \author{Vladimir Vovk}
  \newcommand{\No}{28}
  \twodatestrue
  \newcommand{\firstposted}{April 28, 2009}
\fi

\begin{document}
\maketitle

\begin{abstract}
  This paper establishes a non-stochastic analogue
  of the celebrated result by Dubins and Schwarz
  about reduction of continuous martingales to Brownian motion via time change.
  We consider an idealized financial security with continuous price path,
  without making any stochastic assumptions.
  It is shown that typical price paths possess quadratic variation,
  where ``typical'' is understood in the following game-theoretic sense:
  there exists a trading strategy
  that earns infinite capital without risking more than one monetary unit
  if the process of quadratic variation does not exist.
  Replacing time by the quadratic variation process,
  we show that the price path becomes Brownian motion.
  This is essentially the same conclusion as in the Dubins--Schwarz result,
  except that the probabilities (constituting the Wiener measure)
  emerge instead of being postulated.
  We also give an elegant statement,
  inspired by Peter McCullagh's unpublished work,
  of this result in terms of game-theoretic probability theory.

  \bigskip

  \bigskip

  \noindent
  \emph{The journal version \cite{Vovk:2012FS-short} of this paper
    appeared in \emph{Finance and Stochastics} 
    in 2012.
    The final journal publication is available at Springer via}
    \begin{quote}
      \url{http://dx.doi.org/10.1007/s00780-012-0180-5}.
    \end{quote}
    \emph{As compared to the journal version, this technical report slightly strengthens the main result
    (Theorem~\ref{thm:supermain}) and includes a few further clarifications.}
  \newpage
\end{abstract}

\ifBASIC
  \tableofcontents
\fi

\section{Introduction}
\label{sec:introduction}

This paper is a contribution to the game-theoretic approach to probability.
This approach was explored (by, e.g., von Mises, Wald, and Ville)
as a possible basis for probability theory
at the same time as the now standard measure-theoretic approach (Kolmogorov),
but then became dormant.
The current revival of interest in it
started with A.~P.~Dawid's prequential principle
(\cite{Dawid:1984}, Section 5.1,
\cite{Dawid/Vovk:1999}, Section 3),
and recent work on game-theoretic probability
includes monographs \cite{Shafer/Vovk:2001,\Takeuchi}
and papers
\cite{Kumon/etal:2007,Horikoshi/Takemura:2008,Kumon/Takemura:2008,Kumon/etal:2008,%
Takeuchi/etal:2010,Kumon/etal:2011}.

Treatment of continuous-time processes in game-theoretic probability
often involves non-standard analysis
(see, e.g., \cite{Shafer/Vovk:2001}, Chapters 11--14).
The recent paper \cite{Takeuchi/etal:2009}
suggested avoiding non-standard analysis
and introduced
the key technique of ``high-frequency limit order strategies'',
also used in this paper and its predecessors, \cite{\CTI} and \cite{\CTII}.

An advantage of game-theoretic probability
is that one does not have to start with a full-fledged probability measure
from the outset
to arrive at interesting conclusions,
even in the case of continuous time.
For example,
\cite{\CTI} shows that continuous price paths
satisfy many standard properties of Brownian motion
(such as the absence of isolated zeroes)
and \cite{\CTII} (developing \cite{GTP5} and \cite{Takeuchi/etal:2009})
shows that the variation index of a non-constant continuous price path is 2,
as in the case of Brownian motion.
The standard qualification ``with probability one''
is replaced with ``unless a specific trading strategy
increases the capital it risks manyfold''
(the formal definitions, assuming zero interest rate,
will be given in Section \ref{sec:definitions}).
This paper makes the next step,
showing that the Wiener measure emerges in a natural way
in the continuous trading protocol.
Its main result
contains all main results of \cite{\CTI,\CTII},
together with several refinements,
as special cases.

Other results about the emergence of the Wiener measure
in game-theoretic probability
can be found in \cite{Vovk:1993forecasting} and \cite{\CTIII}.
However, the protocols of those papers are much more restrictive,
involving an externally given quadratic variation
(a game-theoretic analogue of predictable quadratic variation,
generally chosen by a player called Forecaster).
In this paper the Wiener measure emerges
in a situation with surprisingly little \emph{a priori} structure,
involving only two players:
the market and a trader.

The reader will notice that not only our main result
but also many of our definitions resemble those
in Dubins and Schwarz's paper \cite{Dubins/Schwarz:1965},
which can be regarded as the measure-theoretic counterpart of this paper.
The main difference of this paper is that we do not assume
a given probability measure from the outset.
A less important difference is that our main result will not assume
that the price path is unbounded and nowhere constant
(among other things,
this generalization is important to include the main results of \cite{\CTI,\CTII}
as special cases).
A result similar to that of Dubins and Schwarz
was almost simultaneously proved by Dambis \cite{Dambis:1965};
however,
Dambis, unlike Dubins and Schwarz, dealt with predictable quadratic variation,
and his result can be regarded as the measure-theoretic counterpart
of \cite{Vovk:1993forecasting} and \cite{\CTIII}.

Another related result is the well-known observation
(see, e.g., \cite{Follmer/Schied:2011}, Theorem 5.39)
that in the binomial model of a financial market
every contingent claim can be replicated by a self-financing portfolio
whose initial price is the expected value
(suitably discounted if the interest rate is not zero)
of the payoff function with respect to the risk-neutral probability measure.
This insight is, essentially, extended in this paper
to the case of an incomplete market
(the price for completeness in the binomial model
is the artificial assumption that at each step the price can only go up or down
by specified factors)
and continuous time
(continuous-time mathematical finance
usually starts from an underlying probability measure,
with some notable exceptions discussed in Section~\ref{sec:literature}).

This paper's definitions and results have many connections
with several other areas of finance and stochastics,
including stochastic integration,
the
Fundamental Theorems of Asset Pricing,
and model-free option pricing.
These will be discussed in Section~\ref{sec:literature}.

The main part of the paper starts
with the description of our continuous-time trading protocol
and the definition of game-theoretic versions of the notion of probability
(upper and lower price of a set)
in Section~\ref{sec:definitions}.
In Section~\ref{sec:result} we state our main result (Theorem~\ref{thm:main}),
which becomes especially intuitive
if we restrict our attention to the case of the initial price equal to $0$
and price paths that do not converge to a finite value and are nowhere constant:
the upper and lower price of any event
that is invariant with respect to time transformations
then exist and coincide between themselves and with its Wiener measure
(Corollary \ref{cor:main}).
This simple statement was made possible
by Peter McCullagh's unpublished work
on Fisher's fiducial probability:
McCullagh's idea was that fiducial probability
is only defined on the $\sigma$-algebra of events
invariant with respect to a certain group of transformations.
Section \ref{sec:applications} presents several applications
(connected with \cite{\CTI} and \cite{\CTII})
demonstrating the power of Theorem \ref{thm:main}.
The fact that typical price paths possess quadratic variation
is proved in Section \ref{sec:2-variation}.
It is, however, used earlier, in Section \ref{sec:result-constructive},
where it allows us to state a constructive version of Theorem \ref{thm:main}.
The constructive version, Theorem \ref{thm:main-constructive},
says that replacing time by the quadratic variation process
turns the price path into Brownian motion.
In Section~\ref{sec:generalizations} we state generalizations
from events to positive measurable functions
of Theorem~\ref{thm:main} and part of Theorem~\ref{thm:main-constructive};
these are Theorem~\ref{thm:supermain} and Theorem~\ref{thm:supermain-constructive},
respectively.
The easy directions in Theorem \ref{thm:supermain}
and Theorem~\ref{thm:supermain-constructive}
are proved in the same section.
Sections \ref{sec:coherence} and \ref{sec:tight}
prove part of Theorem~\ref{thm:main-constructive}
and prepare the ground for the proof
of the remaining parts of Theorems~\ref{thm:main-constructive}
and~\ref{thm:supermain-constructive}
(in Section \ref{sec:proof-constructive-b})
and Theorem \ref{thm:supermain}
(in Section \ref{sec:proof-main-le}).
Section~\ref{sec:literature} continues
the general discussion started in this section.

The words such as ``positive'', ``negative'', ``before'', ``after'',
``increasing'', and ``decreasing''
will be understood in the wide sense of $\ge$ or $\le$,
as appropriate;
when necessary, we will add the qualifier ``strictly''.
As usual,
$C(E)$ is the space of all continuous functions on a topological space $E$
equipped with the $\sup$ norm.
We often omit the parentheses around $E$ in expressions such as
$C[0,T]:=C([0,T])$.

\section{Upper price for sets}
\label{sec:definitions}

We consider a game between two players,
Reality (a financial market) and Sceptic (a trader),
over the time interval $[0,\infty)$.
First Sceptic chooses his trading strategy
and then Reality chooses a continuous function $\omega:[0,\infty)\to\bbbr$
(the price path of a security).

Let $\Omega$ be the set of all continuous functions $\omega:[0,\infty)\to\bbbr$.
For each $t\in[0,\infty)$,
$\FFF_t$ is defined to be the smallest $\sigma$-algebra
that makes all functions
$\omega\mapsto\omega(s)$, $s\in[0,t]$, measurable.
A \emph{process} $\mathfrak{S}$ is a family of functions
$\mathfrak{S}_t:\Omega\to[-\infty,\infty]$, $t\in[0,\infty)$,
each $\mathfrak{S}_t$ being $\FFF_t$-measurable;
its \emph{sample paths} are the functions $t\mapsto\mathfrak{S}_t(\omega)$.
An \emph{event} is an element of the $\sigma$-algebra
$\FFF_{\infty}:=\vee_t\FFF_t$,
also denoted by $\FFF$.
(We will often consider arbitrary subsets of $\Omega$ as well.)
Stopping times $\tau:\Omega\to[0,\infty]$
w.r.\ to the filtration $(\FFF_t)$
and the corresponding $\sigma$-algebras $\FFF_{\tau}$
are defined as usual;
$\omega(\tau(\omega))$ and $\mathfrak{S}_{\tau(\omega)}(\omega)$
will be simplified to $\omega(\tau)$ and $\mathfrak{S}_{\tau}(\omega)$,
respectively
(occasionally,
the argument $\omega$ will be omitted
in other cases as well).

The class of allowed strategies for Sceptic is defined in two steps.
A \emph{simple trading strategy} $G$
consists of an increasing sequence of stopping times
$\tau_1\le\tau_2\le\cdots$
and, for each $n=1,2,\ldots$, a bounded $\FFF_{\tau_{n}}$-measurable function $h_n$.
It is required that, for each $\omega\in\Omega$,
$\lim_{n\to\infty}\tau_n(\omega)=\infty$.
To such $G$ and an \emph{initial capital} $c\in\bbbr$
corresponds the \emph{simple capital process}
\begin{equation}\label{eq:simple-capital}
  \K^{G,c}_t(\omega)
  :=
  c
  +
  \sum_{n=1}^{\infty}
  h_n(\omega)
  \bigl(
    \omega(\tau_{n+1}\wedge t)-\omega(\tau_n\wedge t)
  \bigr),
  \quad
  t\in[0,\infty)
\end{equation}
(with the zero terms in the sum ignored,
which makes the sum finite for each $t$);
the value $h_n(\omega)$ will be called Sceptic's \emph{bet}
(or \emph{bet on $\omega$}, or \emph{stake}) at time $\tau_n$,
and $\K^{G,c}_t(\omega)$ will be referred to
as Sceptic's capital at time $t$.

A \emph{positive capital process} is any process $\mathfrak{S}$
that can be represented in the form
\begin{equation}\label{eq:positive-capital}
  \mathfrak{S}_t(\omega)
  :=
  \sum_{n=1}^{\infty}
  \K^{G_n,c_n}_t(\omega),
\end{equation}
where the simple capital processes $\K^{G_n,c_n}_t(\omega)$
are required to be positive, for all $t$ and $\omega$,
and the positive series $\sum_{n=1}^{\infty}c_n$ is required to converge.
The sum (\ref{eq:positive-capital}) is always positive
but allowed to take value $\infty$.
Since $\K^{G_n,c_n}_0(\omega)=c_n$ does not depend on $\omega$,
$\mathfrak{S}_0(\omega)$ also does not depend on $\omega$
and will sometimes be abbreviated to $\mathfrak{S}_0$.

\begin{remark}
  The financial interpretation
  of a positive capital process~(\ref{eq:positive-capital})
  is that it represents the total capital of a trader
  who splits his initial capital into a countable number of accounts
  and on each account runs a simple trading strategy
  making sure that this account never goes into debit.
\end{remark}

The \emph{upper price} of a set $E\subseteq\Omega$
(not necessarily $E\in\FFF$)
is defined as
\begin{equation}\label{eq:upper-probability}
  \UpProb(E)
  :=
  \inf
  \bigl\{
    \mathfrak{S}_0
    \bigm|
    \forall\omega\in\Omega:
    \liminf_{t\to\infty}
    \mathfrak{S}_t(\omega)
    \ge
    \III_E(\omega)
  \bigr\},
\end{equation}
where $\mathfrak{S}$ ranges over the positive capital processes
and $\III_E$ stands for the indicator function of $E$.
In the financial terminology
(and ignoring the fact that the $\inf$ in (\ref{eq:upper-probability})
may not be attained),
$\UpProb(E)$ is the price of the cheapest superhedge for the European contingent claim
paying $\III_E$ at time $\infty$.
It is easy to see that the $\liminf_{t\to\infty}$
in (\ref{eq:upper-probability})
can be replaced by $\sup_t$
(and, therefore, by $\limsup_{t\to\infty}$):
we can always stop (i.e., set all bets to $0$)
when $\mathfrak{S}$ reaches the level $1$
(or a level arbitrarily close to $1$).

We say that a set $E\subseteq\Omega$ is \emph{null} if $\UpProb(E)=0$.
If $E$ is null,
there is a positive capital process $\mathfrak{S}$ such that $\mathfrak{S}_0=1$
and $\lim_{t\to\infty}\mathfrak{S}_t(\omega)=\infty$ for all $\omega\in E$
(it suffices to sum over $\epsilon=1/2,1/4,\ldots$
positive capital processes $\mathfrak{S}^{\epsilon}$
satisfying $\mathfrak{S}_0^{\epsilon}=\epsilon$
and $\liminf_{t\to\infty}\mathfrak{S}_t^{\epsilon}\ge\III_E$).
A property of $\omega\in\Omega$ will be said
to hold \emph{for typical $\omega$}
if the set of $\omega$ where it fails is null.
Correspondingly,
a set $E\subseteq\Omega$ is \emph{full} if $\UpProb(E^c)=0$,
where $E^c:=\Omega\setminus E$ stands for the complement of $E$.

We can also define \emph{lower price}:
\begin{equation*}
  \LowProb(E)
  :=
  1-\UpProb(E^c)
\end{equation*}
(intuitively, this is the price of the most expensive subhedge of $\III_E$).
This notion of lower price will not be useful in this paper
(but its simple modification will be).

\begin{remark}
  Another natural setting is where $\Omega$ is defined
  as the set of all continuous functions $\omega:[0,T]\to\bbbr$
  for a given constant $T$ (the time horizon).
  In this case the definition of upper price simplifies:
  instead of $\liminf_{t\to\infty}\mathfrak{S}_t(\omega)$
  we will have simply $\mathfrak{S}_T(\omega)$
  in (\ref{eq:upper-probability}).
\end{remark}

\begin{remark}
  Many alternative names for upper and lower price have been used in literature
  (and even in literature on game-theoretic probability).
  The book \cite{Shafer/Vovk:2001} talks about upper and lower probability in the case of sets
  and upper and lower expectation in the case of functions
  (the latter case will be considered in Section~\ref{sec:generalizations}).
  The journal version \cite{Vovk:2012FS-short} of this paper
  essentially follows \cite{Hoffmann-Jorgensen:1987}
  and \cite{\LevyZeroOne} in using ``outer content'' for ``upper price''
  and ``inner content'' for ``lower price''.
  For terminology used in finance literature,
  see Section \ref{sec:literature}.
\end{remark}

  \subsection{Relation to the standard notion of a self-financing trading strategy}

  Readers accustomed to the standard definition
  of a self-financed trading strategy specifying explicitly the cash position
  (as in \cite{Shiryaev:1999}, Section~VII.1a)
  might find it helpful
  to have the connection between our notion of a simple trading strategy
  and the standard definition spelled out in detail.
  The main difference of the standard definition
  (apart from not being ``simple'', i.e., not trading at discrete times)
  is that it specifies not only the process of trading but also the initial capital.
  In the standard definition,
  we have $d+1$ assets (a bank account and $d$ securities)
  with prices $X^0_t,\ldots,X^d_t$ at time $t$
  (we are using the notation of \cite{Shiryaev:1999}).
  In this paper, $d=1$,
  it is assumed that $X^0_t=1$ for all $t$ (i.e., the interest rate is zero)
  and the notation for $X^1_t$ is $\omega(t)$;
  since $X^0_t$ does not carry any information,
  it is not mentioned explicitly.

  Suppose we are given an initial capital $c$ and a simple trading strategy $G$,
  as described above.
  The corresponding standard trading strategy is defined as a pair of predictable processes
  $(\pi^0_t,\pi^1_t)$;
  intuitively, $\pi^0_t$ (resp.\ $\pi^1_t$)
  is the number of units of $X^0_t$ (resp.\ $X^1_t$)
  in the trader's portfolio.
  We will now describe how the pair $(G,c)$ determines $(\pi^0_t,\pi^1_t)$;
  first we define $\pi^1_t$
  and then explain how $\pi^0_t$ is determined
  by the condition that the trading strategy is self-financing.
  The process $\pi^1_t$ is piecewise constant and is defined by
  \begin{equation*}
    \pi^1_t
    =
    \begin{cases}
      0 & \text{if $t\le\tau_1$}\\
      h_1 & \text{if $\tau_1<t\le\tau_2$}\\
      h_2 & \text{if $\tau_2<t\le\tau_3$}\\
      \ldots;
    \end{cases}
  \end{equation*}
  in particular, $\pi^1_0=0$.
  Being l\`adc\`ag (left-continuous with limits on the right),
  this process is predictable.
  The \emph{gain process} of the standard trading strategy $(\pi^0_t,\pi^1_t)$ is
  \begin{equation*}
    Y^{\pi}_t
    :=
    \int_0^t\pi^0_s\dd X^0_s
    +
    \int_0^t\pi^1_s\dd X^1_s
    =
    \int_0^t\pi^1_s\dd X^1_s
    =
    \K^{G,0}_t,
  \end{equation*}
  in the notation of (\ref{eq:simple-capital}),
  and its \emph{value process} is
  \begin{equation*}
    X^{\pi}_t
    :=
    \pi^0_t X^0_t
    +
    \pi^1_t X^1_t
    =
    \pi^0_t
    +
    \pi^1_t X^1_t.
  \end{equation*}
  Since the initial capital is $c$, we have to define $\pi^0_0:=c$.
  In order to be self-financing,
  the trading strategy $(\pi^0_t,\pi^1_t)$ must satisfy
  $
    X^{\pi}_t
    =
    X^{\pi}_0 + Y^{\pi}_t
  $,
  i.e.,
  \begin{equation*}
    \pi^0_t
    +
    \pi^1_t X^1_t
    =
    c + \K^{G,0}_t
    =
    \K^{G,c}_t.
  \end{equation*}
  Therefore, defining
  \begin{equation*}
    \pi^0_t
    :=
    \K^{G,c}_t
    -
    \pi^1_t X^1_t
  \end{equation*}
  (which agrees with $\pi^0_0:=c$)
  makes the strategy $(\pi^0_t,\pi^1_t)$ self-financing.

  It remains to check that the process $\pi^0_t$ is l\`adc\`ag:
  for each $t\in(0,\infty)$,
  \begin{multline*}
    \pi^0_{t} - \pi^0_{t-}
    =
    (\K^{G,c}_{t} - \K^{G,c}_{t-})
    -
    \pi^1_t
    (X^1_{t} - X^1_{t-})\\
    =
    h_n(\omega) (\omega(t)-\omega(t-))
    -
    \pi^1_t
    (X^1_{t} - X^1_{t-})
    =
    0,
  \end{multline*}
  where $n$ is defined from the condition $t\in(\tau_n,\tau_{n+1}]$.

\section{Main result: abstract version}
\label{sec:result}

A \emph{time transformation} is defined to be
a continuous increasing (not necessarily strictly increasing)
function $f:[0,\infty)\to[0,\infty)$
satisfying $f(0)=0$.
Equipped with the binary operation of composition,
$(f\circ g)(t):=f(g(t))$, $t\in[0,\infty)$,
the time transformations form a (non-commutative) monoid,
with the identity time transformation $t\mapsto t$ as the unit.
The \emph{action} of a time transformation $f$ on $\omega\in\Omega$
is defined to be the composition $\omega^f:=\omega\circ f\in\Omega$,
$(\omega\circ f)(t):=\omega(f(t))$.
The \emph{trail} of $\omega\in\Omega$ is the set
of all $\psi\in\Omega$ such that $\psi^f=\omega$
for some time transformation $f$.
(These notions are often defined for groups rather than monoids:
see, e.g., \cite{Neumann/etal:1994};
in this case the trail is called the orbit.
In their ``time-free'' considerations
Dubins and Schwarz \cite{Dubins/Schwarz:1965,Schwarz:1968,Schwarz:1972}
make simplifying assumptions
that make the monoid of time transformations a group;
we will make similar assumptions in Corollary \ref{cor:main}.)
A subset $E$ of $\Omega$ is \emph{time-superinvariant}
if together with any $\omega\in\Omega$ it contains
the whole trail of $\omega$;
in other words,
if for each $\omega\in\Omega$ and each time transformation $f$ it is true that
\begin{equation}\label{eq:invariant}
  \omega^f\in E
  \Longrightarrow
  \omega\in E.
\end{equation}
The \emph{time-superinvariant class} $\KKK$
is defined to be the family of those events
(elements of $\FFF$)
that are time-superinvariant.

Let $c\in\bbbr$.
The probability measure $\Wiener_c$ on $\Omega$
is defined by the conditions
that $\omega(0)=c$ with probability one
and, for all $0\le s<t$, $\omega(t)-\omega(s)$
is independent of $\FFF_s$ and has the Gaussian distribution
with mean $0$ and variance $t-s$.
(In other words,
$\Wiener_c$ is the distribution of Brownian motion
started at $c$.)
In this paper,
we rely on the classical arguments
for the existence of $\Wiener_c$
(see, e.g., \cite{Karatzas/Shreve:1991}, Chapter~2).
\begin{theorem}\label{thm:main}
  Let $c\in\bbbr$.
  Each event $E\in\KKK$ such that $\omega(0)=c$ for all $\omega\in E$
  satisfies
  \begin{equation}\label{eq:main}
    \UpProb(E)
    =
    \Wiener_c(E).
  \end{equation}
\end{theorem}

The main part of (\ref{eq:main}) is the inequality $\le$,
whose proof will occupy us
in Sections~\ref{sec:coherence}--\ref{sec:proof-main-le}.
The easy part $\ge$ will be established in Section~\ref{sec:generalizations}.

\begin{remark}\label{rem:monotonicity}
  Define a partial order $\le$ on $\Omega$ as follows:
  $\omega'\le\omega$ if and only if there is a time change $f$
  such that $\omega'=\omega\circ f$.
  (The intuition behind this definition is that some information in $\omega$ may be lost,
  even if the time scale is ignored: it is possible that $f(\infty)<\infty$.)
  Then $E$ is time-superinvariant if and only if $E$ is an upper set for this partial order.
\end{remark}

\begin{remark}\label{rem:intersection}
  The time-superinvariant class $\KKK$ is closed
  under countable unions and intersections;
  in particular, it is a monotone class.
  However, it is not closed under complementation,
  and so is not a $\sigma$-algebra
  (unlike McCullagh's invariant $\sigma$-algebras).
  An example of a time-superinvariant event $E$
  such that $E^c$ is not time-superinvariant
  is the set of all increasing (not necessarily strictly increasing)
  $\omega\in\Omega$
  satisfying $\lim_{t\to\infty}\omega(t)=\infty$:
  the implication (\ref{eq:invariant}) is violated
  when $\omega$ is the identity function
  (i.e., $\omega(t)=t$ for all $t$),
  $f=0$,
  and we have $E^c$ in place of $E$.
\end{remark}

\begin{remark}\label{rem:time-superinvariance}
  This remark explains the meaning of the formal notion of time-superinvariance.
  Let $f$ be a time transformation.
  Transforming $\omega$ into $\omega^f$
  is either trivial
  ($\omega$ is replaced by the constant $\omega(0)$, if $f=0$)
  or can be split into three steps:
  (a) remove $[T,\infty)$ from the domain of $\omega$,
  i.e., transform $\omega$ into $\omega':=\omega|_{[0,T)}$,
  for some $T\in(0,\infty]$
  (namely, $T:=\lim_{t\to\infty}f(t)$);
  (b) continuously deform the time interval $[0,T)$ into $[0,T')$
  for some $T'\in(0,\infty]$,
  i.e., transform $\omega'$ into $\omega''\in C[0,T')$
  defined by $\omega''(t):=\omega'(g(t))$
  for some increasing homeomorphism $g:[0,T')\to[0,T)$
  (e.g., the graph of $g$ can be obtained from the graph of $f$
  by removing all horizontal pieces);
  (c) insert countably many (perhaps a finite number of, perhaps zero) horizontal pieces
  into the graph of $\omega''$ making sure to obtain an element of $\Omega$
  (inserting a horizontal piece means replacing $\psi\in\Omega$ with
  \begin{equation*}
    \psi'(t)
    :=
    \begin{cases}
      \psi(t) & \text{if $t<a$}\\
      \psi(a) & \text{if $a\le t<b$}\\
      \psi(t+a-b) & \text{if $t\ge b$},
    \end{cases}
  \end{equation*}
  for some $a$ and $b$, $a<b$, in the domain of $\psi$,
  or
  \begin{equation*}
    \psi'(t)
    :=
    \begin{cases}
      \psi(t) & \text{if $t<c$}\\
      \lim_{s\to c}\psi(s) & \text{if $t\ge c$}
    \end{cases}
  \end{equation*}
  if the domain of $\psi$ is $[0,c)$ for some $c<\infty$
  and $\lim_{s\to c}\psi(s)$ exists in $\bbbr$).
  Therefore,
  the trail of $\omega\in\Omega$ consists of all elements of $\Omega$
  that can be obtained from $\omega$ by an application
  of the following steps:
  (a) remove any number of horizontal pieces from the graph of $\omega$;
  let $[0,T)$ be the domain of the resulting function $\omega'$
  (it is possible that $T<\infty$;
  if $T=0$, output any $\omega''\in\Omega$ satisfying $\omega''(0)=\omega(0)$);
  (b) assuming $T>0$,
  continuously deform the time interval $[0,T)$ into $[0,T')$
  for some $T'\in(0,\infty]$;
  let $\omega''$ be the resulting function with the domain $[0,T')$;
  (c) if $T'=\infty$, output $\omega''$;
  if $T'<\infty$ and $\lim_{t\to T'}\omega(t)$ exists in $\bbbr$,
  extend $\omega''$ to $[0,\infty)$ in any way
  making sure that the extension belongs to $\Omega$
  and output the extension;
  otherwise, nothing is output.
  A set $E$ is time-superinvariant
  if and only if application of these last three steps,
  (a)--(c),
  never leads outside $E$.
\end{remark}

\begin{remark}\label{rem:main}
  By the Dubins--Schwarz result \cite{Dubins/Schwarz:1965}
  and Lemma~\ref{lem:invariance} below,
  we can replace the $\Wiener_c$ in the statement of Theorem~\ref{thm:main}
  by any probability measure $P$ on $(\Omega,\FFF)$
  such that the process $X_t(\omega):=\omega(t)$
  is a martingale w.r.\ to $P$ and the filtration $(\FFF_t)$,
  is unbounded $P$-a.s., is nowhere constant $P$-a.s.,
  and satisfies $X_0=c$ $P$-a.s.
\end{remark}

Because of its generality,
some aspects of Theorem~\ref{thm:main} may appear counterintuitive.
(For example,
the conditions we impose on $E$
imply that $E$ contains all $\omega\in\Omega$ satisfying $\omega(0)=c$
whenever $E$ contains constant $c$.)
In the rest of this section
we will specialize Theorem~\ref{thm:main}
to the more intuitive case of divergent and nowhere constant price paths.

Formally,
we say that $\omega\in\Omega$ is \emph{nowhere constant}
if there is no interval $(t_1,t_2)$, where $0\le t_1<t_2$,
such that $\omega$ is constant on $(t_1,t_2)$,
we say that $\omega$ is \emph{divergent}
if there is no $c\in\bbbr$ such that $\lim_{t\to\infty}\omega(t)=c$,
and we let $\DS\subseteq\Omega$ stand for the set of all $\omega\in\Omega$
that are divergent and nowhere constant.
Intuitively, the condition that the price path $\omega$ should be nowhere constant
means that trading never stops completely,
and the condition that $\omega$ should be divergent
will be satisfied if $\omega$'s volatility does not eventually die away
(cf.\ Remark~\ref{rem:volatility} in Section~\ref{sec:result-constructive} below).
The conditions of being divergent and nowhere constant
in the definition of $\DS$ are similar to,
but weaker than,
Dubins and Schwarz's \cite{Dubins/Schwarz:1965} conditions
of being unbounded and nowhere constant.

All unbounded and strictly increasing time transformations $f:[0,\infty)\to[0,\infty)$
form a group, which will be denoted $\GGG$.
Let us say that an event $E$ is \emph{time-invariant}
if it contains the whole orbit $\{\omega^f\st f\in\GGG\}$
of each of its elements $\omega\in E$.
It is clear that $\DS$ is time-invariant.
Unlike $\KKK$,
the time-invariant events form a $\sigma$-algebra:
$E^c$ is time-invariant whenever $E$ is
(cf.\ Remark~\ref{rem:intersection}).

The following two lemmas will be needed to specialize Theorem~\ref{thm:main}
to subsets of $\DS$.
First of all,
it is not difficult to see that for subsets of $\DS$
there is no difference between time-invariance and time-superinvariance
(which makes the notion of time-superinvariance much more intuitive
for subsets of $\DS$).

\begin{lemma}\label{lem:invariance}
  An event $E\subseteq\DS$ is time-superinvariant
  if and only if it is time-invariant.
\end{lemma}
\begin{proof}
  If $E$ (not necessarily $E\subseteq\DS$) is time-superinvariant,
  $\omega\in E$, and $f\in\GGG$,
  we have $\psi:=\omega^f\in E$ as $\psi^{f^{-1}}=\omega$.
  Therefore, time-superinvariance always implies time-invariance.

  It is clear that, for all $\psi\in\Omega$ and time transformations $f$,
  $\psi^f\notin\DS$ unless $f\in\GGG$.
  Let $E\subseteq\DS$ be time-invariant, $\omega\in E$,
  $f$ be a time transformation, and $\psi^f=\omega$.
  Since $\psi^f\in\DS$,
  we have $f\in\GGG$, and so $\psi=\omega^{f^{-1}}\in E$.
  Therefore,
  time-invariance implies time-superinvariance
  for subsets of $\DS$.
\end{proof}

\begin{lemma}\label{lem:complement}
  An event $E\subseteq\DS$ is time-super\-invariant
  if and only if $\DS\setminus E$ is time-superinvariant.
\end{lemma}
\begin{proof}
  This follows immediately from Lemma~\ref{lem:invariance}.
\end{proof}

For time-invariant events in $\DS$,
(\ref{eq:main}) can be strengthened
to assert the coincidence of the upper and lower price of $E$
with $\Wiener_c(E)$.
However, the notions of upper and lower price
have to be modified slightly.

For any $B\subseteq\Omega$,
a restricted version of upper price can be defined by
\begin{equation*}
  \UpProb(E;B)
  :=
  \inf
  \bigl\{
    \mathfrak{S}_0
    \bigm|
    \forall\omega\in B:
    \liminf_{t\to\infty}
    \mathfrak{S}_t(\omega)
    \ge
    \III_E(\omega)
  \bigr\}
  =
  \UpProb(E\cap B),
\end{equation*}
with $\mathfrak{S}$ again ranging over the positive capital processes.
Intuitively,
this is the definition obtained when $\Omega$ is replaced by $B$:
we are told in advance that $\omega\in B$.
The corresponding restricted version of lower price is
\begin{equation*}
  \LowProb(E;B)
  :=
  1-\UpProb(E^c;B)
  =
  \LowProb(E\cup B^c).
\end{equation*}
We will use these definitions
only in the case where $\UpProb(B)=1$.
Lemma \ref{lem:lower-upper} below shows
that in this case $\LowProb(E;B)\le\UpProb(E;B)$.

We will say that $\UpProb(E;B)$ and $\LowProb(E;B)$
are \emph{restricted to $B$}.
It should be clear by now
that these notions are not related to conditional probability $\Prob(E\givn B)$.
Their analogues in measure-theoretic probability
are the function $E\mapsto\Prob(E\cap B)$,
in the case of upper price,
and the function $E\mapsto\Prob(E\cup B^c)$,
in the case of lower price
(assuming $B$ is measurable).
Both functions coincide with $\Prob$ when $\Prob(B)=1$.

We will also use the restricted versions of the notions
``null'', ``for typical'', and ``full''.
For example, $E$ being \emph{$B$-null} means $\UpProb(E;B)=0$.

Theorem \ref{thm:main} immediately implies the following statement
about the emergence of the Wiener measure in our trading protocol
(another such statement,
more general and constructive but also more complicated,
will be given in Theorem~\ref{thm:main-constructive}(b)).
\begin{corollary}\label{cor:main}
  Let $c\in\bbbr$.
  Each event $E\in\KKK$ satisfies
  \begin{equation}\label{eq:cor-main}
    \UpProb(E;\omega(0)=c,\DS)
    =
    \LowProb(E;\omega(0)=c,\DS)
    =
    \Wiener_c(E)
  \end{equation}
  (in this context,
  $\omega(0)=c$ stands for the event $\{\omega\in\Omega\st\omega(0)=c\}$
  and the comma stands for the intersection).
\end{corollary}
\begin{proof}
  Events $E\cap\DS\cap\{\omega\st\omega(0)=c\}$
  and $E^c\cap\DS\cap\{\omega\st\omega(0)=c\}$
  belong to $\KKK$:
  for the first of them,
  this immediately follows from $\DS\in\KKK$
  and $\KKK$ being closed under intersections
  (cf.\ Remark~\ref{rem:intersection}),
  and for the second,
  it suffices to notice that $E^c\cap\DS=\DS\setminus(E\cap\DS)\in\KKK$
  (cf.\ Lemma~\ref{lem:complement}).
  Applying (\ref{eq:main}) to these two events
  and making use of the inequality $\LowProb\le\UpProb$
  (cf.\ Lemma \ref{lem:lower-upper} and Equation (\ref{eq:DS-one}) below),
  we obtain:
  \begin{multline*}
    \Wiener_c(E)
    =
    1 - \Wiener_c(E^c\cap\DS\cap\{\omega\st\omega(0)=c\})
    =
    1 - \UpProb(E^c;\omega(0)=c,\DS)\\
    =
    \LowProb(E;\omega(0)=c,\DS)
    \le
    \UpProb(E;\omega(0)=c,\DS)\\
    =
    \Wiener_c(E\cap\DS\cap\{\omega\st\omega(0)=c\})
    =
    \Wiener_c(E).
    \qedhere
  \end{multline*}
\end{proof}

We can express the equality (\ref{eq:cor-main})
by saying that the game-theoretic probability of $E$ exists and is equal to $\Wiener_c(E)$
when we restrict our attention to $\omega$ in $\DS$ satisfying $\omega(0)=c$.

\section{Applications}
\label{sec:applications}

The main goal of this section is to demonstrate the power of Theorem~\ref{thm:main};
in particular,
we will see that it implies the main results of \cite{\CTI} and \cite{\CTII}.
One corollary (Corollary~\ref{cor:borderline}) of Theorem~\ref{thm:main}
solves an open problem posed in \cite{\CTII},
and two other corollaries (Corollaries~\ref{cor:Taylor-all} and~\ref{cor:Taylor})
give much more precise results.
At the end of the section we will draw the reader's attention
to several events such that:
Theorem~\ref{thm:main} together with very simple game-theoretic arguments
show that they are full;
the fact that they are full
does not follow from Theorem~\ref{thm:main} alone.

In this section
we deduce the main results of \cite{\CTI} and \cite{\CTII} and other results
as corollaries of Theorem~\ref{thm:main}
and the corresponding results for measure-theoretic Brownian motion.
It is, however, still important to have direct game-theoretic proofs
such as those given in \cite{\CTI,\CTII}.
This will be discussed in Remark~\ref{rem:direct}.

The following obvious fact will be used constantly in this paper:
restricted upper price
is countably (in particular, finitely) subadditive.
(Of course, this fact is obvious only because of our choice of definitions.)
\begin{lemma}\label{lem:subadditivity}
  For any $B\subseteq\Omega$
  and any sequence of subsets $E_1,E_2,\ldots$ of $\Omega$,
  \begin{equation*}
    \UpProb
    \left(
      \bigcup_{n=1}^{\infty}
      E_n;
      B
    \right)
    \le
    \sum_{n=1}^{\infty}
    \UpProb(E_n;B).
  \end{equation*}
  In particular,
  a countable union of $B$-null sets is $B$-null.
\end{lemma}

\subsection{Points of increase}

Let us say that $t\in[0,\infty)$ is a \emph{point of increase}
for $\omega\in\Omega$
if there exists $\delta>0$
such that $\omega(t_1)\le\omega(t)\le\omega(t_2)$
for all $t_1\in((t-\delta)^+,t]$ and $t_2\in[t,t+\delta)$.
Points of decrease are defined in the same way
except that $\omega(t_1)\le\omega(t)\le\omega(t_2)$
is replaced by $\omega(t_1)\ge\omega(t)\ge\omega(t_2)$.
We say that $\omega$ is \emph{locally constant to the right of $t\in[0,\infty)$}
if there exists $\delta>0$ such that $\omega$ is constant
over the interval $[t,t+\delta]$.

A slightly weaker form of the following corollary was proved directly
(by adapting Burdzy's \cite{Burdzy:1990} proof)
in \cite{\CTI}.
\begin{corollary}\label{cor:increase}
  Typical $\omega$ have no points $t$ of increase or decrease
  such that $\omega$ is not locally constant to the right of $t$.
\end{corollary}
\noindent
This result
(without the clause about local constancy)
was established by Dvoretzky, Erd\H{o}s, and Kakutani
\cite{Dvoretzky/etal:1961} for Brownian motion,
and Dubins and Schwarz \cite{Dubins/Schwarz:1965}
noticed that their reduction of continuous martingales
to Brownian motion
shows that it continues to hold
for all almost surely unbounded continuous martingales
that are almost surely nowhere constant.
We will apply Dubins and Schwarz's observation in the game-theoretic framework.
\begin{proof}[Proof of Corollary \ref{cor:increase}]
  Let us first consider only the $\omega\in\Omega$ satisfying $\omega(0)=0$.
  Consider the set $E$ of all $\omega\in\Omega$
  that have points $t$ of increase or decrease
  such that $\omega$ is not locally constant to the right of $t$
  and $\omega$ is not locally constant to the left of $t$
  (with the obvious definition of local constancy to the left of $t$;
  if $t=0$,
  every $\omega$ is locally constant to the left of $t$).
  Since $E$ is time-superinvariant
  (cf.\ Remark~\ref{rem:time-superinvariance}),
  Theorem \ref{thm:main} and the Dvoretzky--Erd\H{o}s--Kakutani result
  show that the event $E$ is null.
  And the following standard game-theoretic argument
  (as in \cite{\CTI}, Theorem~1)
  shows that the event
  that $\omega$ is locally constant to the left
  but not locally constant to the right
  of a point of increase or decrease
  is null.
  For concreteness,
  we will consider the case of a point of increase.
  It suffices to show that for all rational numbers $b>a>0$ and $D>0$
  the event that
  \begin{equation}\label{eq:a-b-D-event}
    \inf_{t\in[a,b]}
    \omega(t)
    =
    \omega(a)
    \le
    \omega(a) + D
    \le
    \sup_{t\in[a,b]}
    \omega(t)
  \end{equation}
  is null
  (see Lemma~\ref{lem:subadditivity}).
  The simple capital process that starts from $\epsilon>0$,
  bets $h_1:=1/D$ at time $\tau_1=a$,
  and bets $h_2:=0$ at time
  $\tau_2:=\min\{t\ge a\st\omega(t)\in\{\omega(a)-D\epsilon,\omega(a)+D\}\}$
  is positive and turns $\epsilon$ (an arbitrarily small amount) into 1
  when (\ref{eq:a-b-D-event}) happens.
  (Notice that this argument works both when $t=0$ and when $t>0$.)

  It remains to get rid of the restriction $\omega(0)=0$.
  Fix a positive capital process $\mathfrak{S}$
  satisfying $\mathfrak{S}_0<\epsilon$ and reaching $1$
  on $\omega$ with $\omega(0)=0$
  that have at least one point $t$ of increase or decrease
  such that $\omega$ is not locally constant to the right of $t$.
  Applying $\mathfrak{S}$ to $\omega-\omega(0)$
  gives another positive capital process,
  which will achieve the same goal but without the restriction $\omega(0)=0$.
\end{proof}

It is easy to see that the qualification
about local constancy to the right of $t$
in Corollary~\ref{cor:increase} is essential.
\begin{proposition}
  The upper price of the following event is one:
  there is a point $t$ of increase
  such that $\omega$ is locally constant to the right of $t$.
\end{proposition}
\begin{proof}
  This proof uses Lemma~\ref{lem:super-coherence}
  stated in Section~\ref{sec:coherence} below.
  Consider the continuous martingale which is Brownian motion
  that starts at $0$ and is stopped as soon as it reaches $1$.
\end{proof}

\subsection{Variation index}
\label{subsec:vi}

For each interval $[u,v]\subseteq[0,\infty)$ and each $p\in(0,\infty)$,
the strong \emph{$p$-variation} of $\omega\in\Omega$ over $[u,v]$
is defined as
\begin{equation}\label{eq:strong-p-variation}
  \var_p^{[u,v]}(\omega)
  :=
  \sup_\kappa
  \sum_{i=1}^{n_{\kappa}}
  \left|
    \omega(t_i)-\omega(t_{i-1})
  \right|^p,
\end{equation}
where $\kappa$ ranges over all partitions $u=t_0\le t_1\le\cdots\le t_{n_{\kappa}}=v$
of the interval $[u,v]$.
It is obvious that there exists a unique number
$\vi^{[u,v]}(\omega)\in[0,\infty]$,
called the \emph{variation index} of $\omega$ over $[u,v]$,
such that $\var_p^{[u,v]}(\omega)$ is finite when $p>\vi^{[u,v]}(\omega)$
and infinite when $p<\vi^{[u,v]}(\omega)$;
notice that $\vi^{[u,v]}(\omega)\notin(0,1)$.

The following result was obtained in \cite{\CTII}
(by adapting Bruneau's \cite{Bruneau:1979} proof);
in measure-theoretic probability
it was established by Lepingle
(\cite{Lepingle:1976}, Theorem~1 and Proposition~3)
for continuous semimartingales
and L\'evy \cite{Levy:1940} for Brownian motion.
\begin{corollary}\label{cor:vi}
  For typical $\omega\in\Omega$,
  the following is true.
  For any interval $[u,v]\subseteq[0,\infty)$ such that $u<v$,
  either $\vi^{[u,v]}(\omega)=2$ or $\omega$ is constant over $[u,v]$.
\end{corollary}
\noindent
(The interval $[u,v]$ was assumed fixed in \cite{\CTII},
but this assumption is easy to get rid of.)
\begin{proof}
  Without loss of generality
  we restrict our attention to the $\omega$ satisfying $\omega(0)=0$
  (see the proof of Corollary~\ref{cor:increase}).
  Consider the set of $\omega\in\Omega$ such that,
  for some interval $[u,v]\subseteq[0,\infty)$,
  neither $\vi^{[u,v]}(\omega)=2$ nor $\omega$ is constant over $[u,v]$.
  This set is time-superinvariant
  (cf.\ Remark~\ref{rem:time-superinvariance}),
  and so in conjunction with Theorem~\ref{thm:main}
  L\'evy's result implies that it is null.
\end{proof}

Corollary \ref{cor:vi} says that, for typical $\omega$,
\begin{equation*}
  \var_p^{[u,v]}(\omega)
  \begin{cases}
    {}<\infty & \text{if $p>2$}\\
    {}=\infty & \text{if $p<2$ and $\omega$ is not constant}.
  \end{cases}
\end{equation*}
However, it does not say anything about the situation for $p=2$.
The following result completes the picture
(solving the problem posed in \cite{\CTII}, Section 5).
\begin{corollary}\label{cor:borderline}
  For typical $\omega\in\Omega$,
  the following is true.
  For any interval $[u,v]\subseteq[0,\infty)$ such that $u<v$,
  either $\var^{[u,v]}_2(\omega)=\infty$ or $\omega$ is constant over $[u,v]$.
\end{corollary}
\begin{proof}
  L\'evy \cite{Levy:1940} proves for Brownian motion
  that $\var^{[u,v]}_2(\omega)=\infty$ almost surely
  (for fixed $[u,v]$, which implies the statement for all $[u,v]$).
  Consider the set of $\omega\in\Omega$ such that,
  for some interval $[u,v]\subseteq[0,\infty)$,
  neither $\var^{[u,v]}_2(\omega)=\infty$ nor $\omega$ is constant over $[u,v]$.
  This set is time-superinvariant,
  and so in conjunction with Theorem~\ref{thm:main}
  L\'evy's result implies that it is null.
\end{proof}

\subsection{More precise results}

Theorem~\ref{thm:main} allows us to deduce much stronger results
than Corollaries~\ref{cor:vi} and \ref{cor:borderline}
from known results about Brownian motion.

Define $\ln^*u:=1\vee\left|\ln u\right|$, $u>0$,
and let $\psi:[0,\infty)\to[0,\infty)$ be Taylor's \cite{Taylor:1972} function
\begin{equation*} 
  \psi(u)
  :=
  \frac{u^2}{2\ln^*\ln^*u}
\end{equation*}
(with $\psi(0):=0$).
For $\omega\in\Omega$, $T\in[0,\infty)$, and $\phi:[0,\infty)\to[0,\infty)$, set
\begin{equation*}
  \var_{\phi,T}(\omega)
  :=
  \sup_\kappa
  \sum_{i=1}^{n_{\kappa}}
  \phi
  \left(\left|
    \omega(t_{i})-\omega(t_{i-1})
  \right|\right),
\end{equation*}
where $\kappa$ ranges
over all partitions $0=t_0\le t_1\le\cdots\le t_{n_{\kappa}}=T$ of $[0,T]$.
In the previous subsection we considered the case $\phi(u):=u^p$;
another interesting case is $\phi:=\psi$.
See \cite{Bruneau:1975} for a much more explicit expression
for $\var_{\psi,T}(\omega)$.

\begin{corollary}\label{cor:Taylor-all}
  For typical $\omega$,
  \begin{equation*} 
    \forall T\in[0,\infty):
    \var_{\psi,T}(\omega)<\infty.
  \end{equation*}
  Suppose $\phi:[0,\infty)\to[0,\infty)$ is such that $\psi(u)=o(\phi(u))$ as $u\to0$.
  For typical $\omega$,
  \begin{equation*} 
    \forall T\in[0,\infty):
    \text{$\omega$ is constant on $[0,T]$ or $\var_{\phi,T}(\omega)=\infty$}.
  \end{equation*}
\end{corollary}

Corollary~\ref{cor:Taylor-all} refines Corollaries~\ref{cor:vi} and~\ref{cor:borderline};
it will be further strengthened by Corollary~\ref{cor:Taylor}.

The quantity $\var_{\psi,T}(\omega)$ is not
nearly as fundamental as the following quantity
introduced by Taylor \cite{Taylor:1972}:
for $\omega\in\Omega$ and $T\in[0,\infty)$, set
\begin{equation}\label{eq:qvar}
  \qvar_T(\omega)
  :=
  \lim_{\delta\to0}
  \sup_{\kappa\in K_{\delta}[0,T]}
  \sum_{i=1}^{n_{\kappa}}
  \psi
  \left(
    \left|
      \omega(t_i)
      -
      \omega(t_{i-1})
    \right|
  \right),
\end{equation}
where $K_{\delta}[0,T]$ is the set of all partitions $0=t_0\le\cdots\le t_{n_{\kappa}}=T$
of $[0,T]$ whose mesh is less than $\delta$:
$\max_i(t_i-t_{i-1})<\delta$.
Notice that the expression after the $\lim_{\delta\to0}$ in (\ref{eq:qvar}) is
increasing in $\delta$;
therefore, $\qvar_T(\omega)\le\var_{\psi,T}(\omega)$.

The following corollary contains Corollaries~\ref{cor:vi}--\ref{cor:Taylor-all}
as special cases.
It is similar to Corollary~\ref{cor:Taylor-all}
but is stated in terms of the process $\qvar$.
\begin{corollary}\label{cor:Taylor}
  For typical $\omega$,
  \begin{equation}\label{eq:Taylor}
    \forall T\in[0,\infty):
    \text{$\omega$ is constant on $[0,T]$ or $\qvar_T(\omega)\in(0,\infty)$}.
  \end{equation}
\end{corollary}

\begin{proof}
  First let us check that under the Wiener measure
  (\ref{eq:Taylor}) holds 
  for almost all $\omega$.
  It is sufficient to prove that $\qvar_T=T$ for all $T\in[0,\infty)$ a.s.
  Furthermore, it is sufficient to consider only rational $T\in[0,\infty)$.
  Therefore, it is sufficient to consider a fixed rational $T\in[0,\infty)$.
  And for a fixed $T$, $\qvar_T=T$ a.s.\ follows from Taylor's result
  (\cite{Taylor:1972}, Theorem~1).

  As usual, let us restrict our attention to the case $\omega(0)=0$.
  In view of Theorem~\ref{thm:main} it suffices to check
  that the complement of the event (\ref{eq:Taylor}) is time-superinvariant,
  i.e., to check (\ref{eq:invariant}),
  where $E$ is the complement of (\ref{eq:Taylor}).
  In other words, it suffices to check that $\omega^{f}=\omega\circ f$
  satisfies (\ref{eq:Taylor})
  whenever $\omega$ satisfies (\ref{eq:Taylor}).
  This follows from Lemma~\ref{lem:inv} below,
  which says that $\qvar_{T}(\omega\circ f)=\qvar_{f(T)}(\omega)$.
\end{proof}

\begin{lemma}\label{lem:inv}
  Let $T\in[0,\infty)$, $\omega\in\Omega$, and $f$ be a time transformation.
  Then $\qvar_{T}(\omega\circ f)=\qvar_{f(T)}(\omega)$.
\end{lemma}

\begin{proof}
  Fix $T\in[0,\infty)$, $\omega\in\Omega$, a time transformation $f$,
  and $c\in[0,\infty]$.
  Our goal is to prove
  \begin{multline}\label{eq:inv}
    \lim_{\delta\to0}
    \sup_{\kappa\in K_{\delta}[0,f(T)]}
    \sum_{i=1}^{n_{\kappa}}
    \psi
    \left(
      \left|
        \omega(t_i)
	-
        \omega(t_{i-1})
      \right|
    \right)
    =
    c\\
    \Longrightarrow
    \lim_{\delta\to0}
    \sup_{\kappa\in K_{\delta}[0,T]}
    \sum_{i=1}^{n_{\kappa}}
    \psi
    \left(
      \left|
        \omega(f(t_i))
	-
        \omega(f(t_{i-1}))
      \right|
    \right)
    =
    c,
  \end{multline}
  in the notation of (\ref{eq:qvar}).
  Suppose the antecedent in (\ref{eq:inv}) holds.
  Notice that the two $\lim_{\delta\to0}$ in (\ref{eq:inv})
  can be replaced by $\inf_{\delta>0}$.

  To prove that the limit on the right-hand side of (\ref{eq:inv})
  is $\le c$,
  take any $\epsilon>0$.
  We will assume $c<\infty$ (the case $c=\infty$ is trivial).
  Let $\delta>0$ be so small that
  \begin{equation*}
    \sup_{\kappa\in K_{\delta}[0,f(T)]}
    \sum_{i=1}^{n_{\kappa}}
    \psi
    \left(
      \left|
        \omega(t_i)
	-
        \omega(t_{i-1})
      \right|
    \right)
    <
    c+\epsilon.
  \end{equation*}
  Let $\delta'>0$ be so small that
  $
    \lvert t-t'\rvert < \delta'
    \Longrightarrow
    \lvert f(t)-f(t')\rvert < \delta
  $.
  Since $f(\kappa)\in K_{\delta}[0,f(T)]$ whenever $\kappa\in K_{\delta'}[0,T]$,
  \begin{equation*}
    \sup_{\kappa\in K_{\delta'}[0,T]}
    \sum_{i=1}^{n_{\kappa}}
    \psi
    \left(
      \left|
        \omega(f(t_i))
	-
        \omega(f(t_{i-1}))
      \right|
    \right)
    <
    c+\epsilon.
  \end{equation*}

  To prove that the limit on the right-hand side of (\ref{eq:inv})
  is $\ge c$,
  take any $\epsilon>0$ and $\delta'>0$.
  We will assume $c<\infty$ (the case $c=\infty$ can be considered analogously).
  Place a finite number $N$ of points including $0$ and $T$ onto the interval $[0,T]$
  so that the distance between any pair of adjacent points
  is less than $\delta'$;
  this set of points will be denoted $\kappa_0$.
  Let $\delta>0$ be so small that
  $\psi(\lvert\omega(t'')-\omega(t')\rvert)<\epsilon/N$
  whenever $\lvert t''-t'\rvert<\delta$.
  Choose a partition
  $\kappa=\{t_0,\ldots,t_{n}\}\in K_{\delta}[0,f(T)]$
  satisfying
  \begin{equation*}
    \sum_{i=1}^{n}
    \psi
    \left(
      \left|
        \omega(t_i)
	-
        \omega(t_{i-1})
      \right|
    \right)
    >
    c-\epsilon.
  \end{equation*}
  Let $\kappa'=\{t'_0,\ldots,t'_{n}\}$
  be a partition of the interval $[0,T]$
  satisfying $f(\kappa')=\kappa$.
  This partition will satisfy
  \begin{equation*}
    \sum_{i=1}^{n}
    \psi
    \left(
      \left|
        \omega(f(t'_i))
	-
        \omega(f(t'_{i-1}))
      \right|
    \right)
    >
    c-\epsilon,
  \end{equation*}
  and the union $\kappa''=\{t''_0,\ldots,t''_{N+n}\}$
  (with its elements listed in the increasing order)
  of $\kappa_0$ and $\kappa'$ will satisfy
  \begin{equation*}
    \sum_{i=1}^{N+n}
    \psi
    \left(
      \left|
        \omega(f(t''_i))
	-
        \omega(f(t''_{i-1}))
      \right|
    \right)
    >
    c-2\epsilon.
  \end{equation*}
  Since $\kappa''\in K_{\delta'}[0,T]$
  and $\epsilon$ and $\delta'$ can be taken arbitrarily small,
  this completes the proof.
\end{proof}

The value $\qvar_T(\omega)$ defined by (\ref{eq:qvar})
can be interpreted as the quadratic variation of the price path $\omega$
over the time interval $[0,T]$.
Another non-stochastic definition of quadratic variation
(see (\ref{eq:A-limit}))
will serve us in Section~\ref{sec:result-constructive}
as the basis for the proof of Theorem~\ref{thm:main}.
For the equivalence of the two definitions,
see Remark~\ref{rem:coincidence}.


\subsection{Limitations of Theorem~\ref{thm:main}}

We said earlier that Theorem~\ref{thm:main}
implies the main result of \cite{\CTI} (see Corollary~\ref{cor:increase}).
This is true in the sense that the extra game-theoretic argument
used in the proof of Corollary~\ref{cor:increase} was very simple.
But this simple argument was essential:
in this subsection we will see
that Theorem~\ref{thm:main} \emph{per se}
does not imply the full statement of Corollary~\ref{cor:increase}.

Let $c\in\bbbr$
and $E\subseteq\Omega$ be such that $\omega(0)=c$ for all $\omega\in E$.
Suppose the set $E$ is null.
We can say that the equality $\UpProb(E)=0$
can be deduced from Theorem~\ref{thm:main}
and the properties of Brownian motion
if (and only if) $\Wiener_c(\overline{E})=0$,
where $\overline{E}$ is the smallest time-superinvariant set
containing $E$
(it is clear that such a set exists and is unique).
It would be nice if all equalities $\UpProb(E)=0$,
for all null sets $E$ satisfying
$\forall\omega\in E:\omega(0)=c$,
could be deduced from Theorem~\ref{thm:main}
and the properties of Brownian motion.
We will see later (Proposition~\ref{prop:counterexample})
that this is not true even for some fundamental null events $E$;
an example of such an event will now be given.

Let us say that a closed interval $[t_1,t_2]\subseteq[0,\infty)$
is an \emph{interval of local maximum} for $\omega\in\Omega$
if (a) $\omega$ is constant on $[t_1,t_2]$
but not constant on any larger interval containing $[t_1,t_2]$,
and (b) there exists $\delta>0$ such that $\omega(s)\le\omega(t)$
for all $s\in((t_1-\delta)^+,t_1)\cup(t_2,t_2+\delta)$
and all $t\in[t_1,t_2]$.
In the case where $t_1=t_2$
we will say ``point'' instead of ``interval''.
It is shown in \cite{\CTI} (Corollary~3)
that, for typical $\omega$, all intervals of local maximum are points;
this also follows from Corollary~\ref{cor:increase},
and is very easy to check directly
(using the same argument as in the proof of Corollary~\ref{cor:increase}).
Let $E$ be the null event that $\omega(0)=c$
and not all intervals of local maximum of $\omega$ are points.
Proposition~\ref{prop:counterexample} says
that $\UpProb(E)=0$ cannot be deduced from Theorem~\ref{thm:main}
and the properties of Brownian motion.
This implies that Corollary~\ref{cor:increase}
also cannot be deduced from Theorem~\ref{thm:main}
and the properties of Brownian motion,
despite the fact that the deduction is possible
with the help of a very easy game-theoretic argument.

Before stating and proving Proposition~\ref{prop:counterexample},
we will introduce formally
the operator $E\mapsto\overline{E}$
and show that it is a bona fide closure operator.
For each $E\subseteq\Omega$,
$\overline{E}$ is defined to be the union of the trails of all points in $E$.
It can be checked that $E\mapsto\overline{E}$
satisfies the standard properties of closure operators:
$\overline{\emptyset}=\emptyset$ and
$\overline{E_1\cup E_2}=\overline{E_1}\cup\overline{E_2}$
are obvious,
and $\overline{\overline{E}}=\overline{E}$ and $E\subseteq\overline{E}$
follow from the fact that the time transformations constitute a monoid.
Therefore (\cite{Engelking:1989}, Theorem 1.1.3 and Proposition 1.2.7),
$E\mapsto\overline{E}$
is the operator of closure in some topology on $\Omega$,
which may be called the \emph{time-superinvariant topology}.
A set $E\subseteq\Omega$ is closed in this topology
if and only if it contains the trail of any of its elements.

\begin{proposition}\label{prop:counterexample}
  Let $c\in\bbbr$
  and $E$ be the set of all $\omega\in\Omega$ such that $\omega(0)=c$
  and $\omega$ has an interval of local maximum
  that is not a point.
  Then $E$ and $\overline{E}$ are events and
  \begin{equation}\label{eq:counterexample}
    0
    =
    \Wiener_c
    \left(
      E
    \right)
    =
    \UpProb(E)
    <
    \UpProb
    \left(
      \overline{E}
    \right)
    =
    \Wiener_c
    \left(
      \overline{E}
    \right)
    =
    1.
  \end{equation}
\end{proposition}
\begin{proof}
  For the equality $\UpProb(E)=0$, see above.
  The equality $\Wiener_c(E)=0$ is a well-known fact
  (and follows from $\UpProb(E)=0$ and Lemma~\ref{lem:main-ge} below).
  It suffices to prove that $\overline{E}\in\FFF$ and $\Wiener_c(\overline{E})=1$;
  Theorem~\ref{thm:main} will then imply $\UpProb(\overline{E})=1$.
  The inclusion $\overline{E}\in\FFF$ and equality $\Wiener_c(\overline{E})=1$
  follow from the following explicit description of $\overline{E}$:
  this set consists of all $\omega\in\Omega$ with $\omega(0)=c$
  that are not increasing functions.
  This can be seen from Remark~\ref{rem:time-superinvariance}
  or from the following argument.
  If $\omega$ is increasing,
  $\omega^f$ will also be increasing for any time transformation $f$.
  Combining this with (\ref{eq:invariant}),
  we can see that the set of all $\omega$ that are not increasing is time-superinvariant;
  since this set contains $E$,
  it also contains $\overline{E}$.
  In the opposite direction,
  we are required to show that any $\omega\in\Omega$ that is not increasing
  is an element of $\overline{E}$,
  i.e., there exists a time transformation $f$ such that $\omega^f\in E$.
  Fix such $\omega$
  and find $0\le a<b$ such that $\omega(a)>\omega(b)$.
  Let $m\in[0,b]$
  be the smallest element of $\arg\max_{t\in[0,b]}\omega(t)$.
  Applying the time transformation
  \begin{equation*}
    f(t)
    :=
    \begin{cases}
      t & \text{if $t<m$}\\
      m & \text{if $m\le t<m+1$}\\
      t-1 & \text{if $t\ge m+1$}
    \end{cases}
  \end{equation*}
  to $\omega$ we obtain an element of $E$.
\end{proof}

\begin{remark}
  Another event $E$ that satisfies
  (\ref{eq:counterexample})
  is the set of all $\omega\in\Omega$ such that $\omega(0)=c$
  and $\omega$ has an interval of local maximum
  that is not a point
  or has an interval of local minimum
  that is not a point
  (with the obvious definition of intervals and points of local minimum).
  Then $\overline{E}$ is the event
  that consists of all non-constant $\omega$ with $\omega(0)=c$.
  This is the largest possible $\overline{E}$ for $E$ satisfying $\UpProb(E)=0$
  (provided we consider only $\omega$ with $\omega(0)=c$):
  indeed, if the constant $c$ is in $\overline{E}$,
  $c$ will also be in $E$,
  and so $\UpProb(E)=1$.
\end{remark}

Proposition~\ref{prop:counterexample} shows that Theorem~\ref{thm:main}
does not make all other game-theoretic arguments redundant.
What is interesting is that already very simple arguments suffice
to deduce all results in \cite{\CTI,\CTII}.

\begin{remark}\label{rem:direct}
  Theorem~\ref{thm:main} does not make the game-theoretic arguments
  in \cite{\CTI,\CTII} redundant
  also in another, perhaps even more important, respect.
  For example,
  Corollary~\ref{cor:increase} is an existence result:
  it asserts the existence of a trading strategy
  whose capital process is positive and increases from $1$ to $\infty$
  when $\omega$ has a point $t$ of increase or decrease
  such that $\omega$ is not locally constant to the right of $t$.
  In principle, such a strategy could be extracted
  from the proof of Theorem~\ref{thm:main},
  but it would be extremely complicated and non-intuitive;
  the result would remain essentially an existence result.
  The proof of Theorem~2 in \cite{\CTI}, on the contrary,
  constructs an explicit trading strategy exploiting
  the existence of points of increase or decrease.
  Similarly, the proof of Theorem~1 in \cite{\CTII}
  constructs an explicit trading strategy
  whose existence is asserted in Corollary~\ref{cor:vi}.
  The recent paper \cite{\CTV} partially extends Corollary~\ref{cor:vi}
  to discontinuous price paths
  showing that $\vi^{[0,T]}(\omega)\le2$
  for all $T<\infty$ for typical $\omega$.
  The trading strategy constructed in \cite{\CTV}
  for profiting from $\vi^{[0,T]}(\omega)>2$ is especially intuitive:
  it just combines
  (following Stricker's \cite{Stricker:1979} idea)
  the strategies for profiting from
  $\liminf_t\omega(t)<a<b<\limsup_t\omega(t)$
  implicit in the standard proof of Doob's martingale convergence theorem.
\end{remark}

\begin{remark}
  All results discussed in this section
  are about sets of upper price zero or lower price one,
  and one might suspect that the class $\KKK$ is so small
  that $\Wiener_c(E)\in\{0,1\}$ for all $c\in\bbbr$
  and all $E\in\KKK$ such that $\omega(0)=c$ when $\omega\in E$;
  this would have been another limitation of Theorem~\ref{thm:main}.
  However, it is easy to check that for each $p\in[0,1]$ and each $c\in\bbbr$
  there exists $E\in\KKK$ satisfying $\omega(0)=c$ for all $\omega\in E$
  and satisfying $\Wiener_c(E)=p$.
  Indeed, without loss of generality we can take $c:=p$,
  and we can then define $E$ to be the event that $\omega(0)=p$,
  $\omega$ reaches levels 0 and 1,
  and $\omega$ reaches level 1 before reaching level 0.
\end{remark}

\section{Main result: constructive version}
\label{sec:result-constructive}

For each $n\in\{0,1,\ldots\}$,
let $\bbbd_n:=\{k2^{-n}\st k\in\bbbz\}$
and define a sequence of stopping times $T^n_k$, $k=-1,0,1,2,\ldots$, inductively
by $T^n_{-1}:=0$,
\begin{align*}
  T^n_0(\omega)
  &:=
  \inf
  \left\{
    t\ge 0
    \st
    \omega(t)\in\bbbd_n
  \right\},\\
  T^n_k(\omega)
  &:=
  \inf
  \left\{
    t\ge T^n_{k-1}(\omega)
    \st
    \omega(t)\in\bbbd_n
    \And
    \omega(t)\ne\omega(T^n_{k-1})
  \right\},
  \enspace
  k=1,2,\ldots
\end{align*}
(as usual, $\inf\emptyset:=\infty$).
For each $t\in[0,\infty)$ and $\omega\in\Omega$,
define
\begin{equation}\label{eq:A}
  A^n_t(\omega)
  :=
  \sum_{k=0}^{\infty}
  \left(
    \omega(T^n_k\wedge t)
    -
    \omega(T^n_{k-1}\wedge t)
  \right)^2,
  \quad
  n=0,1,2,\ldots,
\end{equation}
(cf.\ (\ref{eq:strong-p-variation}) with $p=2$)
and set
\begin{equation}\label{eq:A-limit}
  \UpA_t(\omega)
  :=
  \limsup_{n\to\infty} A^n_t(\omega),
  \quad
  \LowA_t(\omega)
  :=
  \liminf_{n\to\infty} A^n_t(\omega).
\end{equation}
We will see later
(Theorem \ref{thm:main-constructive}(a))
that the event $(\forall t\in[0,\infty):\UpA_t=\LowA_t)$ is full
and that for typical $\omega$ the functions
$\UpA(\omega):t\in[0,\infty)\mapsto\UpA_t(\omega)$
and
$\LowA(\omega):t\in[0,\infty)\mapsto\LowA_t(\omega)$
are elements of $\Omega$
(in particular, they are finite).
But in general we can only say that $\UpA(\omega)$ and $\LowA(\omega)$
are positive increasing functions
(not necessarily strictly increasing)
that can even take value $\infty$.
For each $s\in[0,\infty)$,
define the stopping time
\begin{equation}\label{eq:stopping}
  \tau_s
  :=
  \inf
  \left\{
    t\ge0
    \st
    \UpA|_{[0,t)}=\LowA|_{[0,t)}\in C[0,t)
    \And
    \sup_{u<t}\UpA_u=\sup_{u<t}\LowA_u\ge s
  \right\}.
\end{equation}
(We will see in Section \ref{sec:2-variation},
Lemma \ref{lem:stopping},
that this is indeed a stopping time.)
It will be convenient to use the following convention:
an event stated in terms of $A_{\infty}$,
such as $A_{\infty}=\infty$,
happens if and only if $\UpA=\LowA\in\Omega$
and $A_{\infty}:=\UpA_{\infty}=\LowA_{\infty}$ satisfies the given condition.

Let $P$ be a function
defined on the power set of $\Omega$
and taking values in $[0,1]$
(such as $\UpProb$ or $\LowProb$),
and let $f:\Omega\to\Psi$
be a mapping from $\Omega$ to another set $\Psi$.
The \emph{pushforward} $Pf^{-1}$ of $P$ by $f$
is the function on the power set of $\Psi$ defined by
\begin{equation*}
  Pf^{-1}(E)
  :=
  P(f^{-1}(E)),
  \quad
  E\subseteq\Psi.
\end{equation*}

An especially important mapping for this paper
is the \emph{normalizing time transformation}
$\ntt:\Omega\to\bbbr^{[0,\infty)}$
defined as follows:
for each $\omega\in\Omega$,
$\ntt(\omega)$ is the time-changed price path
$s\mapsto\omega(\tau_s)$, $s\in[0,\infty)$,
with $\omega(\infty)$ set to, e.g., $0$.
(We call it ``normalizing'' since our goal is to ensure
$\UpA_t(\ntt(\omega))=\LowA_t(\ntt(\omega))=t$ for all $t\ge0$
for typical $\omega$.)
For each $c\in\bbbr$, let
\begin{align}
  \UpQ_c
  &:=
  \UpProb(\CDOT;\omega(0)=c,A_{\infty}=\infty)\ntt^{-1}
  \label{eq:Qc-upper}\\
  \LowQ_c
  &:=
  \LowProb(\CDOT;\omega(0)=c,A_{\infty}=\infty)\ntt^{-1}
  \label{eq:Qc-lower}
\end{align}
(as before,
the commas stand for conjunction in this context)
be the pushforwards of the restricted upper and lower price
\begin{align*}
  E\subseteq\Omega&\mapsto\UpProb(E;\omega(0)=c,A_{\infty}=\infty)\\
  E\subseteq\Omega&\mapsto\LowProb(E;\omega(0)=c,A_{\infty}=\infty),
\end{align*}
respectively,
by the normalizing time transformation $\ntt$.

As mentioned earlier,
we use restricted upper and lower price $\UpProb(E;B)$ and $\LowProb(E;B)$
only when $\UpProb(B)=1$.
In Section~\ref{sec:coherence}, (\ref{eq:one}), we will see
that indeed $\UpProb(\omega(0)=c,A_{\infty}=\infty)=1$.

The next theorem shows that the pushforwards of $\UpProb$ and $\LowProb$
we have just defined
are closely connected with the Wiener measure.
Remember that, for each $c\in\bbbr$,
$\Wiener_c$ is the probability measure on $(\Omega,\FFF)$
which is the pushforward of the Wiener measure $\Wiener_0$
by the mapping $\omega\in\Omega\mapsto\omega+c$
(i.e., $\Wiener_c$ is the distribution of Brownian motion
over the time period $[0,\infty)$ started from $c$).

\begin{theorem}\label{thm:main-constructive}
  (a) For typical $\omega$,
  the function
  \begin{equation*}
    A(\omega):
    t\in[0,\infty)
    \mapsto
    A_t(\omega)
    :=
    \UpA_t(\omega)
    =
    \LowA_t(\omega)
  \end{equation*}
  exists,
  is an increasing element of $\Omega$ with $A_0(\omega)=0$,
  and has the same intervals of constancy as $\omega$.
  (b) For all $c\in\bbbr$,
  the restriction of both $\UpQ_c$ and $\LowQ_c$ to $\FFF$
  coincides with the measure $\Wiener_c$ on $\Omega$
  (in particular, $\LowQ_c(\Omega)=1$).
\end{theorem}

\begin{remark}\label{rem:volatility}
  The value $A_t(\omega)$ can be interpreted as the total volatility
  of the price path $\omega$ over the time period $[0,t]$.
  Theorem~\ref{thm:main-constructive}(b) implies that typical $\omega$
  satisfying $A_{\infty}(\omega)=\infty$ are unbounded
  (in particular, divergent).
  If $A_{\infty}(\omega)<\infty$,
  the total volatility $A_{t+1}(\omega)-A_t(\omega)$ of $\omega$ over $(t,t+1]$
  tends to $0$ as $t\to\infty$,
  and so the volatility of $\omega$ can be said to die away.
\end{remark}

\begin{remark}
  Theorem~\ref{thm:main-constructive} will continue to hold
  if the restriction ``${};\omega(0)=c,A_\infty=\infty)$''
  in the definitions (\ref{eq:Qc-upper}) and (\ref{eq:Qc-lower})
  is replaced by ``${};\omega(0)=c,\omega\text{ is unbounded})$''
  (in analogy with \cite{Dubins/Schwarz:1965}).
\end{remark}

\begin{remark}
  Theorem~\ref{thm:main-constructive} depends on the arbitrary choice
  ($\bbbd_n$)
  of the sequence of grids
  to define the quadratic variation process $A_t$.
  To make this less arbitrary,
  we could consider all grids whose mesh tends to zero fast enough
  and which are definable in the standard language of set theory
  (similarly to Wald's \cite{Wald:1937} suggested requirement
  for von Mises's collectives).
  When quadratic variation is defined via partitions of the time (horizontal) axis
  (as in L\'evy's paper \cite{Levy:1940}),
  Dudley \cite{Dudley:1973} shows that the rate of convergence $o(1/\log n)$ of the mesh to zero
  is sufficient for Brownian motion,
  and de la Vega \cite{deLaVega:1974} shows that this rate is slowest possible.
  It is an open question what the optimal rate of convergence is
  when quadratic variation is defined, as in this paper, via partitions of the vertical axis.
\end{remark}

\begin{remark}
  In this paper we construct quadratic variation $A$
  and define the stopping times $\tau_s$ in terms of $A$.
  Dubins and Schwarz \cite{Dubins/Schwarz:1965}
  construct $\tau_s$ directly
  (in a very similar way to our construction of $A$).
  An advantage of our construction
  (the game-theoretic counterpart of that in \cite{Karandikar:1983})
  is that the function $A(\omega)$ is continuous for typical $\omega$,
  whereas the event that the function $s\mapsto\tau_s(\omega)$ is continuous
  has lower price zero
  (Dubins and Schwarz's extra assumptions
  make this function continuous for almost all $\omega$).
\end{remark}

\begin{remark}\label{rem:coincidence}
  Theorem~\ref{thm:main} implies that the two notions of quadratic variation
  that we have discussed so far,
  $\qvar_t(\omega)$ defined by (\ref{eq:qvar}) and $A_t(\omega)$,
  coincide for all $t$ for typical $\omega$.
  Indeed, since $\qvar_t=A_t=t$, $\forall t\in[0,\infty)$,
  holds almost surely in the case of Brownian motion
  (see Lemma~\ref{lem:A-BM} for $A_t=t$),
  it suffices to check that the complement of the event
  $\forall t\in[0,\infty):\qvar_t=A_t$ belongs to $\KKK$.
  This follows from Lemma~\ref{lem:inv}
  and the analogous statement for $A$:
  if $\qvar_t(\omega)=A_t(\omega)$ for all $t$,
  we also have
  \begin{equation*}
    \qvar_{t}(\omega\circ f)
    =
    \qvar_{f(t)}(\omega)
    =
    A_{f(t)}(\omega)
    =
    A_{t}(\omega\circ f)
  \end{equation*}
  for all $t$.
\end{remark}

\section{Functional generalizations}
\label{sec:generalizations}

Theorems \ref{thm:main} and \ref{thm:main-constructive}(b)
are about upper price for sets,
but the former and part of the latter
can be generalized to cover the following more general notion of upper price
for \emph{functionals}, i.e., real-valued functions on $\Omega$.
The \emph{upper price} of a positive functional $F$
\emph{restricted to} a set $B\subseteq\Omega$
is defined by
\begin{equation}\label{eq:upper-expectation}
  \UpExpect(F;B)
  :=
  \inf
  \bigl\{
    \mathfrak{S}_0
    \bigm|
    \forall\omega\in B:
    \liminf_{t\to\infty}
    \mathfrak{S}_t(\omega)
    \ge
    F(\omega)
  \bigr\},
\end{equation}
where $\mathfrak{S}$ ranges over the positive capital processes.
This is the price of the cheapest positive superhedge for $F$
when Reality is restricted to choosing $\omega\in B$.
Restricted upper price for functionals
generalizes restricted upper price for sets:
$\UpProb(E;B)=\UpExpect(\III_E;B)$ for all $E\subseteq\Omega$.
When $B=\Omega$,
we abbreviate $\UpExpect(F;B)$ to $\UpExpect(F)$
and refer to $\UpExpect(F)$ as the \emph{upper price} of $F$.
Notice that $\UpExpect(F;B)=\UpExpect(F\III_B)$.

Let us say that a positive functional $F:\Omega\to[0,\infty)$ is \emph{$\KKK$-measurable}
if, for each constant $c\in[0,\infty)$,
the set $\{\omega\st F(\omega)\ge c\}$ is in $\KKK$.
(We need to spell out this definition since $\KKK$ is not a $\sigma$-algebra:
cf.\ Remark~\ref{rem:intersection}.)
Notice that the $\KKK$-measurability of $F$ means that $F$ is $\FFF$-measurable and,
for each $\omega\in\Omega$ and each time transformation $f$,
\begin{equation}\label{eq:invariant-function}
  F(\omega^f)
  \le
  F(\omega)
\end{equation}
(cf.\ (\ref{eq:invariant})).

\begin{remark}
  The presence of $\le$ in (\ref{eq:invariant-function}) is natural
  as, intuitively, transforming $\omega$ into $\omega^f$
  may involve cutting off part of $\omega$
  (step (a) at the beginning of Remark~\ref{rem:time-superinvariance}).
  It is clear that $F(\omega^f)=F(\omega)$
  when $f\in\GGG$.
\end{remark}

\begin{remark} 
  In terms of the partial order defined in Remark~\ref{rem:monotonicity},
  we can say that a functional $F$ is $\KKK$-measurable
  if and only if it is $\FFF$-measurable and monotonic.
\end{remark}

In this paper we will in fact prove the following generalization
of Theorem~\ref{thm:main}.
\begin{theorem}\label{thm:supermain}
  Let $c\in\bbbr$.
  Each positive $\KKK$-measurable functional $F:\Omega\to[0,\infty)$
  satisfies
  \begin{equation}\label{eq:supermain}
    \UpExpect(F;\omega(0)=c)
    =
    \int F \dd\Wiener_c.
  \end{equation}
\end{theorem}

The proof of the inequality $\ge$ in (\ref{eq:supermain}) is easy
and is accomplished by the following lemma;
it suffices to apply it to $\Wiener_c$ in place of $P$
and to $F\III_{\{\omega(0)=c\}}$ in place of $F$.

\begin{lemma}\label{lem:main-ge}
  Let $P$ be a probability measure on $(\Omega,\FFF)$
  such that the process $X_t(\omega):=\omega(t)$
  is a martingale w.r.\ to $P$ and the filtration $(\FFF_t)$.
  Then $\int F \dd P\le\UpExpect(F)$
  for any positive $\FFF$-measurable functional $F$.
\end{lemma}
\begin{proof}
  Fix a positive $\FFF$-measurable functional $F$ and let $\epsilon>0$.
  Find a positive capital process $\mathfrak{S}$
  of the form (\ref{eq:positive-capital})
  such that $\mathfrak{S}_0<\UpExpect(F)+\epsilon$
  and $\liminf_{t\to\infty}\mathfrak{S}_t(\omega)\ge F(\omega)$
  for all $\omega\in\Omega$.
  It can be checked using the optional sampling theorem
  (it is here that the boundedness of Sceptic's bets is used)
  that each addend in (\ref{eq:simple-capital}) is a martingale,
  and so each partial sum in (\ref{eq:simple-capital}) is a martingale
  and (\ref{eq:simple-capital}) itself is a local martingale.
  Since each addend in (\ref{eq:positive-capital}) is a positive local martingale,
  it is a supermartingale.
  By the monotone convergence theorem,
  the sum (\ref{eq:positive-capital}) of positive supermartingales
  is itself a positive supermartingale:
  if $0\le s<t$,
  \begin{multline*}
    \Expect
    \left(
      \mathfrak{S}_t \givn \FFF_s
    \right)
    =
    \Expect
    \left(
      \sum_{n=1}^{\infty}
      \K_t^{G_n,c_n}
      \givn
      \FFF_s
    \right)\\
    =
    \sum_{n=1}^{\infty}
    \Expect
    \left(
      \K_t^{G_n,c_n}
      \givn
      \FFF_s
    \right)
    \le
    \sum_{n=1}^{\infty}
    \K_s^{G_n,c_n}
    =
    \mathfrak{S}_s,
  \end{multline*}
  $\Expect(\cdot\givn\FFF_s)$ standing for the conditional expectation
  w.r.\ to the probability measure $P$.
  (The positive supermartingale $\mathfrak{S}$ is somewhat unusual
  in that it is not guaranteed to be right-continuous;
  however, it is lower semicontinuous
  as the limit of an increasing sequence of continuous processes.)
  Using Fatou's lemma,
  we now obtain
  \begin{equation}\label{eq:maximal}
    \int F \dd P
    \le
    \int
    \liminf_{t\to\infty}
    \mathfrak{S}_t
    \dd P
    \le
    \liminf_{t\to\infty}
    \int
    \mathfrak{S}_t
    \dd P
    \le
    \mathfrak{S}_0
    <
    \UpExpect(F)
    +
    \epsilon,
  \end{equation}
  where $t$ can be assumed to take only integer values.
  Since $\epsilon$ can be arbitrarily small,
  this implies the statement of the lemma.
\end{proof}

We will deduce the inequality $\le$ in Theorem~\ref{thm:supermain}
from the following generalization of the part of Theorem~\ref{thm:main-constructive}(b)
concerning $\UpQ_c$.

\begin{theorem}\label{thm:supermain-constructive}
  For any $c\in\bbbr$
  and any positive $\FFF$-measurable functional $F:\Omega\to[0,\infty)$,
  \begin{equation}\label{eq:inequality}
    \UpExpect(F\circ\ntt;\omega(0)=c,A_{\infty}=\infty)
    =
    \int_{\Omega}
    F
    \dd\Wiener_c
  \end{equation}
  (with $\circ$ standing for composition of two functions
  and with the convention that $(F\circ\ntt)(\omega):=0$
  when $\omega\notin\ntt^{-1}(\Omega)$).
\end{theorem}

\noindent
We will check that Theorem~\ref{thm:supermain-constructive}
(namely, the inequality $\le$ in (\ref{eq:inequality}))
indeed implies Theorem~\ref{thm:main-constructive}(b)
in Section~\ref{sec:proof-constructive-b}.
In this section we will only prove the easy inequality $\ge$ in (\ref{eq:inequality}).
In Section~\ref{sec:2-variation} (Lemma~\ref{lem:A-BM})
we will see that $\UpA_t(\omega)=\LowA_t(\omega)=t$ for all $t\in[0,\infty)$
for $\Wiener_c$-almost all $\omega$;
therefore,
$\ntt(\omega)=\omega$ for $\Wiener_c$-almost all $\omega$.
In conjunction with Lemma~\ref{lem:main-ge},
this implies the inequality $\ge$ in (\ref{eq:inequality}):
\begin{multline*}
  \UpExpect(F\circ\ntt;\omega(0)=c,A_{\infty}=\infty)
  =
  \UpExpect((F\circ\ntt)\III_{\{\omega(0)=c,A_{\infty}=\infty\}})\\
  \ge
  \int_{\Omega}
  (F\circ\ntt)\III_{\{\omega(0)=c,A_{\infty}=\infty\}}
  \dd\Wiener_c
  =
  \int_{\Omega}
  F
  \dd\Wiener_c.
\end{multline*}

\begin{remark}
  Theorem~\ref{thm:supermain} gives the price
  of the cheapest superhedge for the contingent claim $F$,
  but it is not applicable to the standard contingent claims traded in financial markets,
  which are not $\KKK$-measurable.
  This theorem would be applicable to the imaginary contingent claim
  paying $f(\omega(\tau_S))$ at time $\tau_S$
  (cf.\ (\ref{eq:stopping}); there is no payment if $\tau_S=\infty$),
  where $S>0$ is a given constant
  and $f$ is a given positive and measurable payoff function.
  (If the interest rate $r$ is constant but different from 0,
  we can consider the contingent claim paying $e^{\tau_S r}f(\omega(\tau_S))$
  at time $\tau_S$.)
  The price of the cheapest superhedge
  will be $\int f(\psi(S))\Wiener_c(\dd\psi)$,
  where $c:=\omega(0)$,
  if there are no restrictions on $\omega\in\Omega$,
  but will become
  $\int f(\psi(S))\III_{\{\forall s\in[0,S]:\psi(s)\ge 0\}}\Wiener_c(\dd\psi)$
  if $\omega$ is restricted to be positive (as in many real financial markets).
  \ifWP
    Similar contingent claims were considered by Bick \cite{Bick:1995}
    and later marketed by Soci\'et\'e G\'en\'erale Corporate and Investment Banking
    under the name of timer options,
    but in timer options $\tau_S$ is replaced by the moment
    when the realised variance exceeds the \emph{a priori} chosen bound $S$.
    The methods of this paper can be used to price timer options:
    see, e.g., \cite{Beiglbock/etal:2015}.
  \fi
\end{remark}

The following lemma reduces Theorems~\ref{thm:supermain}
and~\ref{thm:supermain-constructive}
to the case of bounded~$F$.

\begin{lemma}
  Without loss of generality,
  we can assume that the functional $F$ in (\ref{eq:supermain}) and (\ref{eq:inequality})
  is bounded.
\end{lemma}

\begin{proof}
  The inequalities $\ge$ in (\ref{eq:supermain}) and (\ref{eq:inequality})
  have already been proved, so we will concentrate on the inequalities $\le$.
  We will only consider the case of (\ref{eq:supermain});
  the case of (\ref{eq:inequality}) is analogous.

  Suppose (\ref{eq:supermain}) holds
  for all bounded positive $\KKK$-measurable functionals $F:\Omega\to[0,\infty)$,
  and let $F:\Omega\to[0,\infty)$ be an unbounded positive $\KKK$-measurable functional.
  Represent $F$ as the sum of bounded positive $\KKK$-measurable functionals
  $F_n:\Omega\to[0,\infty)$, $n\in\{0,1,\ldots\}$,
  defined by $F_n:=0\vee(F-n)\wedge1$:
  $F=\sum_{n=0}^{\infty}F_n$.
  To check that $F_n$ are indeed $\KKK$-measurable,
  notice that, since $0\vee(u-n)\wedge1$ is an increasing function of $u$,
  (\ref{eq:invariant-function}) will continue to hold if we replace $F$ by $F_n$.
  Now we can apply (\ref{eq:supermain}) to $F_n$:
  \begin{multline*}
    \UpExpect(F;\omega(0)=c)
    =
    \UpExpect
    \left(
      \sum_{n=0}^{\infty}
      F_n
      ;\omega(0)=c
    \right)\\
    \le
    \sum_{n=0}^{\infty}
    \UpExpect
    \left(
      F_n
      ;\omega(0)=c
    \right)
    \le
    \sum_{n=0}^{\infty}
    \int F_n \dd\Wiener_c
    =
    \int F \dd\Wiener_c.
    \qedhere
  \end{multline*}
\end{proof}

Sections \ref{sec:coherence}--\ref{sec:proof-main-le}
are mainly devoted to the proof
of the remaining statements in Theorems \ref{thm:main-constructive},
\ref{thm:supermain}, and \ref{thm:supermain-constructive}
(the last two for bounded $F$),
whereas Section~\ref{sec:literature} is devoted to the discussion
of the financial meaning of our results
and their connections with related probabilistic and financial literature.


\section{Coherence}
\label{sec:coherence}

The following trivial result says that our trading game is \emph{coherent},
in the sense that $\UpProb(\Omega)=1$
(i.e.,
no positive capital process increases its value between time $0$ and $\infty$
by more than a strictly positive constant for all $\omega\in\Omega$).
\begin{lemma}\label{lem:coherence}
  $\UpProb(\Omega)=1$.
  Moreover, for each $c\in\bbbr$,
  $\UpProb(\omega(0)=c)=1$.
\end{lemma}
\begin{proof}
  No positive capital process can strictly increase its value
  on a constant $\omega\in\Omega$.
\end{proof}

Lemma~\ref{lem:coherence}, however, does not even guarantee
that the set of non-constant elements of $\Omega$ has upper price one.
The theory of measure-theoretic probability provides us
with a plethora of non-trivial events of upper price one.
\begin{lemma}\label{lem:super-coherence}
  Let $E$ be an event that almost surely contains the sample path
  of a continuous martingale with time interval $[0,\infty)$.
  Then $\UpProb(E)=1$.
\end{lemma}
\begin{proof}
  This is a special case of Lemma~\ref{lem:main-ge}
  applied to $F:=\III_E$.
\end{proof}

In particular,
applying Lemma~\ref{lem:super-coherence} to Brownian motion
started at $c\in\bbbr$ gives
\begin{equation}\label{eq:DS-one}
  \UpProb(\omega(0)=c,\omega\in\DS)
  =
  1
\end{equation}
and
\begin{equation}\label{eq:one}
  \UpProb(\omega(0)=c,A_{\infty}=\infty)=1
\end{equation}
(for the latter we also need Lemma~\ref{lem:A-BM} below).
Both (\ref{eq:DS-one}) and (\ref{eq:one}) have been used above.

\begin{lemma}\label{lem:lower-upper}
  Let $\UpProb(B)=1$.
  For every set $E\subseteq\Omega$,
  $\LowProb(E;B)\le\UpProb(E;B)$.
\end{lemma}
\begin{proof}
  Suppose $\LowProb(E;B)>\UpProb(E;B)$ for some $E$;
  by the definition of $\LowProb$,
  this would mean that $\UpProb(E;B)+\UpProb(E^c;B)<1$.
  Since $\UpProb(\cdot;B)$ is finitely subadditive
  (see Lemma~\ref{lem:subadditivity}),
  this would imply $\UpProb(\Omega;B)<1$,
  which is equivalent to $\UpProb(B)<1$
  and, therefore, contradicts our assumption.
\end{proof}

\section{Existence of quadratic variation}
\label{sec:2-variation}

In this paper,
the set $\Omega$ is always equipped with the metric
\begin{equation}\label{eq:metric}
  \rho(\omega_1,\omega_2)
  :=
  \sum_{d=1}^{\infty}
  2^{-d}
  \sup_{t\in[0,2^d]}
  \left(
    \lvert \omega_1(t)-\omega_2(t) \rvert
    \wedge
    1
  \right)
\end{equation}
(and the corresponding topology and Borel $\sigma$-algebra,
the latter coinciding with $\FFF$).
This makes it a complete and separable metric space.
The main goal of this section is to prove
that the sequence of continuous functions
$t\in[0,\infty)\mapsto A^n_t(\omega)$
is convergent in $\Omega$ for typical $\omega$;
this is done in Lemma \ref{lem:convergence}.
This will establish the existence of $A(\omega)\in\Omega$ for typical $\omega$,
which is part of Theorem \ref{thm:main-constructive}(a).
It is obvious that, when it exists,
$A(\omega)$ is increasing and $A_0(\omega)=0$.
The last part of Theorem \ref{thm:main-constructive}(a),
asserting that the intervals of constancy of $\omega$ and $A(\omega)$
coincide for typical $\omega$,
will be proved in the next section
(Lemma~\ref{lem:same-intervals}).
\begin{lemma}\label{lem:quadratic-variation}
  For each $T>0$,
  for typical $\omega$,
  $t\in[0,T]\mapsto A^n_t$
  is a Cauchy sequence of functions in $C[0,T]$.
\end{lemma}
\begin{proof}
  Fix a $T>0$ and fix temporarily an $n\in\{1,2,\ldots\}$.
  Let $\kappa\in\{0,1\}$ be such that $T^{n-1}_0=T^n_{\kappa}$
  and,
  for each $k=1,2,\ldots$,
  let
  \begin{equation*}
    \xi_k
    :=
    \begin{cases}
      1 & \text{if $\omega(T^n_{\kappa+2k})=\omega(T^n_{\kappa+2k-2})$}\\
      -1 & \text{otherwise}
    \end{cases}
  \end{equation*}
  (this is only defined when $T^n_{\kappa+2k}<\infty$).
  If $\omega$ were generated by Brownian motion,
  $\xi_k$ would be a random variable taking value $j$, $j\in\{1,-1\}$,
  with probability $1/2$;
  in particular,
  the expected value of $\xi_k$ would be $0$.
  As the standard backward induction procedure shows,
  this remains true in our current framework
  in the following game-theoretic sense:
  there exists a simple trading strategy that,
  when started with initial capital $0$ at time $T^n_{\kappa+2k-2}$,
  ends with $\xi_k$ at time $T^n_{\kappa+2k}$,
  provided both times are finite;
  moreover, the corresponding simple capital process is always between $-1$ and $1$.
  (Namely,
  at time $T^n_{\kappa+2k-1}$
  bet $-2^n$ if $\omega(T^n_{\kappa+2k-1})>\omega(T^n_{\kappa+2k-2})$
  and bet $2^n$ otherwise.)
  Notice that the increment of the process $A^n_t-A^{n-1}_t$
  over the time interval $[T^n_{\kappa+2k-2},T^n_{\kappa+2k}]$ is
  \begin{equation*}
    \eta_k
    :=
    \begin{cases}
      2(2^{-n})^2=2^{-2n+1} & \text{if $\xi_k=1$}\\
      2(2^{-n})^2-(2^{-n+1})^2=-2^{-2n+1} & \text{if $\xi_k=-1$},
    \end{cases}
  \end{equation*}
  i.e., $\eta_k=2^{-2n+1}\xi_k$.

  The game-theoretic version of Hoeffding's inequality
  (see Theorem \ref{thm:super} in Appendix below)
  shows that for any constant $\lambda\in\bbbr$
  there exists a simple capital process $\mathfrak{S}^n$
  with $\mathfrak{S}^n_0=1$ such that,
  for all $K=0,1,2,\ldots$,
  \begin{equation}\label{eq:S}
    \mathfrak{S}^n_{T^n_{\kappa+2K}}
    \ge
    \prod_{k=1}^K
    \exp
    \left(
      \lambda\eta_k
      -
      2^{-4n+1}\lambda^2
    \right).
  \end{equation}
  According to Equation (\ref{eq:positive}) in Appendix
  (with $x_n$ corresponding to $\eta_k$),
  such $\mathfrak{S}^n$ can be defined as the capital process
  of the simple trading strategy betting the current capital times
  \begin{equation*}
    \frac{e^{\lambda 2^{-2n+1}}-e^{-\lambda 2^{-2n+1}}}{2^{-2n+2}}
    \exp
    \left(
      -\frac{\lambda^2}{8}
      \left(
        2^{-2n+2}
      \right)^2
    \right)
  \end{equation*}
  on $A^n_t-A^{n-1}_t$ at each time $T^n_{\kappa+2k-2}$,
  $k\in\{1,2,\ldots\}$.
  In terms of the original security,
  this simple trading strategy bets $0$ on $\omega$
  at each time $T^n_{\kappa+2k-2}$
  and bets the current capital times
  \begin{equation*}
    2
    \left(
      \omega(T^n_{\kappa+2k-2})
      -
      \omega(T^n_{\kappa+2k-1})
    \right)
    \frac{e^{\lambda 2^{-2n+1}}-e^{-\lambda 2^{-2n+1}}}{2^{-2n+2}}
    \exp
    \left(
      -\frac{\lambda^2}{8}
      \left(
        2^{-2n+2}
      \right)^2
    \right)
  \end{equation*}
  on $\omega$
  at each time $T^n_{\kappa+2k-1}$,
  $k\in\{1,2,\ldots\}$.
  It is clear that the process $\mathfrak{S}^n$ is positive:
  it is constant in each time interval $[T^n_{\kappa+2k-2},T^n_{\kappa+2k-1}]$,
  and is linear in $\omega(t)$
  in each time interval $[T^n_{\kappa+2k-1},T^n_{\kappa+2k}]$;
  therefore, its positivity follows from its positivity
  (cf.\ (\ref{eq:S}))
  at the points $T^n_{\kappa+2K}$, $K\in\{0,1,2,\ldots\}$.

  Fix temporarily $\alpha>0$.
  It is easy to see that,
  since the sum of the positive capital processes $\mathfrak{S}^n$
  over $n=1,2,\ldots$ with weights $2^{-n}$
  will also be a positive capital process,
  none of these processes will ever exceed $2^n2/\alpha$
  except for a set of $\omega$ of upper price at most $\alpha/2$.
  The inequality
  \begin{equation*}
    \prod_{k=1}^K
    \exp
    \left(
      \lambda\eta_k
      -
      2^{-4n+1}\lambda^2
    \right)
    \le
    2^n
    \frac{2}{\alpha}
    \le
    e^n
    \frac{2}{\alpha}
  \end{equation*}
  can be equivalently rewritten as
  \begin{equation}\label{eq:equivalent}
    \lambda
    \sum_{k=1}^K
    \eta_k
    \le
    K \lambda^2 2^{-4n+1}
    +
    n
    +
    \ln\frac{2}{\alpha}.
  \end{equation}
  Plugging in the identities
  \begin{align*}
    K
    &=
    \frac
    {
      A^n_{T_{\kappa+2K}^n}
      -
      A^n_{T_{\kappa}^n}
    }
    {
      2^{-2n+1}
    },\\
    \sum_{k=1}^K\eta_k
    &=
    \left(
      A^n_{T_{\kappa+2K}^n}
      -
      A^n_{T_{\kappa}^n}
    \right)
    -
    \left(
      A^{n-1}_{T_{\kappa+2K}^n}
      -
      A^{n-1}_{T_{\kappa}^n}
    \right),
  \end{align*}
  and taking $\lambda:=2^n$,
  we can transform (\ref{eq:equivalent}) to
  \begin{equation}\label{eq:decreasing-prepreliminary}
    \left(
      A^n_{T_{\kappa+2K}^n}
      -
      A^n_{T_{\kappa}^n}
    \right)
    -
    \left(
      A^{n-1}_{T_{\kappa+2K}^n}
      -
      A^{n-1}_{T_{\kappa}^n}
    \right)
    \le
    2^{-n}
    \left(
      A^n_{T_{\kappa+2K}^n}
      -
      A^n_{T_{\kappa}^n}
    \right)
    +
    \frac{n+\ln\frac{2}{\alpha}}{2^n},
  \end{equation}
  which implies
  \begin{equation}\label{eq:decreasing-preliminary}
    A^n_{T_{\kappa+2K}^n}
    -
    A^{n-1}_{T_{\kappa+2K}^n}
    \le
    2^{-n}
    A^n_{T_{\kappa+2K}^n}
    +
    2^{-2n+1}
    +
    \frac{n+\ln\frac{2}{\alpha}}{2^n}.
  \end{equation}
  This is true for any $K=0,1,2,\ldots$;
  choosing the largest $K$ such that $T^n_{\kappa+2K}\le t$,
  we obtain
  \begin{equation}\label{eq:decreasing}
    A^n_t
    -
    A^{n-1}_t
    \le
    2^{-n}
    A^n_t
    +
    2^{-2n+2}
    +
    \frac{n+\ln\frac{2}{\alpha}}{2^n},
  \end{equation}
  for any $t\in[0,\infty)$
  (the simple case $t<T^n_{\kappa}$ has to be considered separately).
  Proceeding in the same way but taking $\lambda:=-2^n$,
  we obtain
  \begin{equation*}
    \left(
      A^n_{T_{\kappa+2K}^n}
      -
      A^n_{T_{\kappa}^n}
    \right)
    -
    \left(
      A^{n-1}_{T_{\kappa+2K}^n}
      -
      A^{n-1}_{T_{\kappa}^n}
    \right)
    \ge
    -2^{-n}
    \left(
      A^n_{T_{\kappa+2K}^n}
      -
      A^n_{T_{\kappa}^n}
    \right)
    -
    \frac{n+\ln\frac{2}{\alpha}}{2^n}
  \end{equation*}
  instead of (\ref{eq:decreasing-prepreliminary})
  and
  \begin{equation*}
    A^n_{T_{\kappa+2K}^n}
    -
    A^{n-1}_{T_{\kappa+2K}^n}
    \ge
    -2^{-n}
    A^n_{T_{\kappa+2K}^n}
    -
    2^{-2n+1}
    -
    \frac{n+\ln\frac{2}{\alpha}}{2^n}
  \end{equation*}
  instead of (\ref{eq:decreasing-preliminary}),
  which gives
  \begin{equation}\label{eq:increasing}
    A^n_t
    -
    A^{n-1}_t
    \ge
    -2^{-n}
    A^n_t
    -
    2^{-2n+2}
    -
    \frac{n+\ln\frac{2}{\alpha}}{2^n}
  \end{equation}
  instead of (\ref{eq:decreasing}).
  We know that that (\ref{eq:decreasing}) and (\ref{eq:increasing})
  hold for all $t\in[0,\infty)$ and all $n=1,2,\ldots$
  except for a set of $\omega$ of upper price at most $\alpha$.

  Now we have all ingredients to complete the proof.
  Suppose there exists $\alpha>0$
  such that (\ref{eq:decreasing}) and (\ref{eq:increasing})
  hold for all $n=1,2,\ldots$
  (this is true for typical $\omega$).
  First let us show that the sequence $A^n_T$, $n=1,2,\ldots$, is bounded.
  Define a new sequence $B^n$, $n=0,1,2,\ldots$, as follows:
  $B^0:=A^0_T$
  and $B^n$, $n=1,2,\ldots$, are defined inductively by
  \begin{equation}\label{eq:decreasing-B}
    B^n
    :=
    \frac{1}{1-2^{-n}}
    \left(
      B^{n-1}
      +
      2^{-2n+2}
      +
      \frac{n+\ln\frac{2}{\alpha}}{2^n}
    \right)
  \end{equation}
  (notice that this is equivalent to (\ref{eq:decreasing})
  with $B^n$ in place of $A^n_t$ and $=$ in place of $\le$).
  As $A^n_T\le B^n$ for all $n$,
  it suffices to prove that $B^n$ is bounded.
  If it is not, $B^N\ge1$ for some $N$.
  By (\ref{eq:decreasing-B}),
  $B^n\ge1$ for all $n\ge N$.
  Therefore, again by (\ref{eq:decreasing-B}),
  \begin{equation*}
    B^n
    \le
    B^{n-1}
    \frac{1}{1-2^{-n}}
    \left(
      1
      +
      2^{-2n+2}
      +
      \frac{n+\ln\frac{2}{\alpha}}{2^n}
    \right),
    \quad
    n>N,
  \end{equation*}
  and the boundedness of the sequence $B^n$
  follows from $B^N<\infty$
  and
  \begin{equation*}
    \prod_{n=N+1}^{\infty}
    \frac{1}{1-2^{-n}}
    \left(
      1
      +
      2^{-2n+2}
      +
      \frac{n+\ln\frac{2}{\alpha}}{2^n}
    \right)
    <
    \infty.
  \end{equation*}
  Now it is obvious that the sequence $A^n_t$ is Cauchy in $C[0,T]$:
  (\ref{eq:decreasing}) and (\ref{eq:increasing}) imply
  \begin{equation*}
    \left|
      A^n_t
      -
      A^{n-1}_t
    \right|
    \le
    2^{-n}
    A^n_T
    +
    2^{-2n+2}
    +
    \frac{n+\ln\frac{2}{\alpha}}{2^n}
    =
    O(n/2^n).
    \qedhere
  \end{equation*}
\end{proof}

Lemma \ref{lem:quadratic-variation} implies
that, for typical $\omega$, the sequence $t\in[0,\infty)\mapsto A^n_t$
is Cauchy in $\Omega$.
Therefore, we have the following implication.
\begin{lemma}\label{lem:convergence}
  The event that
  the sequence of functions $t\in[0,\infty)\mapsto A^n_t$
  converges in $\Omega$
  is full.
\end{lemma}

We can see that the first term in the conjunction in (\ref{eq:stopping})
holds for typical $\omega$;
let us check that $\tau_s$ itself is a stopping time.
\begin{lemma}\label{lem:stopping}
  For each $s\ge0$,
  the function $\tau_s$ defined by (\ref{eq:stopping}) is a stopping time.
\end{lemma}
\begin{proof}
  It suffices to check that the condition $\tau_s\le t$
  can be written as
  \begin{equation}\label{eq:le-t}
        \forall(q_1,q_2)\subseteq(0,s)
        \;
        \exists q \in (0,t)\cap\bbbq:
        \UpA_q = \LowA_q \in (q_1,q_2),
  \end{equation}
  where $(q_1,q_2)$ range over the non-empty intervals with rational end-points.
  Let $T$ be the largest number in $[0,\infty]$
  such that the functions $\UpA|_{[0,T)}$ and $\LowA|_{[0,T)}$ coincide
  and are continuous;
  we will use $A'$ as the common notation for $\UpA|_{[0,T)}=\LowA|_{[0,T)}$.
  The condition $\tau_s\le t$ means that for some $t'\in[0,t]$
  the domain of $A'$ includes $[0,t')$ and $\sup_{u<t'}A'_u=s$.
  Now it is clear that the condition (\ref{eq:le-t})
  is satisfied if $\tau_s\le t$.
  In the opposite direction,
  suppose (\ref{eq:le-t}) is satisfied.
  Then $\UpA_u=\LowA_u$ whenever $u\in(0,t)$ satisfies $\LowA_u<s$.
  Indeed, if we had $\LowA_u<\UpA_u$ for such $u$,
  we could choose $(q_1,q_2)\subseteq(0,s)$ satisfying
  $\LowA_u<q_1<q_2<\UpA_u$
  and there would be no $q$ satisfying the required properties
  in (\ref{eq:le-t}):
  if $q\le u$, $\LowA_q\le\LowA_u<q_1$,
  and if $q\ge u$, $\UpA_q\ge\UpA_u>q_2$.
  Combining this result with (\ref{eq:le-t}),
  we can see that there is a function $A''$ with a domain $[0,t'')\subseteq[0,t)$
  such that $A''_u=\UpA_u=\LowA_u$ for all $u\in[0,t'')$
  and $\sup A''=s$.
  The function $A''$ is increasing and,
  by (\ref{eq:le-t}),
  continuous;
  this implies $\tau_s\le t$.
\end{proof}

Let us now consider the case of Brownian motion.
\begin{lemma}\label{lem:A-BM}
  For any $c\in\bbbr$,
  $\Wiener_c(\forall t\in[0,\infty):\UpA_t=\LowA_t=t)=1$.
\end{lemma}
\begin{proof}
  It suffices to consider only rational values of $t$
  and, therefore, a fixed value of $t$.
  The convergence $A^n_t\to t$ (see (\ref{eq:A})) in $\Wiener_c$-probability
  can be deduced from the law of large numbers applied to $T^n_k$:
  \begin{itemize}
  \item
    the law of large numbers implies that $A^n_t\to t$ in $\Wiener_c$-probability
    since $\int(T^n_k-T^n_{k-1})\dd\Wiener_c=2^{-2n}$
    (this is a combination of the second statement of Theorem 2.49
    in \cite{Morters/Peres:2010},
    which is a corollary of Wald's second lemma,
    with the strong Markov property of Brownian motion);
  \item
    the law of large numbers is applicable
    because $\int(T^n_k-T^n_{k-1})^2\dd\Wiener_c<\infty$
    (see the proof of the second statement of Theorem 2.49
    in \cite{Morters/Peres:2010}).
  \end{itemize}
  It remains to apply Lemma~\ref{lem:convergence},
  which, in combination with Lemma~\ref{lem:main-ge}
  (applied to the indicator functions of events),
  implies that the sequence $A^n$ converges in $\Omega$
  $\Wiener_c$-almost surely.
\end{proof}

\begin{remark}
  This section is about the quadratic variation of the price path,
  but in finance the quadratic variation of the stochastic logarithm
  (see, e.g., \cite{Jacod/Shiryaev:2003}, p.~134)
  of a price process
  is usually even more important than the quadratic variation of the price process itself.
  A pathwise version of the stochastic logarithm has been studied
  by Norvai\v{s}a in \cite{Norvaisa:2000,Norvaisa:2001}.
  Consider an $\omega\in\Omega$ such that $A(\omega)$ exists,
  belongs to $\Omega$, and has the same intervals of constancy as $\omega$;
  Theorem \ref{thm:main-constructive}(a) says
  that these conditions are satisfied for typical $\omega$.
  Fix a time horizon $T>0$ and suppose, additionally,
  that $\inf_{t\in[0,T]}\omega(t)>0$.
  The limit
  \begin{equation*}
    R_t(\omega)
    :=
    \lim_{n\to\infty}
    \sum_{k=0}^{\infty}
    \frac
    {
      \omega(T^n_k\wedge t)
      -
      \omega(T^n_{k-1}\wedge t)
    }
    {
      \omega(T^n_{k-1}\wedge t)
    }
  \end{equation*}
  (where we use the same notation as in (\ref{eq:A}))
  exists for all $t\in[0,T]$
  and the function $R(\omega):t\in[0,T]\mapsto R_t(\omega)$ satisfies
  (\cite{Norvaisa:2001}, Proposition 56)
  \begin{equation*}
    R_t(\omega)
    =
    \ln\frac{\omega(t)}{\omega(0)}
    +
    \frac12
    \int_0^t
    \frac{\dd A_s(\omega)}{\omega^2(s)},
    \quad
    t\in[0,T].
  \end{equation*}
  In financial terms,
  the value $R_t(\omega)$ is the cumulative return of the security $\omega$ over $[0,t]$
  (\cite{Norvaisa:2000}, Section~2);
  in probabilistic terms,
  $R(\omega)$ is the pathwise stochastic logarithm of $\omega$.
  The quadratic variation of $R(\omega)$ can be defined as
  \begin{equation*}
    \lim_{n\to\infty}
    \sum_{k=0}^{\infty}
    \left(
      R_{T^n_k\wedge t}(\omega)
      -
      R_{T^n_{k-1}\wedge t}(\omega)
    \right)^2
    =
    \int_0^t
    \frac{\dd A_s(\omega)}{\omega^2(s)}
  \end{equation*}
  (the existence of the limit and the equality
  are also parts of Proposition 56 in \cite{Norvaisa:2001}).
\end{remark}

\begin{remark}
  Analogues for c\`adl\`ag price paths of the main results of this section
  can be found in \cite{Vovk:2015-LMJ}.
\end{remark}

\section{Tightness}
\label{sec:tight}

In this section we will do some groundwork
for the proof of Theorems~\ref{thm:main-constructive}(b)
and~\ref{thm:supermain-constructive},
and will also finish the proof of Theorem~\ref{thm:main-constructive}(a).
We start from the results that show (see the next section)
that $\LowQ_c$
is tight in the topology induced by the metric~(\ref{eq:metric}).
\begin{lemma}\label{lem:modulus-P}
  For each $\alpha>0$ and $S\in\{1,2,4\ldots\}$,
  \begin{multline}\label{eq:modulus-P}
    \LowProb
    \bigl(
      \forall\delta\in(0,1)\;
      \forall s_1,s_2\in[0,S]:
      \left(
        0\le s_2-s_1\le\delta
        \And
        \tau_{s_2}<\infty
      \right)\\
      \Longrightarrow
      \left|
        \omega(\tau_{s_2})
        -
        \omega(\tau_{s_1})
      \right|
      \le
      230\,
      \alpha^{-1/2}
      S^{1/4}
      \delta^{1/8}
    \bigr)
    \ge
    1-\alpha.
  \end{multline}
\end{lemma}
\begin{proof}
  Let $S=2^d$, where $d\in\{0,1,2,\ldots\}$.
  For each $m=1,2,\ldots$,
  divide the interval $[0,S]$ into $2^{d+m}$ equal subintervals
  of length $2^{-m}$.
  Fix, for a moment, such an $m$,
  and set
  $
    \beta=\beta_m
    :=
    (2^{1/4}-1)
    2^{-m/4}
    \alpha
  $
  (where $2^{1/4}-1$ is the normalizing constant
  ensuring that the $\beta_m$ sum to $\alpha$)
  and
  \begin{equation}\label{eq:t}
    t_i
    :=
    \tau_{i2^{-m}},
    \enspace
    \omega_i
    :=
    \omega(t_i),
    \quad
    i=0,1,\ldots,2^{d+m}
  \end{equation}
  (we will be careful to use $\omega_i$ only when $t_i<\infty$).

  We will first replace the quadratic variation process $A$
  (in terms of which the stopping times $\tau_s$ are defined)
  by a version of $A^l$ for a large enough $l$.
  If $\tau$ is any stopping time
  (we will be interested in $\tau=t_i$ for various $i$),
  set, in the notation of (\ref{eq:A}),
  \begin{equation*}
    A^{n,\tau}_t(\omega)
    :=
    \sum_{k=0}^{\infty}
    \left(
      \omega(\tau\vee T^n_k\wedge t)
      -
      \omega(\tau\vee T^n_{k-1}\wedge t)
    \right)^2,
    \quad
    t\ge\tau,
    \quad
    n=1,2,\ldots
  \end{equation*}
  (we omit parentheses in expressions of the form $x\vee y\wedge z$
  since $(x\vee y)\wedge z=x\vee(y\wedge z)$,
  provided $x\le z$).
  The intuition is that $A^{n,\tau}_t(\omega)$
  is the version of $A^{n}_t(\omega)$ that starts at time $\tau$ rather than $0$.

  For $i=0,1,\ldots,2^{d+m}-1$,
  let $\mathfrak{E}_i$ be the event that
  $t_i<\infty$ implies that (\ref{eq:increasing}),
  with $\alpha$ replaced by $\gamma>0$
  and $A^n_t$ replaced by $A^{n,t_i}_t$,
  holds for all $n=1,2,\ldots$ and $t\in[t_i,\infty)$.
  Applying a trading strategy
  similar to that used in the proof of Lemma \ref{lem:quadratic-variation}
  but starting at time $t_i$ rather than $0$,
  we can see that the lower price of $\mathfrak{E}_i$
  is at least $1-\gamma$.
  The inequality
  \begin{equation*}
    A^{n,t_i}_t
    -
    A^{n-1,t_i}_t
    \ge
    -2^{-n}
    A^{n,t_i}_t
    -
    2^{-2n+2}
    -
    \frac{n+\ln\frac{2}{\gamma}}{2^n}
  \end{equation*}
  holds for all $t\in[t_i,t_{i+1}]$ and all $n$
  on the event $\{t_i<\infty\}\cap\mathfrak{E}_i$.
  For the value $t:=t_{i+1}$ this inequality implies
  \begin{equation*}
    A^{n,t_i}_{t_{i+1}}
    \ge
    \frac{1}{1+2^{-n}}
    \left(
      A^{n-1,t_i}_{t_{i+1}}
      -
      2^{-2n+2}
      -
      \frac{n+\ln\frac{2}{\gamma}}{2^n}
    \right)
  \end{equation*}
  (including the case $t_{i+1}=\infty$).
  Applying the last inequality to $n=l+1,l+2,\ldots$
  (where $l$ will be chosen later),
  we obtain that
  \begin{equation}\label{eq:from-below}
    A^{\infty,t_i}_{t_{i+1}}
    \ge
    \left(
      \prod_{n=l+1}^{\infty} \frac{1}{1+2^{-n}}
    \right)
    A^{l,t_i}_{t_{i+1}}
    -
    \sum_{n=l+1}^{\infty}
    \left(
      2^{-2n+2}
      +
      \frac{n+\ln\frac{2}{\gamma}}{2^n}
    \right)
  \end{equation}
  holds on the whole of $\{t_i<\infty\}\cap\mathfrak{E}_i$
  except perhaps a null set.
  The qualification ``except a null set''
  allows us not only to assume that 
  $A^{\infty,t_i}_{t_{i+1}}$
  exists in (\ref{eq:from-below})
  but also to assume that
  $A^{\infty,t_i}_{t_{i+1}}=A_{t_{i+1}}-A_{t_i}=2^{-m}$.
  Let $\gamma:=\frac{1}{3}2^{-d-m}\beta$
  and choose $l=l(m)$ so large
  that (\ref{eq:from-below}) implies
  $A^{l,t_i}_{t_{i+1}}\le2^{-m+1/2}$
  (this can be done as both the product and the sum in (\ref{eq:from-below})
  are convergent,
  and so the product can be made arbitrarily close to $1$
  and the sum can be made arbitrarily close to $0$).
  Doing this for all $i=0,1,\ldots,2^{d+m}-1$ will ensure
  that the lower price of
  \begin{equation}\label{eq:bound}
    t_i<\infty
    \Longrightarrow
    A^{l,t_i}_{t_{i+1}}\le2^{-m+1/2},
    \quad
    i=0,1,\ldots,2^{d+m}-1,
  \end{equation}
  is at least $1-\beta/3$.

  An important observation for what follows is
  that the process defined as $(\omega(t)-\omega(t_i))^2-A^{l,t_i}_t$ for $t\ge t_i$
  and as $0$ for $t<t_i$
  is a simple capital process
  (corresponding to betting $2(\omega(T^l_k)-\omega(t_i))$
  at each time $T^l_k>t_i$).
  Now we can see that
  \begin{equation}\label{eq:2-variation}
    \sum_{i=1,\ldots,2^{d+m}:t_i<\infty}
    (\omega_i - \omega_{i-1})^2
    \le
    2^{1/2}\frac{3}{\beta}S
  \end{equation}
  will hold on the event (\ref{eq:bound}),
  except for a set of $\omega$ of upper price at most $\beta/3$:
  indeed, there is a positive simple capital process
  taking value at least
  $
    2^{1/2}S
    +
    \sum_{i=1}^{j}
    (\omega_i - \omega_{i-1})^2
    -
    j2^{-m+1/2}
  $
  on the conjunction of events (\ref{eq:bound}) and $t_j<\infty$
  at time $t_j$,
  $j=0,1,\ldots,2^{d+m}$,
  and this simple capital process
  will make at least $2^{1/2}\frac{3}{\beta}S$ at time $\tau_S$
  (in the sense of $\liminf$ if $\tau_S=\infty$)
  out of initial capital $2^{1/2}S$
  if (\ref{eq:bound}) happens but (\ref{eq:2-variation}) fails to happen.

  For each $\omega\in\Omega$,
  define
  \begin{equation*} 
    J(\omega)
    :=
    \left\{
      i=1,\ldots,2^{d+m}:
      t_i<\infty
      \And
      \lvert\omega_i-\omega_{i-1}\rvert\ge\epsilon
    \right\},
  \end{equation*}
  where $\epsilon=\epsilon_m$ will be chosen later.
  It is clear that $\lvert J(\omega)\rvert\le2^{1/2}3S/\beta\epsilon^2$
  on the set (\ref{eq:2-variation}).
  Consider the simple trading strategy
  whose capital increases by $(\omega(t_i)-\omega(\tau))^2-A^{l,\tau}_{t_i}$
  between each time $\tau\in[t_{i-1},t_i]\cap[0,\infty)$
  when $\lvert\omega(\tau)-\omega_{i-1}\rvert=\epsilon$
  for the first time during $[t_{i-1},t_i]\cap[0,\infty)$
  (this is guaranteed to happen when $i\in J(\omega)$)
  and the corresponding time $t_i$,
  $i=1,\ldots,2^{d+m}$,
  and which is not active (i.e., sets the bet to $0$) otherwise.
  (Such a strategy exists, as explained in the previous paragraph.)
  This strategy will make at least $\epsilon^2$
  out of $(2^{1/2}3S/\beta\epsilon^2)2^{-m+1/2}$
  provided all three of the events (\ref{eq:bound}), (\ref{eq:2-variation}),
  and
  \begin{equation*} 
    \exists i\in\{1,\ldots,2^{d+m}\}:
    t_i<\infty
    \And
    \lvert\omega_i-\omega_{i-1}\rvert\ge2\epsilon
  \end{equation*}
  happen.
  (And we can make the corresponding simple capital process positive
  by being active for at most $2^{1/2}3S/\beta\epsilon^2$ values of $i$
  and setting the bet to $0$ as soon as (\ref{eq:bound}) becomes violated.)
  This corresponds to making
  at least $1$ out of $(2^{1/2}3S/\beta\epsilon^4)2^{-m+1/2}$.
  Solving the equation $(2^{1/2}3S/\beta\epsilon^4)2^{-m+1/2}=\beta/3$ in $\epsilon$
  gives $\epsilon=(2\times3^2S2^{-m}/\beta^2)^{1/4}$.
  Therefore,
  \begin{multline}\label{eq:max}
    \max_{i=1,\ldots,2^{d+m}:t_i<\infty}
    \lvert\omega_i-\omega_{i-1}\rvert
    \le
    2\epsilon
    =
    2(2\times3^2S2^{-m}/\beta^2)^{1/4}\\
    =
    2^{5/4} 3^{1/2}
    \left(
      2^{1/4} - 1
    \right)^{-1/2}
    \alpha^{-1/2}
    S^{1/4}
    2^{-m/8}
  \end{multline}
  except for a set of $\omega$ of upper price $\beta$.
  By the countable subadditivity of upper price
  (Lemma~\ref{lem:subadditivity}),
  (\ref{eq:max}) holds for all $m=1,2,\ldots$
  except for a set of $\omega$ of upper price at most
  $\sum_m\beta_m=\alpha$.

  We have now allowed $m$ to vary and so will write $t^m_i$
  instead of $t_i$ defined by (\ref{eq:t}).
  Fix an $\omega\in\Omega$ satisfying $A(\omega)\in\Omega$
  and (\ref{eq:max}) for $m=1,2,\ldots$\,.
  Intervals of the form $[t^m_{i-1}(\omega),t^m_i(\omega)]\subseteq[0,\infty)$,
  for $m\in\{1,2,\ldots\}$ and $i\in\{1,2,3,\ldots,2^{d+m}\}$,
  will be called \emph{predyadic}
  (\emph{of order $m$}).
  Given an interval $[s_1,s_2]\subseteq[0,S]$ of length at most $\delta\in(0,1)$
  and with $\tau_{s_2}<\infty$,
  we can cover $(\tau_{s_1}(\omega),\tau_{s_2}(\omega))$
  (without covering any points
  in the complement of $[\tau_{s_1}(\omega),\tau_{s_2}(\omega)]$)
  by adjacent predyadic intervals with disjoint interiors
  such that, for some $m\in\{1,2,\ldots\}$:
  there are between one and two predyadic intervals of order $m$;
  for $i=m+1,m+2,\ldots$,
  there are at most two predyadic intervals of order $i$
  (start from finding the point in $[s_1,s_2]$ of the form $j2^{-k}$
  with integer $j$ and $k$ and the smallest possible $k$, and cover
  $(\tau_{s_1}(\omega),\tau_{j2^{-k}}]$ and $[\tau_{j2^{-k}},\tau_{s_2}(\omega))$
  by predyadic intervals in the greedy manner).
  Combining (\ref{eq:max}) and
  $
    2^{-m}\le\delta
  $,
  we obtain
  \begin{align*}
    \left|
      \omega
      \left(
        \tau_{s_2}
      \right)
      -
      \omega
      \left(
        \tau_{s_1}
      \right)
    \right|
    &\le
    2^{9/4}3^{1/2}
    \left(
      2^{1/4} - 1
    \right)^{-1/2}
    \alpha^{-1/2}
    S^{1/4}\\
    &\quad\times
    \left(
      2^{-m/8}
      +
      2^{-(m+1)/8}
      +
      2^{-(m+2)/8}
      +
      \cdots
    \right)\\
    &=
    2^{9/4}3^{1/2}
    \left(
      2^{1/4} - 1
    \right)^{-1/2}
    \left(
      1 - 2^{-1/8}
    \right)^{-1}
    \alpha^{-1/2}
    S^{1/4}
    2^{-m/8}\\
    &\le
    2^{9/4}3^{1/2}
    \left(
      2^{1/4} - 1
    \right)^{-1/2}
    \left(
      1 - 2^{-1/8}
    \right)^{-1}
    \alpha^{-1/2}
    S^{1/4}
    \delta^{1/8},
  \end{align*}
  which is stronger than (\ref{eq:modulus-P})
  (as
  $
    2^{9/4}3^{1/2}
    \left(
      2^{1/4} - 1
    \right)^{-1/2}
    \left(
      1 - 2^{-1/8}
    \right)^{-1}
    \approx
    228.22
  $).
\end{proof}

Now we can prove the following elaboration of Lemma \ref{lem:modulus-P},
which will be used in the next two sections.
\begin{lemma}\label{lem:super-modulus-P}
  For each $\alpha>0$,
  \begin{multline}\label{eq:super-modulus-P}
    \LowProb
    \bigl(
      \forall S\in\{1,2,4,\ldots\}\;
      \forall\delta\in(0,1)\;
      \forall s_1,s_2\in[0,S]:\\
      \left(
        0\le s_2-s_1\le\delta
        \And
        \tau_{s_2}<\infty
      \right)\\
      \Longrightarrow
      \left|
        \omega(\tau_{s_2})
        -
        \omega(\tau_{s_1})
      \right|
      \le
      430\,
      \alpha^{-1/2}
      S^{1/2}
      \delta^{1/8}
    \bigr)
    \ge
    1-\alpha.
  \end{multline}
\end{lemma}
\begin{proof}
  Replacing $\alpha$ in (\ref{eq:modulus-P}) by
  $\alpha_S:=(1-2^{-1/2})S^{-1/2}\alpha$ for $S=1,2,4,\ldots$
  (where $1-2^{-1/2}$ is the normalizing constant
  ensuring that the $\alpha_S$ sum to $\alpha$ over $S$),
  we obtain
  \begin{multline*}
    \LowProb
    \bigl(
      \forall\delta\in(0,1)\;
      \forall s_1,s_2\in[0,S]:
      \left(
        0\le s_2-s_1\le\delta
        \And
        \tau_{s_2}<\infty
      \right)\\
      \Longrightarrow
      \left|
        \omega(\tau_{s_2})
        -
        \omega(\tau_{s_1})
      \right|
      \le
      230\,
      (1-2^{-1/2})^{-1/2}
      \alpha^{-1/2}
      S^{1/2}
      \delta^{1/8}
    \bigr)\\
    \ge
    1-(1-2^{-1/2})S^{-1/2}\alpha.
  \end{multline*}
  The countable subadditivity of upper price now gives
  \begin{multline*}
    \LowProb
    \bigl(
      \forall S\in\{1,2,4,\ldots\}\;
      \forall\delta\in(0,1)\;
      \forall s_1,s_2\in[0,S]:\\
      \left(
        0\le s_2-s_1\le\delta
        \And
        \tau_{s_2}<\infty
      \right)
      \Longrightarrow{}\\
      \left|
        \omega(\tau_{s_2})
        -
        \omega(\tau_{s_1})
      \right|
      \le
      230\,
      (1-2^{-1/2})^{-1/2}
      \alpha^{-1/2}
      S^{1/2}
      \delta^{1/8}
    \bigr)
    \ge
    1-\alpha,
  \end{multline*}
  which is stronger than (\ref{eq:super-modulus-P})
  (as
  $
    230\,
    (1-2^{-1/2})^{-1/2}
    \approx
    424.98
  $).
\end{proof}

The following lemma develops inequality (\ref{eq:2-variation})
and will be useful in the proof of Theorem \ref{thm:main-constructive}.
\begin{lemma}\label{lem:super-2-variation}
  For each $\alpha>0$,
  \begin{multline}\label{eq:super-2-variation}
    \LowProb
    \Biggl(
      \forall S\in\{1,2,4,\ldots\}\;
      \forall m\in\{1,2,\ldots\}:\\[-3mm]
      \sum_{i=1,\ldots,S2^m:t_i<\infty}
      \Bigl(
        \omega(t_i)
        -
        \omega(t_{i-1})
      \Bigr)^2
      \le
      64\,
      \alpha^{-1}
      S^2
      2^{m/16}
    \Biggr)
    \ge
    1-\alpha,
  \end{multline}
  in the notation of (\ref{eq:t}).
\end{lemma}
\begin{proof}
  Replacing $\beta/3$ in (\ref{eq:2-variation})
  with $2^{-1}(2^{1/16}-1)S^{-1}2^{-m/16}\alpha$,
  where $S$ ranges over $\{1,2,4,\ldots\}$
  and $m$ over $\{1,2,\ldots\}$,
  we obtain
  \begin{multline*}
    \LowProb
    \Biggl(
      \sum_{i=1,\ldots,S2^m:t_i<\infty}
      \Bigl(
        \omega(t_i)
        -
        \omega(t_{i-1})
      \Bigr)^2\\
      \le
      2^{3/2}(2^{1/16}-1)^{-1}
      \alpha^{-1}
      S^2
      2^{m/16}
    \Biggr)
    \ge
    1-2^{-1}(2^{1/16}-1)S^{-1}2^{-m/16}\alpha.
  \end{multline*}
  By the countable subadditivity of upper price this implies
  \begin{multline*}
    \LowProb
    \Biggl(
      \forall S\in\{1,2,4,\ldots\}\;
      \forall m\in\{1,2,\ldots\}:
      \sum_{i=1,\ldots,S2^m:t_i<\infty}
      \Bigl(
        \omega(t_i)
        -
        \omega(t_{i-1})
      \Bigr)^2\\
      \le
      2^{3/2}(2^{1/16}-1)^{-1}
      \alpha^{-1}
      S^2
      2^{m/16}
    \Biggr)
    \ge
    1-\alpha,
  \end{multline*}
  which is stronger than (\ref{eq:super-2-variation})
  (as
  $
    2^{3/2}(2^{1/16}-1)^{-1}
    \approx
    63.88
  $).
\end{proof}

The following lemma completes the proof of Theorem \ref{thm:main-constructive}(a).
\begin{lemma}\label{lem:same-intervals}
  For typical $\omega$,
  $A(\omega)$ has the same intervals of constancy as $\omega$.
\end{lemma}
\begin{proof}
  The definition of $A$ immediately implies
  that $A(\omega)$ is always constant
  on every interval of constancy of $\omega$
  (provided $A(\omega)$ exists).
  Therefore,
  we are only required to prove that
  typical $\omega$ are constant on every interval of constancy of $A(\omega)$.

  The proof can be extracted from the proof of Lemma \ref{lem:modulus-P}.
  It suffices to prove that,
  for any $\alpha>0$, $S\in\{1,2,4,\ldots\}$,
  rational $c>0$,
  and interval $[a,b]$ with rational end-points $a$ and $b$
  such that $a<b$,
  the upper price of the following event is at most $\alpha$:
  $\omega$ changes by at least $c$ over $[a,b]$,
  $A$ is constant over $[a,b]$, and $[a,b]\subseteq[0,\tau_S]$.
  Fix such $\alpha$, $S$, $c$, and $[a,b]$,
  and let $E$ stand for the event described in the previous sentence.
  Choose $m\in\{1,2,\ldots\}$ such that $2^{-m+1/2}/c^2\le\alpha/2$
  and choose the corresponding $l=l(m)$
  as in the proof of Lemma \ref{lem:modulus-P}
  but with $1-\beta/3$ replaced by $1-\alpha/2$
  (cf.\ (\ref{eq:bound})).
  The positive simple capital process
  $2^{-m+1/2}+(\omega(t)-\omega(a))^2-A^{l,a}_t$,
  started at time $a$
  and stopped when $t$ reaches $b\wedge\tau_S$,
  when $A^{l,a}_t$ reaches $2^{-m+1/2}$,
  or when $\lvert\omega(t)-\omega(a)\rvert$ reaches $c$,
  whatever happens first,
  makes $c^2$ out of $2^{-m+1/2}$
  on the conjunction of (\ref{eq:bound}) and the event $E$.
  Therefore, the upper price of the conjunction is at most $\alpha/2$,
  and the upper price of $E$ is at most $\alpha$.
\end{proof}

In view of Lemma \ref{lem:same-intervals}
we can strengthen (\ref{eq:super-modulus-P}) to
\begin{multline*}
  \LowProb
  \bigl(
    \forall S\in\{1,2,4,\ldots\}\;
    \forall\delta\in(0,1)\;
    \forall t_1,t_2\in[0,\infty):\\
    \bigl(
      \lvert A_{t_2}-A_{t_1}\rvert\le\delta
      \And
      A_{t_1}\in[0,S]
      \And
      A_{t_2}\in[0,S]
    \bigr)
    \Longrightarrow\\
    \left|
      \omega(t_2)
      -
      \omega(t_1)
    \right|
    \le
    430\,
    \alpha^{-1/2}
    S^{1/2}
    \delta^{1/8}
  \bigr)
  \ge
  1-\alpha.
\end{multline*}

\section{Proof of the remaining parts of Theorems \ref{thm:main-constructive}(b)
and~\ref{thm:supermain-constructive}}
\label{sec:proof-constructive-b}

Let $c\in\bbbr$ be a fixed constant.
Results of the previous section imply
the tightness of $\LowQ_c$
(for details, see below).
\begin{lemma}\label{lem:tight}
  For each $\alpha>0$ there exists a compact set $\mathfrak{K}\subseteq\Omega$
  such that $\LowQ_c(\mathfrak{K})\ge1-\alpha$.
\end{lemma}

In particular, Lemma~\ref{lem:tight} asserts that $\LowQ_c(\Omega)=1$.
This fact and the results of Section~\ref{sec:coherence} allow us
to check that Theorem~\ref{thm:supermain-constructive}
implies Theorem~\ref{thm:main-constructive}(b).
First, the inequality $\le$ in (\ref{eq:inequality}) implies
\begin{multline*}
  \UpQ_c(E)
  =
  \UpProb(\ntt^{-1}(E);\omega(0)=c,A_{\infty}=\infty)\\
  =
  \UpExpect(\III_E\circ\ntt;\omega(0)=c,A_{\infty}=\infty)
  \le
  \int_{\Omega}
  \III_E
  \dd\Wiener_c
  =
  \Wiener_c(E)
\end{multline*}
for all $E\in\FFF$.
Therefore,
\begin{align}
  \LowQ_c(E)
  &=
  \LowProb(\ntt^{-1}(E);\omega(0)=c,A_{\infty}=\infty)
  \notag\\
  &=
  1-\UpProb
  \left(
    \ntt^{-1}(E^c)\cup\left(\ntt^{-1}(\Omega)\right)^c;
    \omega(0)=c,A_{\infty}=\infty
  \right)
  \notag\\
  &=
  1-\UpProb(\ntt^{-1}(E^c);\omega(0)=c,A_{\infty}=\infty)
  \label{eq:chain-2}\\
  &\ge
  1-\Wiener_c(E^c)
  =
  \Wiener_c(E)
  \notag
\end{align}
and so, by Lemma \ref{lem:lower-upper} and (\ref{eq:one}),
\begin{equation*}
  \UpQ_c(E)=\LowQ_c(E)=\Wiener_c(E)
\end{equation*}
for all $E\in\FFF$.
The equality in line (\ref{eq:chain-2})
follows from $\LowProb(\ntt^{-1}(\Omega);\omega(0)=c,A_{\infty}=\infty)=1$,
which in turn follows from (and is in fact equivalent to) $\LowQ_c(\Omega)=1$.
Therefore, we only need to finish
the proof of  Theorem~\ref{thm:supermain-constructive}.

More precise results than Lemma~\ref{lem:tight} can be stated in terms
of the \emph{modulus of continuity}
of a function $\psi\in\bbbr^{[0,\infty)}$
on an interval $[0,S]\subseteq[0,\infty)$:
\begin{equation*}
  \m^S_{\delta}(\psi)
  :=
  \sup_{s_1,s_2\in[0,S]:\lvert s_1-s_2\rvert\le\delta}
  \lvert\psi(s_1)-\psi(s_2)\rvert,
  \quad
  \delta>0;
\end{equation*}
it is clear that $\lim_{\delta\to0}\m^S_{\delta}(\psi)=0$
if and only if $\psi$ is continuous (equivalently, uniformly continuous) on $[0,S]$.
\begin{lemma}\label{lem:super-modulus}
  For each $\alpha>0$,
  \begin{equation*} 
    \LowQ_c
    \left(
      \forall S\in\{1,2,4,\ldots\}\;
      \forall\delta\in(0,1):
      \m^S_{\delta}
      \le
      430\,
      \alpha^{-1/2}
      S^{1/2}
      \delta^{1/8}
    \right)
    \ge
    1-\alpha.
  \end{equation*}
\end{lemma}
\noindent
Lemma~\ref{lem:super-modulus} immediately follows from Lemma~\ref{lem:super-modulus-P},
and Lemma \ref{lem:tight} immediately follows
from Lemma \ref{lem:super-modulus}
and the Arzel\`a--Ascoli theorem
(as stated in \cite{Karatzas/Shreve:1991}, Theorem~2.4.9).

We start the proof of the remaining part of Theorem~\ref{thm:supermain-constructive}
from a series of reductions.
To establish the inequality $\le$ in (\ref{eq:inequality})
we only need to establish
$
  \UpExpect(F\circ\ntt;\omega(0)=c,A_{\infty}=\infty)
  <
  \int F\dd\Wiener_c+\epsilon
$
for each positive constant $\epsilon$.
\begin{enumerate}
\item\label{it:c}
  We can assume that $F$ in (\ref{eq:inequality})
  is lower semicontinuous on $\Omega$.
  Indeed, if it is not,
  by the Vitali--Carath\'eodory theorem
  (see, e.g., \cite{Rudin:1987short}, Theorem 2.25)
  for any compact $\mathfrak{K}\subseteq\Omega$
  (assumed non-empty)
  there exists a lower semicontinuous function $G$ on $\mathfrak{K}$
  such that $G\ge F$ on $\mathfrak{K}$
  and $\int_\mathfrak{K} G \dd\Wiener_c \le \int_\mathfrak{K} F \dd\Wiener_c + \epsilon$.
  Without loss of generality we assume $\sup G\le\sup F$,
  and we extend $G$ to all of $\Omega$
  by setting $G:=\sup F$ outside $\mathfrak{K}$.
  Choosing $\mathfrak{K}$ with large enough $\Wiener_c(\mathfrak{K})$
  (which can be done since the probability measure $\Wiener_c$ is tight:
  see, e.g., \cite{Billingsley:1968}, Theorem 1.4),
  we will have $G\ge F$
  and $\int G \dd\Wiener_c \le \int F \dd\Wiener_c + 2\epsilon$.
  Achieving $\mathfrak{S}_0\le\int G\dd\Wiener_c+\epsilon$
  and $\liminf_{t\to\infty}\mathfrak{S}_t(\omega)\ge(G\circ\ntt)(\omega)$,
  where $\mathfrak{S}$ is a positive capital process,
  will automatically achieve $\mathfrak{S}_0\le\int F\dd\Wiener_c + 3\epsilon$
  and $\liminf_{t\to\infty}\mathfrak{S}_t(\omega)\ge(F\circ\ntt)(\omega)$.
\item\label{it:d}
  We can further assume that $F$ is continuous on $\Omega$.
  Indeed, since each lower semicontinuous function on a metric space
  is the limit of an increasing sequence of continuous functions
  (see, e.g., \cite{Engelking:1989}, Problem 1.7.15(c)),
  given a lower semicontinuous positive function $F$ on $\Omega$
  we can find a series of positive continuous functions $G^n$ on $\Omega$,
  $n=1,2,\ldots$,
  such that $\sum_{n=1}^{\infty}G^n=F$.
  The sum $\mathfrak{S}$ of positive capital processes
  $\mathfrak{S}^1,\mathfrak{S}^2,\ldots$
  achieving $\mathfrak{S}^n_0\le\int G^n\dd\Wiener_c+2^{-n}\epsilon$
  and $\liminf_{t\to\infty}\mathfrak{S}^n_t(\omega)\ge(G^n\circ\ntt)(\omega)$,
  $n=1,2,\ldots$,
  will achieve $\mathfrak{S}_0\le\int F\dd\Wiener_c + \epsilon$
  and $\liminf_{t\to\infty}\mathfrak{S}_t(\omega)\ge(F\circ\ntt)(\omega)$.
\item\label{it:e}
  We can further assume that $F$ depends on $\psi\in\Omega$
  only via $\psi|_{[0,S]}$ for some $S\in(0,\infty)$.
  Indeed, let us fix $\epsilon>0$ and prove
  $\UpExpect(F\circ\ntt;\omega(0)=c,A_{\infty}=\infty)\le\int F\dd\Wiener_c + C\epsilon$
  for some positive constant $C$
  assuming $\UpExpect(G\circ\ntt;\omega(0)=c,A_{\infty}=\infty)\le\int G\dd\Wiener_c$
  for all continuous positive $G$ that depend on $\psi\in\Omega$
  only via $\psi|_{[0,S]}$ for some $S\in(0,\infty)$.
  Choose a compact set $\mathfrak{K}\subseteq\Omega$
  with $\Wiener_c(\mathfrak{K})>1-\epsilon$ and $\LowQ_c(\mathfrak{K})>1-\epsilon$
  (cf.\ Lemma~\ref{lem:tight}).
  Set
  $
    F^S(\psi)
    :=
    F(\psi^S)
  $,
  where $\psi^S$ is defined by $\psi^S(s):=\psi(s\wedge S)$
  and $S$ is sufficiently large in the following sense.
  Since $F$ is uniformly continuous on $\mathfrak{K}$
  and the metric is defined by (\ref{eq:metric}),
  $F$ and $F^S$ can be made arbitrarily close in $C(\mathfrak{K})$;
  in particular, let $\|F-F^S\|_{C(\mathfrak{K})}<\epsilon$.
  Choose positive capital processes $\mathfrak{S}^0$ and $\mathfrak{S}^1$ such that
  \begin{align*}
    \mathfrak{S}^0_0 &\le \int F^S\dd\Wiener_c+\epsilon,&
    \liminf_{t\to\infty}\mathfrak{S}^0_t(\omega) &\ge (F^S\circ\ntt)(\omega),\\
    \mathfrak{S}^1_0 &\le \epsilon,&
    \liminf_{t\to\infty}\mathfrak{S}^1_t(\omega) &\ge (\III_{\mathfrak{K}^c}\circ\ntt)(\omega),
  \end{align*}
  for all $\omega\in\Omega$ satisfying $\omega(0)=c$ and $A_{\infty}(\omega)=\infty$.
  The sum $\mathfrak{S}:=\mathfrak{S}^0+(\sup F)\mathfrak{S}^1+\epsilon$
  will satisfy
  \begin{align*}
    \mathfrak{S}_0
    &\le
    \int F^S\dd\Wiener_c
    +
    (\sup F + 2)\epsilon
    \le
    \int_\mathfrak{K} F^S\dd\Wiener_c
    +
    (2\sup F + 2)\epsilon\\
    &\le
    \int_\mathfrak{K} F\dd\Wiener_c
    +
    (2\sup F + 3)\epsilon
    \le
    \int F\dd\Wiener_c
    +
    (2\sup F + 3)\epsilon
  \end{align*}
  and
  \begin{equation*}
    \liminf_{t\to\infty}\mathfrak{S}_t(\omega)
    \ge
    (F^S\circ\ntt)(\omega)
    +
    (\sup F)
    (\III_{\mathfrak{K}^c}\circ\ntt)(\omega)
    +
    \epsilon
    \ge
    (F\circ\ntt)(\omega),
  \end{equation*}
  provided $\omega(0)=c$ and $A_{\infty}(\omega)=\infty$.
  We assume $S\in\{1,2,4,\ldots\}$,
  without loss of generality.
\item\label{it:f}
  We can further assume that $F(\psi)$ depends on $\psi\in\Omega$
  only via the values $\psi(iS/N)$,
  $i=1,\ldots,N$ (remember that we are interested in the case $\psi(0)=c$),
  for some $N\in\{1,2,\ldots\}$.
  Indeed, let us fix $\epsilon>0$ and prove
  $\UpExpect(F\circ\ntt;\omega(0)=c,A_{\infty}=\infty)\le\int F\dd\Wiener_c + C\epsilon$
  for some positive constant $C$
  assuming $\UpExpect(G\circ\ntt;\omega(0)=c,A_{\infty}=\infty)\le\int G\dd\Wiener_c$
  for all continuous positive $G$ that depend on $\psi\in\Omega$
  only via $\psi(iS/N)$, $i=1,\ldots,N$, for some $N$.
  Let $\mathfrak{K}\subseteq\Omega$ be the compact set in $\Omega$
  defined as
  $
    \mathfrak{K}
    :=
    \left\{
      \psi\in\Omega
      \st
      \psi(0)=c
      \And
      \forall\delta>0:
      m^S_{\delta}(\psi)
      \le
      f(\delta)
    \right\}
  $
  for some $f:(0,\infty)\to(0,\infty)$ satisfying
  $\lim_{\delta\to0}f(\delta)=0$
  (cf.\ the Arzel\`a--Ascoli theorem)
  and chosen in such a way that
  $\Wiener_c(\mathfrak{K})>1-\epsilon$
  and $\LowQ_c(\mathfrak{K})>1-\epsilon$.
  Let $g$ be the modulus of continuity of $F$ on $\mathfrak{K}$,
  $
    g(\delta)
    :=
    \sup_{\psi_1,\psi_2\in\mathfrak{K}:\rho(\psi_1,\psi_2)\le\delta}
    \lvert F(\psi_1)-F(\psi_2)\rvert
  $;
  we know that $\lim_{\delta\to0}g(\delta)=0$.
  Set
  $
    F_N(\psi)
    :=
    F(\psi_N)
  $,
  where $\psi_N$ is the piecewise linear function
  whose graph is obtained
  by joining the points $(iS/N,\psi(iS/N))$, $i=0,1,\ldots,N$,
  and $(\infty,\psi(S))$,
  and $N$ is so large that $g(f(S/N))\le\epsilon$.
  Since
  \begin{equation*}
    \psi\in\mathfrak{K}
    \enspace\Longrightarrow\enspace
    \left\|\psi-\psi_N\right\|_{C[0,S]}
    \le
    f(S/N)
    \enspace\Longrightarrow\enspace
    \rho(\psi,\psi_N)
    \le
    f(S/N)
  \end{equation*}
  (we assume, without loss of generality,
  that the graph of $\psi$ is horizontal over $[S,\infty)$),
  we have $\|F-F_N\|_{C(\mathfrak{K})}\le\epsilon$.
  Choose positive capital processes $\mathfrak{S}^0$ and $\mathfrak{S}^1$ such that
  \begin{align*}
    \mathfrak{S}^0_0 &\le \int F_N\dd\Wiener_c+\epsilon,&
    \liminf_{t\to\infty}\mathfrak{S}^0_t(\omega) &\ge (F_N\circ\ntt)(\omega),\\
    \mathfrak{S}^1_0 &\le \epsilon,&
    \liminf_{t\to\infty}\mathfrak{S}^1_t(\omega) &\ge (\III_{\mathfrak{K}^c}\circ\ntt)(\omega),
  \end{align*}
  provided $\omega(0)=c$ and $A_{\infty}(\omega)=\infty$.
  The sum $\mathfrak{S}:=\mathfrak{S}^0+(\sup F)\mathfrak{S}^1+\epsilon$ will satisfy
  \begin{align*}
    \mathfrak{S}_0
    &\le
    \int F_N\dd\Wiener_c
    +
    (\sup F+2)\epsilon
    \le
    \int_\mathfrak{K} F_N\dd\Wiener_c
    +
    (2\sup F+2)\epsilon\\
    &\le
    \int_\mathfrak{K} F\dd\Wiener_c
    +
    (2\sup F+3)\epsilon
    \le
    \int F\dd\Wiener_c
    +
    (2\sup F+3)\epsilon
  \end{align*}
  and
  \begin{equation*}
    \liminf_{t\to\infty}\mathfrak{S}_t(\omega)
    \ge
    (F_N\circ\ntt)(\omega)
    +
    (\sup F)
    (\III_{\mathfrak{K}^c}\circ\ntt)(\omega)
    +
    \epsilon
    \ge
    (F\circ\ntt)(\omega),
  \end{equation*}
  provided $\omega(0)=c$ and $A_{\infty}(\omega)=\infty$.
\item\label{it:g}
  We can further assume that
  \begin{equation}\label{eq:F}
    F(\psi)
    =
    U
    \left(
      \psi(S/N),
      \psi(2S/N),
      \ldots,
      \psi(S)
    \right)
  \end{equation}
  where the function $U:\bbbr^{N}\to[0,\infty)$
  is not only continuous but also has compact support.
  (We will sometimes say that $U$ is the \emph{generator} of $F$.)
  Indeed, let us fix $\epsilon>0$ and prove
  $\UpExpect(F\circ\ntt;\omega(0)=c,A_{\infty}=\infty)\le\int F\dd\Wiener_c + C\epsilon$
  for some positive constant $C$
  assuming $\UpExpect(G\circ\ntt;\omega(0)=c,A_{\infty}=\infty)\le\int G\dd\Wiener_c$
  for all $G$ whose generator has compact support.
  Let $B_R$ be the open ball of radius $R$ and centred at the origin
  in the space $\bbbr^{N}$ with the $\ell_{\infty}$ norm.
  We can rewrite (\ref{eq:F})
  as $F(\psi)=U(\sigma(\psi))$
  where $\sigma:\Omega\to\bbbr^N$ reduces each $\psi\in\Omega$
  to
  $
    \sigma(\psi)
    :=
    \left(
      \psi(S/N),
      \psi(2S/N),
      \ldots,
      \psi(S)
    \right)
  $.
  Choose $R>0$ so large that
  $\Wiener_c(\sigma^{-1}(B_R))>1-\epsilon$
  and $\LowQ_c(\sigma^{-1}(B_R))>1-\epsilon$
  (the existence of such $R$ follows
  from the Arzel\`a--Ascoli theorem and Lemma~\ref{lem:tight}).
  Alongside $F$, whose generator is denoted $U$,
  we will also consider $F^*$ with generator
  \begin{equation*}
    U^*(z)
    :=
    \begin{cases}
      U(z) & \text{if $z\in\overline{B_R}$}\\
      0 & \text{if $z\in B^c_{2R}$}
    \end{cases}
  \end{equation*}
  (where $\overline{B_R}$ is the closure of $B_R$ in $\bbbr^N$);
  in the remaining region $B_{2R}\setminus\overline{B_R}$,
  $U^*$ is defined arbitrarily
  (but making sure that $U^*$ is continuous
  and takes values in $[0,\sup U]$;
  this can be done by the Tietze--Urysohn theorem,
  \cite{Engelking:1989}, Theorem 2.1.8).
  Choose positive capital processes $\mathfrak{S}^0$ and $\mathfrak{S}^1$ such that
  \begin{align*}
    \mathfrak{S}^0_0 &\le \int F^*\dd\Wiener_c+\epsilon,&
    \liminf_{t\to\infty}\mathfrak{S}^0_t(\omega)
      &\ge (F^*\circ\ntt)(\omega),\\
    \mathfrak{S}^1_0 &\le \epsilon,&
    \liminf_{t\to\infty}\mathfrak{S}^1_t(\omega)
      &\ge (\III_{(\sigma^{-1}(B_R))^c}\circ\ntt)(\omega),
  \end{align*}
  provided $\omega(0)=c$ and $A_{\infty}(\omega)=\infty$.
  The sum $\mathfrak{S}:=\mathfrak{S}^0+(\sup F)\mathfrak{S}^1$ will satisfy
  \begin{align*}
    \mathfrak{S}_0
    &\le
    \int F^*\dd\Wiener_c
    +
    (\sup F+1)\epsilon
    \le
    \int_{\sigma^{-1}(B_R)} F^*\dd\Wiener_c
    +
    (2\sup F+1)\epsilon\\
    &=
    \int_{\sigma^{-1}(B_R)} F\dd\Wiener_c
    +
    (2\sup F+1)\epsilon
    \le
    \int F\dd\Wiener_c
    +
    (2\sup F+1)\epsilon
  \end{align*}
  and
  \begin{multline*}
    \liminf_{t\to\infty}
    \mathfrak{S}_t(\omega)
    \ge
    (F^*\circ\ntt)(\omega)
    +
    (\sup F)
    (\III_{(\sigma^{-1}(B_R))^c}\circ\ntt)(\omega)\\
    \ge
    (F\circ\ntt)(\omega),
  \end{multline*}
  provided $\omega(0)=c$ and $A_{\infty}(\omega)=\infty$.
\item\label{it:h}
  Since every continuous $U:\bbbr^{N}\to[0,\infty)$
  with compact support
  can be arbitrarily well approximated in $C(\bbbr^{N})$
  by an infinitely differentiable (positive) function with compact support
  (see, e.g., \cite{Adams/Fournier:2003short}, Theorem 2.29),
  we can further assume that the generator $U$ of $F$
  is an infinitely differentiable function with compact support.
\item\label{it:i}
  By Lemma~\ref{lem:tight},
  it suffices to prove that,
  given $\epsilon>0$ and a compact set $\mathfrak{K}$ in $\Omega$,
  some positive capital process $\mathfrak{S}$
  with $\mathfrak{S}_0\le\int F \dd\Wiener_c + \epsilon$ achieves
  $\liminf_{t\to\infty}\mathfrak{S}_t(\omega)\ge(F\circ\ntt)(\omega)$
  for all $\omega\in\ntt^{-1}(\mathfrak{K})$
  such that $\omega(0)=c$ and $A_{\infty}(\omega)=\infty$.
  Indeed, we can choose $\mathfrak{K}$
  with $\LowQ_c(\mathfrak{K})$ so close to $1$
  that the sum of $\mathfrak{S}$
  and a positive capital process eventually attaining
  $\sup F$ on $(\ntt^{-1}(\mathfrak{K}))^c$
  will give a positive capital process
  starting from at most $\int F\dd\Wiener_c + 2\epsilon$
  and attaining $(F\circ\ntt)(\omega)$ in the limit,
  provided $\omega(0)=c$ and $A_{\infty}(\omega)=\infty$.
\end{enumerate}
From now on we fix a compact $\mathfrak{K}\subseteq\Omega$,
assuming, without loss of generality,
that the statements inside the outer parentheses
in (\ref{eq:super-modulus-P}) and (\ref{eq:super-2-variation})
are satisfied for some $\alpha>0$ when $\ntt(\omega)\in\mathfrak{K}$.

In the rest of the proof we will be using,
often following \cite{Shafer/Vovk:2001}, Section 6.2,
the standard method going back to Lindeberg \cite{Lindeberg:1922}.
For $i=N-1$,
define a function
$\overline{U}_i:\bbbr\times[0,\infty)\times\bbbr^i\to\bbbr$
by
\begin{equation}\label{eq:overline-U}
  \overline{U}_i(x,D;x_1,\ldots,x_i)
  :=
  \int_{-\infty}^{\infty} U_{i+1}(x_1,\ldots,x_i,x+z) \Normal_{0,D}(dz),
\end{equation}
where $U_N$ stands for $U$
and $\Normal_{0,D}$ is the Gaussian probability measure on $\bbbr$
with mean $0$ and variance $D\ge0$.
Next define, for $i=N-1$,
\begin{equation}\label{eq:U}
  U_i(x_1,\ldots,x_i)
  :=
  \overline{U}_i(x_i,S/N;x_1,\ldots,x_i).
\end{equation}
Finally, we can alternately use (\ref{eq:overline-U}) and (\ref{eq:U})
for $i=N-2,\ldots,1,0$ to define inductively
other $\overline{U}_i$ and $U_i$
(with (\ref{eq:U}) interpreted as
$
  U_0
  :=
  \overline{U}_0(c,S/N)
$
when $i=0$).
Notice that $U_0=\int F\dd\Wiener_c$.

Informally,
the functions (\ref{eq:overline-U}) and (\ref{eq:U})
constitute Sceptic's goal:
assuming $\ntt(\omega)\in\mathfrak{K}$,
$\omega(0)=c$, and $A_{\infty}(\omega)=\infty$,
he will keep his capital at time $\tau_{iS/N}$, $i=0,1,\ldots,N$, close to
$U_i(\omega(\tau_{S/N}),\omega(\tau_{2S/N}),\ldots,\omega(\tau_{iS/N}))$
and his capital at any other time $t\in[0,\tau_S]$
close to
$
  \overline{U}_i(\omega(t),D;
  \omega(\tau_{S/N}),\omega(\tau_{2S/N}),\ldots,\omega(\tau_{iS/N}))
$
where $i:=\lfloor NA_t/S\rfloor$ and $D:=(i+1)S/N-A_t$.
This will ensure that his capital at time $\tau_S$
is close to or exceeds $(F\circ\ntt)(\omega)$
when his initial capital is $U_0=\int F\dd\Wiener_c$,
$\omega(0)=c$,
and $A_{\infty}(\omega)=\infty$.

The proof is based on the fact that each function
$\overline{U}_i(x,D;x_1,\ldots,x_i)$
satisfies the heat equation in the variables $x$ and $D$:
\begin{equation}\label{eq:heat}
  \frac{\partial\overline{U}_i}{\partial D}
  (x,D;x_1,\ldots,x_i)
  =
  \frac12\frac{\partial^2\overline{U}_i}{\partial x^2}
  (x,D;x_1,\ldots,x_i)
\end{equation}
for all $x\in\bbbr$, all $D>0$,
and all $x_1,\ldots,x_i\in\bbbr$.
This can be checked by direct differentiation.

Sceptic will only bet at the times of the form $\tau_{kS/LN}$,
where $L\in\{1,2,\ldots\}$ is a constant that will later be chosen large
and $k$ is integer.
For $i=0,\ldots,N$ and $j=0,\ldots,L$ let us set
\begin{equation*}
  t_{i,j}
  :=
  \tau_{iS/N+jS/LN},
  \quad
  X_{i,j}
  :=
  \omega(t_{i,j}),
  \quad
  D_{i,j}
  :=
  S/N-jS/LN.
\end{equation*}
For any array $Y_{i,j}$,
we set $dY_{i,j} := Y_{i,j+1} - Y_{i,j}$.

Using Taylor's formula
and omitting the arguments $\omega(\tau_{S/N}),\ldots,\omega(\tau_{iS/N})$,
we obtain, for $i=0,\ldots,N-1$ and $j=0,\ldots,L-1$,
\begin{multline}\label{eq:3}
  d \overline{U}_i(X_{i,j},D_{i,j})
  =
  \frac{\partial\overline{U}_i}{\partial x} (X_{i,j},D_{i,j}) dX_{i,j}
  +
  \frac{\partial\overline{U}_i}{\partial D} (X_{i,j},D_{i,j}) dD_{i,j}
\\
  +
  \frac12 \frac{\partial^2\overline{U}_i}{\partial x^2} (X_{i,j}',D_{i,j}')
  (dX_{i,j})^2
  +
  \frac{\partial^2\overline{U}_i}{\partial x \partial D} (X_{i,j}',D_{i,j}')
  dX_{i,j} dD_{i,j}
\\
  +
  \frac12 \frac{\partial^2\overline{U}_i}{\partial D^2} (X_{i,j}',D_{i,j}')
  (dD_{i,j})^2,
\end{multline}
where $(X_{i,j}',D_{i,j}')$ is a point
strictly between $(X_{i,j},D_{i,j})$ and $(X_{i,j+1},D_{i,j+1})$.
Applying Taylor's formula to
$\partial^2\overline{U}_i/\partial x^2$,
we find
\begin{multline*}
  \frac{\partial^2\overline{U}_i}{\partial x^2} (X_{i,j}',D_{i,j}')
  =
  \frac{\partial^2\overline{U}_i}{\partial x^2} (X_{i,j},D_{i,j})\\
  +
  \frac{\partial^3\overline{U}_i}{\partial x^3}
  (X_{i,j}'',D_{i,j}'') \Delta X_{i,j}
  +
  \frac{\partial^3\overline{U_i}}{\partial D \partial x^2}
  (X_{i,j}'',D_{i,j}'') \Delta D_{i,j},
\end{multline*}
where $(X_{i,j}'',D_{i,j}'')$ is a point strictly between
$(X_{i,j},D_{i,j})$ and $(X_{i,j}',D_{i,j}')$,
and $\Delta X_{i,j}$ and $\Delta D_{i,j}$ satisfy
$\lvert\Delta X_{i,j}\rvert \le \lvert dX_{i,j}\rvert$,
$\lvert\Delta D_{i,j}\rvert \le \lvert dD_{i,j}\rvert$.
Plugging this equation and the heat equation (\ref{eq:heat})
into~(\ref{eq:3}), we obtain
\begin{multline}\label{eq:5}
  d\overline{U}_i(X_{i,j},D_{i,j})
  =
  \frac{\partial\overline{U}_i}{\partial x} (X_{i,j},D_{i,j}) dX_{i,j}
  +
  \frac12
  \frac{\partial^2\overline{U}_i}{\partial x^2} (X_{i,j},D_{i,j})
  \left(
    (dX_{i,j})^2 + dD_{i,j}
  \right)
\\
  +
  \frac12
  \frac{\partial^3\overline{U}_i}{\partial x^3}
  (X_{i,j}'',D_{i,j}'')
  \Delta X_{i,j} (dX_{i,j})^2
  +
  \frac12
  \frac{\partial^3\overline{U}_i}{\partial D \partial x^2}
  (X_{i,j}'',D_{i,j}'')
  \Delta D_{i,j} (dX_{i,j})^2
\\
  +
  \frac
    {\partial^2\overline{U}}
    {\partial x \partial D}
  (X_{i,j}',D_{i,j}')
  dX_{i,j} dD_{i,j}
  +
  \frac12
  \frac{\partial^2\overline{U}}{\partial D^2} (X_{i,j}',D_{i,j}')
  (dD_{i,j})^2.
\end{multline}

To show that Sceptic can achieve his goal,
we will describe a simple trading strategy
that results in increase of his capital of approximately (\ref{eq:5})
during the time interval $[t_{i,j},t_{i,j+1}]$
(we will make sure that the cumulative error of our approximation
is small with high probability,
which will imply the statement of the theorem).
We will see that there is a trading strategy
resulting in the capital increase equal to the first addend
on the right-hand side of (\ref{eq:5}),
that there is another trading strategy
resulting in the capital increase approximately equal to the second addend,
and that the last four addends are negligible.
The sum of the two trading strategies will achieve our goal.

The trading strategy
whose capital increase over $[t_{i,j},t_{i,j+1}]$ is the first addend is obvious:
it bets $\partial\overline{U}_i/\partial x$ at time $t_{i,j}$.
The bet is bounded as average of $\partial U_{i+1}/\partial x_{i+1}$,
the boundedness of which can be seen from the recursive formula
\begin{multline*}
  U_k(x_1,\ldots,x_k)
  =
  \int_{-\infty}^{\infty}
  U_{k+1}(x_1,\ldots,x_k,x_k+z)
  \Normal_{0,S/N}(dz),\\
  k=i+1,\ldots,N-1,
\end{multline*}
and $U_N=U$ being an infinitely differentiable function with compact support.

The second addend involves the expression
$(dX_{i,j})^2 + dD_{i,j} = (\omega_{i,j+1}-\omega_{i,j})^2-S/LN$.
To analyze it, we will need the following lemma.
\begin{lemma}\label{lem:bound}
  For all $\delta>0$ and $\beta>0$,
  there exists a positive integer $l$ such that
  \begin{equation*}
    t_{i,j+1}<\infty
    \Longrightarrow
    \left|
      \frac{A^{l,t_{i,j}}_{t_{i,j+1}}}{S/LN}
      -
      1
    \right|
    <
    \delta
    \quad
  \end{equation*}
  holds for all $i=0,\ldots,N-1$ and $j=0,\ldots,L-1$
  except for a set of $\omega$ of upper price at most $\beta$.
\end{lemma}

Lemma \ref{lem:bound} can be proved similarly to (\ref{eq:bound}).
(The inequality in (\ref{eq:bound}) is one-sided,
so it was sufficient to use only (\ref{eq:increasing});
for Lemma \ref{lem:bound} both (\ref{eq:increasing}) and (\ref{eq:decreasing})
should be used.)

We know that $(\omega(t)-\omega(t_{i,j}))^2-A^{l,t_{i,j}}_t$
is a simple capital process
(see the proof of Lemma \ref{lem:modulus-P}).
Therefore,
there is indeed a simple trading strategy
resulting in capital increase approximately equal
to the second addend on the right-hand side of (\ref{eq:5}),
with the cumulative approximation error
that can be made arbitrarily small
on a set of $\omega$ of lower price arbitrarily close to $1$.
(Analogously to the analysis of the first addend,
$\partial^2\overline{U}_i/\partial x^2$
is bounded as average of $\partial^2 U_{i+1}/\partial x_{i+1}^2$.)

Let us show that the last four terms
on the right-hand side of (\ref{eq:5})
are negligible when $L$ is sufficiently large
(assuming $S$, $N$, and $U$ fixed).
All the partial derivatives involved in those terms are bounded:
the heat equation implies
\begin{align*}
  \frac{\partial^3\overline{U}_i}{\partial D \partial x^2}
  &=
  \frac{\partial^3\overline{U}_i}{\partial x^2 \partial D}
  =
  \frac12
  \frac{\partial^4\overline{U}_i}{\partial x^4},
\\
  \frac{\partial^2\overline{U}_i}{\partial x \partial D}
  &=
  \frac12
  \frac{\partial^3\overline{U}_i}{\partial x^3},
\\
  \frac{\partial^2\overline{U}_i}{\partial D^2}
  &=
  \frac12
  \frac{\partial^3\overline{U}_i}{\partial D \partial x^2}
  =
  \frac14
  \frac{\partial^4\overline{U}_i}{\partial x^4},
\end{align*}
and $\partial^3\overline{U}_i/\partial x^3$
and $\partial^4\overline{U}_i/\partial x^4$,
being averages of $\partial^3 U_{i+1}/\partial x_{i+1}^3$
and $\partial^4 U_{i+1}/\partial x_{i+1}^4$,
respectively,
are bounded.
We can assume that
\begin{equation*}
  \lvert dX_{i,j}\rvert
  \le
  C_1 L^{-1/8},
  \quad
  \sum_{i=0}^{N-1}
  \sum_{j=0}^{L-1}
  (dX_{i,j})^2
  \le
  C_2 L^{1/16}
\end{equation*}
(cf.\ (\ref{eq:super-modulus-P}) and (\ref{eq:super-2-variation}), respectively)
for $\ntt(\omega)\in\mathfrak{K}$
and some constants $C_1$ and $C_2$
(remember that $S$, $N$, $U$, and, of course, $\alpha$ are fixed;
without loss of generality we can assume that $N$ and $L$ are powers of $2$).
This makes the cumulative contribution of the four terms
have at most the order of magnitude $O(L^{-1/16})$;
therefore,
Sceptic can achieve his goal for $\ntt(\omega)\in\mathfrak{K}$
by making $L$ sufficiently large.

To ensure that his capital is always positive,
Sceptic stops playing as soon as his capital hits $0$.
Increasing his initial capital by a small amount
we can make sure that this will never happen when $\ntt(\omega)\in\mathfrak{K}$
(for $L$ sufficiently large).

\section{Proof of the inequality $\le$ in Theorem~\ref{thm:supermain}}
\label{sec:proof-main-le}

Fix a bounded positive $\KKK$-measurable functional $F$.
Let $a:=\int F \dd\Wiener_c$;
our goal is to show that $\UpExpect(F;\omega(0)=c)\le a$.
Define $\Omega'$ to be the set of all $\omega\in\Omega$
such that $\omega(0)=c$ and $\forall t\in[0,\infty):\UpA_t(\omega)=\LowA_t(\omega)=t$.
We know (Lemma~\ref{lem:A-BM}) that $\Wiener_c(\Omega')=1$.
It is clear that $\tau_s(\omega)=s$ for all $\omega\in\Omega'$,
and so $\ntt(\omega)=\omega$ for all $\omega\in\Omega'$.
By Theorem~\ref{thm:supermain-constructive},
\begin{equation*}
  \UpExpect(F\III_{\Omega'})
  =
  \UpExpect(F;\Omega')
  =
  \UpExpect(F\circ\ntt;\Omega')
  \le
  \UpExpect(F\circ\ntt;\omega(0)=c,A_{\infty}=\infty)
  =
  a
\end{equation*}
(we will not need the opposite inequality in that theorem).
Therefore, for any $\epsilon>0$
there exists a positive capital process $\mathfrak{S}$
such that $\mathfrak{S}_0\le a+\epsilon$
and $\liminf_{t\to\infty}\mathfrak{S}_t\ge F\III_{\Omega'}$.
We assume, without loss of generality, that $\mathfrak{S}$ is bounded.
Moreover, the proof of Theorem~\ref{thm:supermain-constructive}
shows that $\mathfrak{S}$ can be chosen \emph{time-invariant},
in the sense that
$\mathfrak{S}_{f(t)}(\omega)=\mathfrak{S}_{t}(\omega\circ f)$
for all time transformations $f$ and all $t\in[0,\infty)$.
This property will also be assumed to be satisfied
until the end of this section.
In conjunction with the time-superinvariance of $F$
(which is equivalent to (\ref{eq:invariant-function}))
and the last statement of Theorem~\ref{thm:main-constructive}(a),
it implies,
for typical $\omega\in\Omega$ satisfying $\omega(0)=c$ and $A_\infty(\omega)=\infty$,
\begin{multline}\label{eq:infinite-A}
  \liminf_{t\to\infty}
  \mathfrak{S}_t(\omega)
  =
  \liminf_{t\to\infty}
  \mathfrak{S}_t(\psi^f)
  =
  \liminf_{t\to\infty}
  \mathfrak{S}_{f(t)}(\psi)\\
  \ge
  (F\III_{\Omega'})(\psi)
  =
  F(\psi)
  \ge
  F(\omega),
\end{multline}
where $\psi$ is any element of $\Omega'$ that satisfies $\psi^f=\omega$
for some time transformation $f$,
necessarily satisfying $\lim_{t\to\infty}f(t)=\infty$
(we can always take $\psi:=\ntt(\omega)$ and $f:=A(\omega)$;
$\omega=\ntt(\omega)\circ A(\omega)$ follows from $\omega(t)=\omega(\tau_{A_t(\omega)})$).
It is easy to modify $\mathfrak{S}$
so that $\mathfrak{S}_0$ is increased by at most $\epsilon$
and the inequality between the two extreme terms in (\ref{eq:infinite-A})
becomes true for all, rather than for typical,
$\omega\in\Omega$ satisfying $\omega(0)=c$ and $A_\infty(\omega)=\infty$.

Let us now consider $\omega\in\Omega$
such that $\omega(0)=c$ but $A_\infty(\omega)=\infty$ is not satisfied.
Without loss of generality we assume
that $A(\omega)$ exists and is an element of $\Omega$
with the same intervals of constancy as $\omega$
and that the statement in the outermost parentheses in (\ref{eq:super-modulus-P})
holds for some $\alpha>0$.
Set $b:=A_\infty(\omega)<\infty$.
Suppose $\liminf_{t\to\infty}\mathfrak{S}_t(\omega)\le F(\omega)-\delta$
for some $\delta>0$;
to complete the proof, it suffices to arrive at a contradiction.
By the statement in the outermost parentheses in (\ref{eq:super-modulus-P}),
the function $\ntt(\omega)|_{[0,b)}$
can be continued to the closed interval $[0,b]$
so that it becomes an element $g$ of $C[0,b]$.
Let $\Gamma(g)$ be the set of all extensions of $g$
that are elements of $\Omega$.
By the time-superinvariance of $F$,
all $\psi\in\Gamma(g)$ satisfy $F(\psi)\ge F(\omega)$.
Since $\liminf_{t\to b-}\mathfrak{S}_t(\psi)\le F(\omega)-\delta$
(remember that $\mathfrak{S}$ is time-invariant)
and the function $t\mapsto\mathfrak{S}_t$ is lower semicontinuous
(see (\ref{eq:positive-capital})),
$\mathfrak{S}_b(\psi)\le F(\omega)-\delta\le F(\psi)-\delta$,
for each $\psi\in\Gamma(g)$.
Continue $g$, which is now fixed, by measure-theoretic Brownian motion
starting from $g(b)$,
so that the extension is an element of $\Omega'$ with probability one.
Let us represent $\mathfrak{S}$ in the form (\ref{eq:positive-capital})
and use the argument in the proof of Lemma~\ref{lem:main-ge}.
We can see that $\mathfrak{S}_t(\xi)$, $t\ge b$,
where $\xi$ is $g$ extended by the trajectory of Brownian motion
starting from $g(b)$,
is a positive measure-theoretic supermartingale
with the time interval $[b,\infty)$.
Now we have the following analogue of (\ref{eq:maximal}):
\begin{equation*}
  \int_{\Gamma(g)}
  \liminf_{t\to\infty}
  \mathfrak{S}_t
  \dd P
  \le
  \liminf_{t\to\infty}
  \int_{\Gamma(g)}
  \mathfrak{S}_t
  \dd P
  \le
  \int_{\Gamma(g)}
  \mathfrak{S}_b
  \dd P
  \le
  \int_{\Gamma(g)}
  F
  \dd P
  -
  \delta,
\end{equation*}
$P$ referring to the underlying probability measure of the Brownian motion
(concentrated on $\Gamma(g)$).
However,
$\int_{\Gamma(g)}\liminf_{t\to\infty}\mathfrak{S}_t\dd P<\int_{\Gamma(g)}F\dd P$
contradicts the choice of $\mathfrak{S}$:
cf.\ (\ref{eq:infinite-A}) and Lemma~\ref{lem:A-BM}.

\section{Other connections with literature}
\label{sec:literature}

This section discusses several areas of stochastics
(in Subsection \ref{subsec:stochastic-integral})
and mathematical finance
(in Subsections \ref{subsec:FTAP} and \ref{subsec:model-free})
which are especially closely connected
with this paper's approach.

\subsection{Stochastic integration}
\label{subsec:stochastic-integral}

The natural financial interpretation of the stochastic integral is that
$\int_0^t \pi_s \dd X_s$
is the trader's profit at time $t$
from holding $\pi_s$ units of a financial security
with price path $X$ at time $s$
(see, e.g., \cite{Shiryaev:1999}, Remark III.5a.2).
It is widely believed that $\int_0^t \pi_s \dd X_s$
cannot in general be defined pathwise;
since our picture does not involve a probability measure on $\Omega$,
we restricted ourselves to countable combinations
(see (\ref{eq:positive-capital})) of integrals
of simple integrands (see (\ref{eq:simple-capital})).
This definition served our purposes well,
but in this subsection we will discuss other possible definitions,
always assuming that $X_s$ is a continuous function of $s$.

The pathwise definition of $\int_0^t \pi_s \dd X_s$
is straightforward when the total variation
(i.e., strong 1-variation in the terminology of Subsection~\ref{subsec:vi})
of $X_s$ over $[0,t]$ is finite;
it can be defined as, e.g., the Lebesgue--Stiltjes integral.
It has been known for a long time that the Riemann--Stiltjes definition
also works in the case
$1/\vi(\pi)+1/\vi(X)>1$
(Youngs' theory;
see, e.g., \cite{Dudley/Norvaisa:2011}, Section~2.2).
Unfortunately, in the most interesting case $\vi(\pi)=\vi(X)=2$
this condition is not satisfied.

Another pathwise definition of stochastic integral
is due to F\"ollmer \cite{Follmer:1981}.
F\"ollmer considers a sequence of partitions of the interval $[0,\infty)$
and assumes that the quadratic variation of $X$ exists,
in a suitable sense, along this sequence.
Our definition of quadratic variation
given in Section~\ref{sec:result-constructive}
resembles F\"ollmer's definition;
in particular,
our Theorem~\ref{thm:main-constructive}(a) implies
that F\"ollmer's quadratic variation exists for typical $\omega$
along the sequence of partitions $T^n$
(as defined at the beginning of Section~\ref{sec:result-constructive}).
In the statement of his theorem (\cite{Follmer:1981}, p.~144),
F\"ollmer defines the pathwise integral $\int_0^t f(X_s)\dd X_s$
for a $C^1$ function $f$ assuming that the quadratic variation of $X$ exists
and proves It\^o's formula for his integral.
In particular,
F\"ollmer's pathwise integral $\int_0^t f(\omega(s))\dd\omega(s)$
along $T^n$ exists for typical $\omega$
and satisfies It\^o's formula.
There are two obstacles to using F\"ollmer's definition in this paper:
in order to prove the existence of the quadratic variation
we already need our simple notion of integration
(which defines the notion of ``typical'' in Theorem~\ref{thm:main-constructive}(a));
the class of integrals $\int_0^t f(\omega(s))\,\dd\omega(s)$
with $f\in C^1$ is too restrictive for our purposes,
and using it would complicate the proofs.

An interesting development of Youngs' theory
is Lyons's \cite{Lyons:1998} theory of rough paths.
In Lyons's theory,
we can deal directly only with the rough paths $X$ satisfying $\vi(X)<2$
(by means of Youngs' theory).
In order to treat rough paths satisfying $\vi(X)\in[n,n+1)$,
where $n=2,3,\ldots$,
we need to postulate the values of the iterated integrals
$X^i_{s,t}:=\int_{s<u_1<\cdots<u_i<t}\dd X_{u_1}\cdots\dd X_{u_i}$
for $i=2,\ldots,n$
(satisfying so-called Chen's consistency condition).
According to Corollary~\ref{cor:vi},
only the case $n=2$ is relevant for our idealized market,
and in this case Lyons's theory is much simpler than in general
(but to establish Corollary~\ref{cor:vi} we already used our simple integral).
Even in the case $n=2$ there are different natural choices
of $X^2_{s,t}$
(e.g., those leading to It\^o-type and to Stratonovich-type integrals);
and in the case $n>2$ the choice would inevitably become even more \emph{ad hoc}.

Another obstacle to using Lyons's theory in this paper
is that the smoothness restrictions that it imposes are too strong
for our purposes.
In principle,
we could use the integral $\int_0^t G\dd\omega$
to define the capital brought by a strategy $G$ for trading in $\omega$
by time $t$.
However, similarly to F\"ollmer's, Lyons's theory requires that $G$
should take a position of the form $f(\omega(t))$ at time $t$,
where $f$ is a differentiable function
whose derivative $f'$ is a Lipschitz function
(\cite{Davie:2007}, Theorems 3.2 and 3.6).
This restriction would again complicate the proofs.

\subsection{Fundamental Theorems of Asset Pricing}
\label{subsec:FTAP}

The First and Second Fundamental Theorems of Asset Pricing (FTAPs, for brevity)
are families of mathematical statements;
e.g., we have different statements for one-period, multi-period,
discrete-time, and continuous-time markets.
A very special case of the Second FTAP,
the one covering binomial models,
was already discussed briefly in Section~\ref{sec:introduction}.
In the informal comparisons of our results and the FTAPs in this subsection
we only consider the case of one security
whose price path $X_t$ is assumed to be continuous.
(In the background,
there is also an implicit security, such as cash or bond, serving as our num\'eraire.)

The First FTAP says that a stochastic model for the security price path $X_t$
admits no arbitrage (or satisfies a suitable modification of this condition,
such as no free lunch with vanishing risk)
if and only if there is an equivalent martingale measure
(or a suitable modification thereof, such as an equivalent sigma-martingale measure).
The Second FTAP says that the market is complete
if and only if there is only one equivalent martingale measure
(as, e.g., in the case of the classical Black--Scholes model).
The completeness of the market means that each contingent claim has a unique fair price
defined in terms of hedging.

Theorems~\ref{thm:main} and~\ref{thm:supermain} are connected
(admittedly, somewhat loosely)
with the Second FTAP,
namely its part saying that each contingent claim has a unique fair price
provided there is a unique equivalent martingale measure.
For example, Theorem~\ref{thm:main} and Corollary~\ref{cor:main}
essentially say that each contingent claim of the form $\III_E$,
where $E\in\KKK$ and $\omega(0)=c$ for all $\omega\in E$,
has a fair price and its fair price is equal
to the Wiener measure $\Wiener_c(E)$ of $E$.
The scarcity of contingent claims that we can show to have a fair price is not surprising:
it is intuitively clear that our market is heavily incomplete.
According to Remark~\ref{rem:main},
we can replace the Wiener measure by many other measures.
The proofs of both the Second FTAP and our Theorems~\ref{thm:main} and~\ref{thm:supermain}
construct fair prices of contingent claims
using hedging arguments.
Extending this paper's results to a wider class of contingent claims
is an interesting direction of further research.

Theorems~\ref{thm:main} and~\ref{thm:supermain} are much more closely connected
with a generalized version of the Second FTAP
(see \cite{Follmer/Schied:2011}, Theorem~5.32, for a discrete-time version)
which says, in the first approximation,
that the range of arbitrage-free prices of a contingent claim
coincides with the range of the expectations of its payoff function
w.r.\ to the equivalent martingale measures.
We can even say (completely disregarding mathematical rigour for a moment)
that Theorem~\ref{thm:supermain} is a special case of the generalized Second FTAP:
by the Dubins--Schwarz result,
$\omega$ is a time-changed Brownian motion
under the martingale measures,
and so the $\KKK$-measurability of $F$
implies that the unique fair price of the contingent claim
with the payoff function $F$
is $\int F d\Wiener_{\omega(0)}$.

The conditions of the First, Second, and generalized Second FTAP include
a given probability measure on the sample space
(our stochastic model of the market).
In the case of continuous time,
it is this postulated probability measure
that allows one to use It\^o's notion of stochastic integral
for defining basic financial notions
such as the resulting capital of a trading strategy.
No such condition is needed in the case of our results.



The notion of arbitrage is pivotal in mathematical finance;
in particular,
it enters both the First FTAP
and the generalized Second FTAP.
This paper's results and discussions were not couched in terms of arbitrage,
although there were two places where arbitrage-type notions
did enter the picture.

First, we used the notion of coherence in Section~\ref{sec:coherence}.
The most standard notion of arbitrage is that no trading strategy
can start from zero capital and end up with positive capital
that is strictly positive with a strictly positive probability.
Our condition of coherence is similar but much weaker;
and of course, it does not involve probabilities.
We show that this condition is satisfied automatically in our framework.

The second place where we need arbitrage-type notions is in the interpretation of results
such as Corollaries~\ref{cor:increase} and \ref{cor:vi}--\ref{cor:Taylor}.
For example, Corollary~\ref{cor:vi} implies that
$\vi^{[0,1]}(\omega)\in\{0,2\}$ for typical $\omega$.
Remembering our definitions,
this means that either $\vi^{[0,1]}(\omega)\in\{0,2\}$
or a predefined trading strategy makes infinite capital (at time 1)
starting from one monetary unit
and never risking going into debt.
If we do not believe that making infinite capital risking only one monetary unit
is possible for a predefined trading strategy
(i.e., that the market is ``efficient'', in a very weak sense),
we should expect $\vi^{[0,1]}(\omega)\in\{0,2\}$.
This looks like an arbitrage-type argument,
but there are two important differences:
\begin{itemize}
\item
  Our condition of market efficiency is only needed for the interpretation of our results;
  their mathematical statements do not depend on it.
  The standard no-arbitrage conditions are used directly in mathematical theorems
  (such as the First FTAP and the generalized Second FTAP).
\item
  The usual no-arbitrage conditions
  are conditions on the currently observed prices
  or our stochastic model of the market (or both).
  On the contrary, our condition of market efficiency
  describes what we expect to happen, or not to happen, on the actual price path.
\end{itemize}

It should be noted that our condition of market efficiency
(a predefined trading strategy is not expected
to make infinite capital risking only one monetary unit)
is much closer to Delbaen and Schachermayer's \cite{Delbaen/Schachermayer:1994} version
of the no-arbitrage condition,
which is known as NFLVR (no free lunch with vanishing risk),
than to the classical no-arbitrage condition.
The classical no-arbitrage condition only considers trading strategies
that start from 0 and never go into debt,
whereas the NFLVR condition allows trading strategies that start from 0
and are permitted to go into slight debt.
Our condition of market efficiency allows risking one monetary unit,
but this can be rescaled so that the trading strategies considered
start from zero and are only allowed to go into debt limited
by an arbitrarily small $\epsilon>0$.

%
%

\begin{remark}
  Mathematical statements of the First FTAP sometimes involve the condition
  that $X_t$ should be a semimartingale:
  see, e.g., Delbaen and Schachermayer's version
  in \cite{Delbaen/Schachermayer:1994}, Theorem~1.1.
  However, this condition is not a big restriction:
  in the same paper, Delbaen and Schachermayer show that the NFLVR condition
  already implies that $X_t$ is a semimartingale
  (under some additional conditions, such as $X_t$ being locally bounded;
  see \cite{Delbaen/Schachermayer:1994}, Theorem~7.2).
  A direct proof of the last result,
  using financial arguments and not depending on the Bichteler--Dellacherie theorem,
  is given in the recent paper \cite{Beiglbock/etal:2011}.
\end{remark}

We could have used the notion of arbitrage
to restate part of Theorem~\ref{thm:supermain}:
if the contingent claim
with a bounded and $\KKK$-measurable payoff function $F:\Omega\to[0,\infty)$
is worth strictly more than $\int F d\Wiener_{\omega(0)}$ at time $0$,
we can turn capital 0 at time 0 into capital 1 at time $\infty$.
Indeed, we can short such a contingent claim and divide the proceeds
$\int F d\Wiener_{\omega(0)}+\epsilon$, where $\epsilon>0$,
into two parts:
investing $\int F d\Wiener_{\omega(0)}+\epsilon/2$
into a trading strategy bringing capital $F(\omega)$ at time $\infty$
allows us to meet our obligation;
we keep the remaining $\epsilon/2$
(and we can scale up our portfolio to replace $\epsilon/2$ by 1).
We did not introduce the corresponding notion of arbitrage formally
since this restatement does not seem to add much to the theorem.

\subsection{Model uncertainty and robust results}
\label{subsec:model-free}

In this subsection we will discuss some known approaches
to mathematical finance that do not assume from the outset
a given probability model.

One natural relaxation of the standard framework
replaces the probability model with a family, more or less extensive,
of probability models
(there is a ``model uncertainty'').
Results proved under model uncertainty may be called robust.
We get some robustness for free already in the standard Black--Scholes framework:
option prices do not depend on the drift parameter $\mu$
in the probability model $\dd X_t/X_t=\mu\dd t + \sigma\dd W_t$,
$W_t$ being Brownian motion.
``Volatility uncertainty'', i.e.,
uncertainty about the value of $\sigma$, is much more serious.
A natural assumption,
sometimes called the ``uncertain volatility model'',
is that $\sigma$ can change dynamically
between known limits $\underline{\sigma}$ and $\overline{\sigma}$,
$\underline{\sigma}<\overline{\sigma}$.
Study of volatility uncertainty under this assumption
was originated by Avellaneda et al.\ \cite{Avellaneda/etal:1995}
and Lyons \cite{Lyons:1995}
and has been the object of intensive study recently;
whereas older paper concentrated on robust pricing of contingent claims
whose payoff depends on the underlying security's value
at one maturity date,
recent work treats the much more difficult case
of general path-dependent contingent claims.
This research has given rise to two important developments:
Denis and Martini's \cite{Denis/Martini:2006}
``almost pathwise'' theory of stochastic calculus
and Peng's \cite{Peng:2007,Peng:2010} $G$-stochastic calculus
(in our current context,
$G$ refers to the function
$G(y):=\sup_{\sigma\in[\underline{\sigma},\overline{\sigma}]}\sigma^2 y$).




Definitions similar to our (\ref{eq:upper-probability}) and (\ref{eq:upper-expectation})
are standard in the literature on model uncertainty:
see, e.g., Mykland \cite{Mykland:2000}, (3.3),
Denis and Martini \cite{Denis/Martini:2006}
(the definition of $\Lambda(f)$ on p.~834),
or Cassese \cite{Cassese:2008}, (4.4).
Different terms corresponding to our ``upper price'' have been used,
such as ``conservative ask price'' (Mykland)
and ``cheapest riskless superreplication price'' (Denis and Martini);
we will continue using ``upper price'' as a generic notion.
A major difficulty for such definitions
lies in defining the class of capital processes;
it is here that pre-specifying a family of probability models
proves to be particularly useful.


Finally,
we will discuss approaches that are
completely model-free.
Bick and Willinger \cite{Bick/Willinger:1994}
use F\"ollmer's construction of stochastic integral
discussed in Subsection \ref{subsec:stochastic-integral}
to define capital processes of trading strategies.
Even though their framework is not stochastic,
the conditions that they impose on the price paths
in order for dynamic hedging to be successful
are not so different from the standard conditions.
The assumption used in their Proposition~1 is,
in their notation,
$[Y,Y]_t=Y^0+\sigma^2t$,
where $S(t)=\exp(Y(t))$ is the price path
and $[Y,Y]_t$ is the pathwise quadratic variation of its logarithm;
this is similar to the Black--Scholes model.
They also consider (in Proposition~3) a more general case
$d[S,S]_t=\beta^2(S(t),t)$,
but $\beta$ has to be a continuous function that is known in advance.

Section~4 of Dawid et al.'s \cite{\Dawid} can be recast as a study of the upper price
of the American option paying $f(X_t^*)$,
where $f$ is a fixed positive and increasing function,
$t$ is the exercise time (chosen by the option's owner),
$X_t^*:=\max_{s\le t}X_s$
(time is discrete in \cite{\Dawid}),
and $X_s\ge0$ is the price of the underlying security at time $s$.
Corollary~2 in~\cite{\Dawid} implies
that the upper price of this option
is $X_0\int_{X_0}^{\infty}\frac{f(x)}{x^2}\dd x$.
This is compatible with Theorem~\ref{thm:supermain}
since $X_0/x^2$, $x\in[X_0,\infty)$,
is the density of the maximum of Brownian motion
started at $X_0$ and stopped when it hits 0
(cf.\ the first statement of Theorem 2.49
in \cite{Morters/Peres:2010}).

Let us assume, for simplicity, that $X_0=1$ (as in \cite{Hobson:1998}).
The simplest American option with payoff $f(X_t^*)$
is the one corresponding to the identity function $f(x)=x$;
it is a kind of a perpetual lookback option
(as discussed in, e.g., \cite{Duffie/Harrison:1993}, Section~5).
The upper price of this option is, of course, infinite:
$\int_1^{\infty}(1/x)\,\dd x=\infty$.
To get a finite price,
we can fix a finite maturity date $T$
and consider a European option with payoff $X_T^*:=\sup_{t\le T}X_t$
(we no longer assume that time is discrete).
To find a non-trivial upper price of this \emph{European lookback option},
Hobson \cite{Hobson:1998} considers trading strategies
that trade not only in the underlying security $X$
but also in call options on $X$ with maturity date $T$
and all possible strike prices
(making some regularity assumptions about the call prices);
he also finds the upper prices
for some modifications of European lookback options.
In order to avoid the use of the stochastic integral,
the dynamic part of the trading strategies that he considers is very simple;
there is only finite trading activity in each security.
Hobson's paper has been developed in various directions:
see, e.g., the recent review \cite{Hobson:2011} and references therein.
One important issue that arises
when we specify the prices of vanilla options at the outset
is whether these prices lead to arbitrage opportunities;
it has been investigated, for various notions of arbitrage,
in \cite{Davis/Hobson:2007} and \cite{Cox/Obloj:2011}.

An advantage of this paper's main results
is that the prices they provide are ``almost'' two-sided
(serve as both ask and bid prices): cf.\ Corollary \ref{cor:main}.
Their disadvantage
is that they allow us to price such a narrow class of contingent claims:
their payoff functions are required to be $\KKK$-measurable.
In principle, Hobson's idea of using vanilla options
for pricing exotic options
may lead to interesting developments of this paper's approach.
One could consider a whole spectrum of trading frameworks,
even in the case of one underlying security $X$.
One extreme is the framework of this paper
and, in the case of a discontinuous price path, \cite{\CTV}.
The security is not supported by any derivatives,
which leads to the paucity of contingent claims that can be priced.
The other extreme is where, alongside $X$,
we are allowed to trade in all European contingent claims
for all maturity dates.
Perhaps the most interesting research questions arise
in between the two extremes,
where only some European contingent claims
are available for use in hedging.

\section*{Appendix: Hoeffding's process}
\renewcommand{\thesection}{A}
\refstepcounter{section}
\addcontentsline{toc}{section}{Appendix: Hoeffding's process}

In this appendix we will check
that Hoeffding's original proof of his inequality
(\cite{Hoeffding:1963}, Theorem~2)
remains valid in the game-theoretic framework.
This observation is fairly obvious,
but all details will be spelled out
for convenience of reference.
This appendix is concerned with the case of discrete time,
and it will be convenient to redefine some notions
(such as ``process'').


Perhaps the most useful product of Hoeffding's method
is a positive supermartingale starting from 1
and attaining large values
when the sum of bounded martingale differences is large.
Hoeffding's inequality can be obtained
by applying the maximal inequality to this supermartingale.
However, we do not need Hoeffding's inequality in this paper,
and instead of Hoeffding's positive supermartingale
we will have a positive ``supercapital process'',
to be defined below.

This is a version of the basic forecasting protocol from \cite{Shafer/Vovk:2001}:

\bigskip

\noindent
\textsc{Game of forecasting bounded variables}

\smallskip

\noindent
\textbf{Players:} Sceptic, Forecaster, Reality

\smallskip

\noindent
\textbf{Protocol:}

\parshape=6
\IndentI   \WidthI
\IndentI   \WidthI
\IndentII  \WidthII
\IndentII  \WidthII
\IndentII  \WidthII
\IndentII  \WidthII
\noindent
Sceptic announces $\K_0\in\bbbr$.\\
FOR $n=1,2,\dots$:\\
  Forecaster announces interval $[a_n,b_n]\subseteq\bbbr$
    and number $\mu_n\in(a_n,b_n)$.\\
  Sceptic announces $M_n\in\bbbr$.\\
  Reality announces $x_n\in[a_n,b_n]$.\\
  Sceptic announces $\K_n \le \K_{n-1} + M_n (x_n - \mu_n)$.

\bigskip

\noindent
On each round $n$ of the game
Forecaster outputs an interval $[a_n,b_n]$ which, in his opinion,
will cover the actual observation $x_n$ to be chosen by Reality,
and also outputs his expectation $\mu_n$ for $x_n$.
The forecasts are being tested by Sceptic,
who is allowed to gamble against them.
The expectation $\mu_n$ is interpreted as the price of a ticket
which pays $x_n$ after Reality's move becomes known;
Sceptic is allowed to buy any number $M_n$,
positive or negative (perhaps zero),
of such tickets.
When $x_n$ falls outside $[a_n,b_n]$,
Sceptic becomes infinitely rich;
without loss of generality
we include the requirement $x_n\in[a_n,b_n]$ in the protocol;
furthermore, we will always assume that $\mu_n\in(a_n,b_n)$.
Sceptic is allowed to choose his initial capital $\K_0$
and is allowed to throw away part of his money at the end of each round.

It is important that the game of forecasting bounded variables
is a perfect-information game:
each player can see the other players' moves
before making his or her (Forecaster and Sceptic are male and Reality is female) own move;
there is no randomness in the protocol.

A \emph{process} is a real-valued function
defined on all finite sequences
\begin{equation*}
  (a_1,b_1,\mu_1,x_1,\ldots,a_N,b_N,\mu_N,x_N),
  \quad
  N=0,1,\ldots,
\end{equation*}
of Forecaster's and Reality's moves in the game of forecasting bounded variables.
If we fix a strategy for Sceptic,
Sceptic's capital $\K_N$, $N=0,1,\ldots$,
become a function of Forecaster's and Reality's previous moves;
in other words,
Sceptic's capital becomes a process.
The processes that can be obtained this way
are called \emph{supercapital processes}.

The following theorem is essentially inequality (4.16) in \cite{Hoeffding:1963}.
\begin{theorem}\label{thm:super}
  For any $h\in\bbbr$,
  the process
  \begin{equation*}
    \prod_{n=1}^N
    \exp
    \left(
      h(x_n-\mu_n)
      -
      \frac{h^2}{8} (b_n-a_n)^2
    \right)
  \end{equation*}
  is a supercapital process.
\end{theorem}
\begin{proof}
  Assume, without loss of generality,
  that Forecaster is additionally required
  to always set $\mu_n:=0$.
  (Adding the same number to $a_n$, $b_n$, and $\mu_n$ on each round
  will not change anything for Sceptic.)
  Now we have $a_n<0<b_n$.

  It suffices to prove that on round $n$ Sceptic can turn a capital of $\K$
  into a capital of at least
  \begin{equation*}
    \K
    \exp
    \left(
      h x_n
      -
      \frac{h^2}{8} (b_n-a_n)^2
    \right);
  \end{equation*}
  in other words,
  that he can obtain a payoff of at least
  \begin{equation*}
    \exp
    \left(
      h x_n
      -
      \frac{h^2}{8} (b_n-a_n)^2
    \right)
    -
    1
  \end{equation*}
  using the available tickets
  (paying $x_n$ and costing $0$).
  This will follow from the inequality
  \begin{equation}\label{eq:positive}
    \exp
    \left(
      h x_n
      -
      \frac{h^2}{8} (b_n-a_n)^2
    \right)
    -
    1
    \le
    x_n
    \frac{e^{h b_n}-e^{h a_n}}{b_n-a_n}
    \exp
    \left(
      -
      \frac{h^2}{8} (b_n-a_n)^2
    \right),
  \end{equation}
  which can be rewritten as
  \begin{equation}\label{eq:goal}
    \exp
    \left(
      h x_n
    \right)
    \le
    \exp
    \left(
      \frac{h^2}{8} (b_n-a_n)^2
    \right)
    +
    x_n
    \frac{e^{h b_n}-e^{h a_n}}{b_n-a_n}.
  \end{equation}

  Our goal is to prove (\ref{eq:goal}).
  By the convexity of the function $\exp$,
  it suffices to prove
  \begin{equation}\label{eq:suffices}
    \frac{x_n-a_n}{b_n-a_n}
    e^{h b_n}
    +
    \frac{b_n-x_n}{b_n-a_n}
    e^{h a_n}
    \le
    \exp
    \left(
      \frac{h^2}{8} (b_n-a_n)^2
    \right)
    +
    x_n
    \frac{e^{h b_n}-e^{h a_n}}{b_n-a_n},
  \end{equation}
  i.e.,
  \begin{equation}\label{eq:goal-transformed}
    \frac
    {
      b_n e^{h a_n}
      -
      a_n e^{h b_n}
    }
    {b_n-a_n}
    \le
    \exp
    \left(
      \frac{h^2}{8} (b_n-a_n)^2
    \right),
  \end{equation}
  i.e.,
  \begin{equation}\label{eq:simpler}
    \ln
    \left(
      b_n e^{h a_n}
      -
      a_n e^{h b_n}
    \right)
    \le
    \frac{h^2}{8} (b_n-a_n)^2
    +
    \ln(b_n-a_n).
  \end{equation}
  (The logarithm on the left-hand side of (\ref{eq:simpler}) is well defined
  since the numerator of the left-hand side of (\ref{eq:goal-transformed})
  is strictly positive,
  which follows from the left-hand side of (\ref{eq:goal-transformed})
  being the value at $x_n=0$ of the left-hand side of (\ref{eq:suffices}),
  linear in $x_n$ and strictly positive for both $x_n=a_n$ and $x_n=b_n$.)
  The derivative of the left-hand side of (\ref{eq:simpler}) in $h$ is
  \begin{equation*}
    \frac
    {
      a_n b_n e^{h a_n}
      -
      a_n b_n e^{h b_n}
    }
    {
      b_n e^{h a_n}
      -
      a_n e^{h b_n}
    }
  \end{equation*}
  and the second derivative, after cancellations and regrouping, is
  \begin{equation*}
    (b_n-a_n)^2
    \frac
    {
      \left(
        b_n e^{h a_n}
      \right)
      \left(
        -a_n e^{h b_n}
      \right)
    }
    {
      \left(
        b_n e^{h a_n}
        -
        a_n e^{h b_n}
      \right)^2
    }.
  \end{equation*}
  The last ratio is of the form $u(1-u)$ where $0<u<1$.
  Hence it does not exceed $1/4$,
  and the second derivative itself does not exceed $(b_n-a_n)^2/4$.
  Inequality (\ref{eq:simpler}) now follows from the second-order Taylor expansion
  of the left-hand side around $h=0$.
\end{proof}

\subsection*{Acknowledgments}

The final statement of Theorem~\ref{thm:main}
is due to Peter McCullagh's insight
and Tamas Szabados's penetrating questions.
The game-theoretic version of Hoeffding's inequality
is inspired by a question asked by Yoav Freund.
Rimas Norvai\v{s}a's useful comments and
explanations are gratefully appreciated.
This paper very much benefitted from the feedback from several anonymous reviewers,
the Associate Editor, and Professor Martin Schweizer.
Their contributions ranged from very specific,
such as noticing a mistake in the statement of Corollary \ref{cor:main}
in an early version,
to general comments that have led, e.g.,
to new applications of Theorem~\ref{thm:main},
to Theorems~\ref{thm:supermain} and \ref{thm:supermain-constructive},
to the inclusion of Section~\ref{sec:literature},
and to improved presentation.
At the latest stages of the work on the journal version of this paper
I benefitted from comments by Wouter Koolen, Roman Chychyla,
John Shawe-Taylor, Jan Ob\l\'oj, and Johannes Ruf.
After finishing the journal version I have had productive discussions
with Nicolas Perkowski, David Pr\"omel, Martin Huesmann,
Alexander M. G. Cox, Pietro Siorpaes, and Beatrice Acciaio.

This work was supported in part by EPSRC (grant EP/F002998/1).

\end{document}